\documentclass[aihp]{imsart}

\RequirePackage{amsthm,amsmath,amsfonts,amssymb}
\RequirePackage[numbers]{natbib}
\RequirePackage[colorlinks,citecolor=blue,urlcolor=blue]{hyperref}%% uncomment this for coloring bibliography citations and linked URLs
\RequirePackage{graphicx}%% uncomment this for including figures

\usepackage{dsfont}
\usepackage{bm}
\usepackage{empheq}
\usepackage[bbgreekl]{mathbbol}
\DeclareSymbolFontAlphabet{\mathbbm}{bbold}
\DeclareSymbolFontAlphabet{\mathbb}{AMSb}
\usepackage{scalerel}[2016/12/29]
\usepackage{parskip}
\usepackage[scr=boondox]{mathalfa}
\usepackage[usenames, dvipsnames]{xcolor}
\usepackage{pgf,tikz}
\usetikzlibrary{arrows,shapes,positioning}
\usetikzlibrary{arrows.meta}
\usetikzlibrary{decorations.pathmorphing}
\usetikzlibrary{patterns,decorations.pathreplacing}

\startlocaldefs
\numberwithin{equation}{section}
\theoremstyle{plain}
\newtheorem{theorem}{Theorem}[section]
\newtheorem{corollary}[theorem]{Corollary}
\newtheorem{lemma}[theorem]{Lemma}
\newtheorem{proposition}[theorem]{Proposition}
\theoremstyle{definition}
\newtheorem{definition}[theorem]{Definition}
\newtheorem{remark}[theorem]{Remark}
\newtheorem*{assumption}{Assumption}
\newcommand{\R}{\mathds{R}}
\newcommand{\Z}{\mathds{Z}}
\newcommand{\N}{\mathds{N}}

\renewcommand{\P}[1]{\mathds{P}\left(#1\right)}
\newcommand{\ind}{\mathds{1}}
\newcommand{\E}[1]{\mathds{E}\left[#1\right]}
\newcommand{\itg}{\displaystyle\int}

\newcommand{\cond}{\ \middle| \ }

\newcommand{\Pois}{\mathscr{P}\!\mathscr{o}\!\mathscr{i}\!\mathscr{s}}

\newcommand{\model}{\mathfrak{M}}
\newcommand{\Model}{\mathfrak{S}}
\newcommand{\lm}{\mathfrak{L}}
\renewcommand{\L}{\bm{\mathscr{L}}}
\newcommand{\ones}{\bm{1}}
\newcommand{\F}{\mathcal{F}}
\newcommand{\G}{\mathcal{G}}

\newcommand{\D}{\mathcal{D}}
\newcommand{\M}{\bm{M}}
\newcommand{\Pg}{\bm{P}}
\newcommand{\Q}{\mathcal{Q}}
\renewcommand{\mp}[2]{\mathscr{D}^{#1}_{#2}}
\newcommand{\T}{\mathcal{T}}

\newcommand{\I}{\mathcal{I}}
\newcommand{\U}{\mathcal{U}}

\newcommand{\Mass}{\scaleobj{1.2}{\mathscr{M}}}
\newcommand{\Massi}{\scaleobj{1.2}{\mathscr{M}}^{\ic}}
\newcommand{\Massiexp}{\scaleobj{1.2}{\mathscr{M}}^{\ic,\exp}}
\newcommand{\f}{\scaleobj{1.2}{\mathscr{F}}}

\newcommand{\Lrv}{\mathrm{\mathbb{L}}}
\newcommand{\Lf}{\mathrm{L}}
\newcommand{\Ld}[1]{\mathrm{L}^{1,#1}_{\infty}}
\newcommand{\VLd}[1]{\Vert_{1,#1}^{\infty}}

\newcommand{\Hmat}{\mathrm{H}}
\newcommand{\Hmod}[1]{\Hmat_{\scaleobj{1}{#1}}}

\newcommand{\Equ}[2]{\bm{{\HAR}}_{#1}\bm{(} #2 \bm{)}}

\newcommand{\noexpl}{\text{\textbf{\upshape{NoExpl}}}}

\newcommand{\spec}{\text{\textbf{\upshape{Spec1}}}}
\newcommand{\specbis}{\text{\textbf{\upshape{Spec1.1}}}}
\newcommand{\subg}{\text{\textbf{\upshape{SubG}}}}
\newcommand{\fime}{\text{\textbf{\upshape{FiMe}}}}
\newcommand{\exptail}{\text{\textbf{\upshape{ExpTail}}}}
\newcommand{\specter}{\text{\textbf{\upshape{Spec2}}}}

\newcommand{\Cnoexpl}{\hyperlink{noexpl}{$\bm{(}\noexpl\bm{)}$}}

\newcommand{\Cspec}{\hyperlink{spec}{$\bm{(}\spec\bm{)}$}}
\newcommand{\Cspecbis}{\hyperlink{specbis}{$\bm{(}\specbis\bm{)}$}}
\newcommand{\Csubg}{\hyperlink{subg}{$\bm{(}\subg\bm{)}$}}
\newcommand{\Cfime}{\hyperlink{fime}{$\bm{(}\fime\bm{)}$}}
\newcommand{\Cexptail}{\hyperlink{exptail}{$\bm{(}\exptail\bm{)}$}}
\newcommand{\Cspecter}{\hyperlink{specter}{$\bm{(}\specter\bm{)}$}}

\newcommand{\coef}[3]{\genfrac{\langle}{\rangle}{0pt}{}{#1}{#2}_{\!\! #3}}

\DeclarePairedDelimiter\floor{\lfloor}{\rfloor}
\DeclareMathOperator{\lip}{Lip}

\DeclareMathOperator{\card}{Card}

\DeclareMathOperator{\id}{id}
\DeclareMathOperator{\spr}{SpR}
\DeclareMathOperator{\Id}{Id}
\DeclareMathOperator{\temp}{Temp}
\DeclareMathOperator{\Hd}{Y}
\DeclareMathOperator{\J}{J}
\DeclareMathOperator{\tJ}{\tilde{J}}

\DeclareMathOperator{\HAR}{HAR}
\DeclareMathOperator{\tp}{tp}
\DeclareMathOperator{\bd}{bd}

\DeclareMathOperator{\ic}{ic}

\endlocaldefs

\begin{document}

\begin{frontmatter}

\title{Hawkes autoregressive processes: a new model for multiscale and heterogeneous processes}
%\title{A sample article title with some additional note\thanksref{T1}}
\runtitle{Hawkes AutoRegressive Processes}
%\thankstext{T1}{A sample of additional note to the title.}

\begin{aug}
%%%%%%%%%%%%%%%%%%%%%%%%%%%%%%%%%%%%%%%%%%%%%%%
%% Additional information such as            %%
%% identifying the corresponding author must %%
%% be included in in the Acknowledgments     %%
%% section if necessary.                     %%
\author[A]{\inits{T.}\fnms{Théo}~\snm{Leblanc}\ead[label=e1]{leblanc@ceremade.dauphine.fr}\orcid{0009-0000-2194-9512}}
%%%%%%%%%%%%%%%%%%%%%%%%%%%%%%%%%%%%%%%%%%%%%%
%% Addresses                                %%
%%%%%%%%%%%%%%%%%%%%%%%%%%%%%%%%%%%%%%%%%%%%%%
\address[A]{CEREMADE, Université Paris Dauphine - PSL, Paris, France\printead[presep={,\ }]{e1}}
\end{aug}

\begin{abstract}
Both Hawkes processes and autoregressive processes rely on linear functionals of their past, while modeling different types of data. Since datasets arising from observations of the same phenomenon may be heterogeneous and sampled at different time scales, it is natural to study multiscale and heterogeneous processes, such as those obtained by combining Hawkes and autoregressive dynamics. In this paper, we introduce this new Hawkes autoregressive (HAR) model incorporating both continuous- and discrete-time dynamics, and establish several probabilistic results, including the existence of a stationary version, a cluster representation, as well as stability and ergodic properties.
\end{abstract}

\begin{abstract}[language=french]
Les processus de Hawkes et les processus autorégressifs reposent tous deux sur des fonctionnelles linéaires de leur passé, tout en modélisant des données de nature différente. Les données issues de l'observation d'un même phénomène peuvent être hétérogènes et échantillonnées à différentes échelles de temps ; il est donc naturel d'étudier des processus multi-échelles et hétérogènes, tels que ceux obtenus en combinant des dynamiques de Hawkes et autorégressives. Dans cet article, nous introduisons un nouveau modèle Hawkes autorégressif (HAR), combinant des dynamiques en temps continu et en temps discret, et établissons plusieurs résultats probabilistes, notamment l'existence d'une version stationnaire, une représentation en clusters, ainsi que des propriétés de stabilité et d'ergodicité.
\end{abstract}

\begin{keyword}[class=MSC]
\kwd[Primary ]{60G55}
\kwd{62M10}
\kwd[; secondary ]{60G10}
\kwd{60F15}
\end{keyword}

\begin{keyword}
\kwd{Hawkes processes}
\kwd{Autoregressive processes}
\kwd{Cluster representation}
\kwd{Stability}
\kwd{Ergodic theorem}
\end{keyword}

\end{frontmatter}

%%%%%%%%%%%%%%%%%%%%%%%%%%%%%%%%%%%%%%%%%%%%%%
%%%% Main text entry area:

\section{Introduction}

\subsection{Motivations}

Hawkes processes, introduced by Hawkes in 1971 \cite{hawkes_spectra_1971}, are a class of self-interacting point processes designed to model the occurrence of events in continuous time. Their defining feature is that the conditional intensity depends on the history of the process, allowing past events to influence future ones. In the multivariate setting, the process consists of several interacting components, each generating its own sequence of events. An event in one component may increase or decrease the likelihood of future events in other components, giving rise to excitation, inhibition, or both. This interaction structure makes Hawkes processes particularly well suited to capturing complex temporal dependencies in systems of mutually interacting entities. Thanks to this flexibility, Hawkes processes have been successfully applied in a wide range of domains, including seismology \cite{ogata_statistical_1988}, finance \cite{embrechts_multivariate_2011}, epidemiology \cite{meyer_space-time_2011}, and neuroscience \cite{lambert_reconstructing_2018,bonnet_inference_2023,galves_infinite_2013,raad_stability_2020}. In neuroscience, they provide a natural framework for modeling neuronal networks: each component represents a neuron, and events correspond to its spike times. The dependence on past activity captures key biological mechanisms such as synaptic integration and interaction between neurons. Over the past decades, Hawkes processes have attracted considerable attention, both for their theoretical properties and their practical applications (see, e.g., \cite{bremaud_stability_1996,bacry_limit_2013,delattre_hawkes_2015}).\\

Autoregressive (AR) processes, in contrast to Hawkes processes, evolve in discrete time and take continuous values. However, they share a key feature with Hawkes processes: their dependence on past observations. This analogy explains why Hawkes processes are sometimes described as autoregressive point processes. More precisely, the value of an AR process at time $t \in \Z$ is given by a function of its previous values $W_s$, for $s < t$, together with an additive noise term \cite{box_box_2013,shumway_time_2025}. Depending on how this dependence is specified, AR processes can exhibit self-excitation, self-inhibition, or a combination of both, similarly to Hawkes processes \cite{jones_nonlinear_1997}. AR processes also admit a multivariate extension, commonly referred to as vector autoregressive (VAR or MVAR) models, in which several time series interact and evolve jointly, each depending on its own past as well as that of the others \cite{lutkepohl_new_2005}. Because of their simplicity and flexibility, autoregressive models are widely used across many fields, particularly in economics and finance \cite{engle_autoregressive_1982,bollerslev_generalized_1986}. More recently, they have also found applications in neuroscience, where they are used to model EEG signals and brain rhythms \cite{lawhern_detecting_2013,zhang_classification_2017}.\\

Hawkes processes model the continuous arrival of events over time, whereas autoregressive (AR) processes describe discrete-time series with real-valued observations. These two classes of models are therefore complementary, as they share a similar dependence structure while operating in fundamentally different settings. In many applications, experiments produce heterogeneous data, combining event times and regularly sampled signals. This naturally motivates the development of models capable of handling such heterogeneous data, by combining processes of different nature, for instance a continuous-time point process with a discrete-time one. \\
Hawkes processes have already been combined with other types of stochastic processes. For instance, in \cite{dion_exponential_2020,amorino_nonparametric_2025}, the authors couple a diffusion process with a Hawkes process, yielding a heterogeneous model in which the jumps of the Hawkes process are embedded into a stochastic diffusion. In these models, however, the interaction is one-sided, as the diffusion does not influence the Hawkes process. Such Hawkes--diffusion models have found several applications in finance \cite{hawkes_hawkes_2022}.\\ 
Another important feature of real-world data is the presence of multiple time scales. Different phenomena may be observed simultaneously but recorded at different frequencies, leading to intrinsically multiscale datasets. This calls for models that can incorporate such multiscale structures. Multivariate autoregressive processes provide a natural way to encode multiscale features. In this article, we consider a multivariate AR process $W = (W^1,\cdots,W^P)$ in which each component $W^p = (W^p_t)_{t\in \tau_p \Z}$ evolves on its own temporal grid $\tau_p \Z$, with $\tau_p$ representing its sampling period. This formulation allows different components to operate at different time scales while remaining coupled through their dependence structure. For instance, multiscale processes naturally arise from wavelet representations, which decompose signals across temporal scales and motivate autoregressive constructions adapted to such structures \cite{basseville_multiscale_1992, daoudi_multiscale_1999}. In a more general perspective, multiscale behaviour can also be introduced in both Hawkes and autoregressive processes by allowing interactions to act over multiple memory horizons. This is typically achieved by parametrising the dependence structure with several time scales, thereby capturing both short and long term effects; see \cite{bacry_sparse_2020} for Hawkes processes and \cite{baranowski_multiscale_2027} for autoregressive processes.\\
The main objective of this paper is to introduce and analyse a new model combining a Hawkes process and an autoregressive process within a multiscale framework.\\
A multiscale and heterogeneous model was recently introduced in \cite{spaziani_heterogeneous_2025} to describe the interactions between neuronal activity and brain rhythms. This model combines discrete-time Hawkes processes with multiscale autoregressive processes, providing a detailed description of brain dynamics and functional connectivity. While this discrete-time framework simplifies the mathematical analysis, it does not capture the full complexity of continuous-time event dynamics.\\
In this article, we extend this line of work by introducing a model that couples a continuous-time Hawkes process with a multiscale autoregressive process through reciprocal interactions. We refer to this model as the Hawkes Autoregressive (HAR) process. Moving to continuous time raises significant theoretical challenges, such as the possible accumulation of events in finite time, which are absent in the discrete setting. It also leads to sharper and more informative conditions ensuring the existence of the process. The next section details how a Hawkes process and an autoregressive process can be naturally combined to define a HAR process.

\subsection{Heuristic for HAR processes}

We provide a heuristic to combine Hawkes and Autoregressive processes in order to obtain HAR processes. For this heuristic, we focus on the one-dimensional case.

\textbf{Hawkes processes:} Given a filtration $(\F_t)_t$, a point process $N$ admits $\lambda_t$ as stochastic predictable intensity if
\[ \lambda_t dt = \E{dN_t \cond \mathcal{F}_{t-}}.\]
The intensity at time $t$ can be interpreted as the infinitesimal probability of having a point between $t$ and $t+dt$ given the past up to time $t$. We say that $N$ is a Hawkes process (for the filtration $\F_t$) if the intensity depends on the past as follows, $\lambda_t = \phi\Big(\itg^{t-}_{-\infty} h(t-s)dN_s\Big)$, where $\phi$ and $h$ are two functions. Unlike $h$, we assume that $\phi$ is non-negative. The function $h$ models self-interaction of the point process $N$ on itself. If $h\geq 0$ (resp. $h\leq 0$) then $N$ is self exciting (resp. inhibiting). Function $\phi$ allows for the introduction of nonlinear dependency, \cite{bremaud_stability_1996,hawkes_spectra_1971}. Therefore, for Hawkes processes we have,
\begin{equation}\label{eq1.1}
\E{dN_t \cond \mathcal{F}_{t-}} = \phi\left(\itg^{t-}_{-\infty} h(t-s)dN_s\right) dt.
\end{equation}
\textbf{Autoregressive processes:} For AR processes, the current value can be expressed as a function of past observations, up to some noise. AR processes can be defined by the following equation
\begin{equation}\label{eq1.2} W_k =  \sum_{r\geq 1} a_r W_{k-r} + \varepsilon_k,
\end{equation}
where the $a_r$'s are fixed coefficients for positive integers $r$ and $\varepsilon_k$ is a random noise independent of the $W_{k'}$ with $k'<k$ \cite{shumway_time_2025,box_box_2013}. The function $a : r\in\N^* \mapsto a_r =:a(r)$, where $\N^*$ stands for the positive integers, models past dependency. One can generalise AR processes by allowing more complex past dependency leading to $W_k = \Psi(W_{k-1},W_{k-2},\cdots) + \varepsilon_k$, $k\in\Z$, 
with $\Psi$ a function. To obtain an expression analogous to \eqref{eq1.1} we stick to the following level of generality,
\begin{equation}\label{eq1.3} 
W_k = \psi\left( \sum_{k'<k} \J(k-k',W_{k'})\right) + \varepsilon_k,
\end{equation}
where $\psi, \J$ are given functions. Taking $\psi = \id$ and $\J(s,w) = a(s)w$ leads to the classical AR model described in \eqref{eq1.2}. Function $\psi$ allows the introduction of some nonlinearity and is the analogue of $\phi$ in \eqref{eq1.1}.

By denoting $\mp{}{}$ the counting measure on the set of the integers, \eqref{eq1.3} becomes $W_k =  \psi\Big( \itg_{-\infty}^{k-} \J(k-k',W_{k'}) d\mp{}{k'}\Big) + \varepsilon_k$. In particular, if $(\F_k)_k$ is the natural history of $\varepsilon$, we obtain the following equality:
\begin{equation}\label{eq1.5} \E{W_k \cond \mathcal{F}_{k-1}} = \psi\left(  \itg_{-\infty}^{k-} \J(k-k',W_{k'}) d\mp{}{k'}\right).\end{equation}

General Hawkes and AR processes have very similar structures, see \eqref{eq1.1} and \eqref{eq1.5}. This structural similarity naturally suggests combining them. For this purpose we add cross interaction terms in the intensity equation and in the AR equation, leading to the following system of equations.
\begin{equation}\label{eq1.6}
    \left\{
    \begin{aligned}
        \lambda_t & = \phi\left( \itg^{t-}_{-\infty} h_1(t-s)dN_s + \itg_{-\infty}^{t-} \J_2(t-k,W_{k}) d\mp{}{k}\right)\\
        W_k & = \psi\left( \itg^{k-}_{-\infty} h_{3}(k-s)dN_s + \itg_{-\infty}^{k-} \J_4(k-k',W_{k'}) d\mp{}{k'} \right) + \varepsilon_k,
    \end{aligned}
    \right.
\end{equation}
where $h_1,\J_2,h_3,\J_4$ are functions. Equations \eqref{eq1.6} provide the heuristic definition of HAR processes.

\subsection{Contributions and related works}

In this article, we establish several fundamental probabilistic properties of Hawkes Autoregressive (HAR) processes. Our contributions focus on three main aspects: well-posedness of the model (in particular non-explosion and existence of a stationary regime), structural representation through clusters, and long-time behaviour including convergence to equilibrium and ergodic properties. These results provide a theoretical foundation for further analysis of HAR processes.

\textbf{Well-posedness and stationary regime.}
A first issue is to ensure that the process is well-defined, in particular that no explosion occurs and that a stationary version exists. For Hawkes processes, Brémaud and Massoulié \cite{bremaud_stability_1996} proved existence and uniqueness of a stationary process under a spectral radius condition on the matrix of $\Lf^1$ norms of the interaction functions. For AR$(\infty)$ processes, stationarity and causality typically follow from summability conditions on the coefficients, which extend to spectral radius conditions in the multivariate case \cite{lutkepohl_new_2005}. We extend these classical results to the HAR framework by proving first non-explosion in Theorem \ref{th3.4}, and then existence of a stationary process under a unified spectral radius condition involving both the Hawkes and autoregressive interactions in Theorem \ref{th3.9}. The proof relies on a Picard iteration argument, adapted to handle the coupled continuous-discrete structure of the model.

\textbf{Cluster representation and exponential moments.}
A key tool in the study of Hawkes processes is cluster representation introduced by Hawkes and Oakes \cite{hawkes_cluster_1974}, which allows one to decompose the process into independent clusters and derive fine probabilistic properties such as mixing and moment bounds \cite{boly_mixing_2023,leblanc_exponential_2024}. We extend this representation to linear HAR processes in Theorem \ref{th4.5}. In contrast with classical Hawkes processes, where clusters are governed by kernels of the form $h(t-s)$ depending only on the time lag, HAR clusters are described by bivariate functions $h(s,t)$ that depend separately on past and present times, reflecting the coupling between continuous and discrete dynamics. We also mention the work of \cite{roueff_locally_2016}, which introduces \textit{locally stationary Hawkes processes}, defined as a family of point processes indexed by a scaling parameter $T>0$, with cluster functions of the form $h_T(s,t) = g(t-s,t/T)$ and baseline intensities depending on $t/T$. In this setting, the authors study in particular the dependence of moments and Laplace transforms on the scaling parameter $T$. Related limit theorems are obtained in \cite{deschatre_limit_2025} for models of the form $h_T(s,t) = g(t/T)\phi(t-s)$. The HAR framework differs significantly from these settings. In our case, the cluster representation involves $1$-periodic bivariate functions satisfying $h(s,t) = h(s+1,t+1)$, which reflects an intrinsic multiscale structure rather than a slow temporal deformation driven by a scaling parameter. This representation provides a powerful tool to analyse the process and, in particular, allows us to derive exponential moment bounds in Theorem \ref{th corps moment exp}. These results build on and adapt recent developments on exponential moments and generalized cluster structures \cite{leblanc_exponential_2024,leblanc_sharp_2025}.

\textbf{Long-time behaviour, stability and ergodic properties.}
Finally, we investigate the asymptotic behaviour of HAR processes. In the spirit of stability results for Hawkes processes \cite{bremaud_stability_1996}, we show that the influence of the initial condition vanishes over time, leading to convergence toward the stationary regime. More precisely, in Theorem \ref{th6.3} we prove that two HAR processes constructed with different initial conditions converge exponentially fast toward each other as time goes to infinity, for exponentially decaying interaction functions, extending results of \cite{clinet_statistical_2017} where exponential $\Lrv^1$ convergence rates are proved in the case of exponential Hawkes processes. In addition, in Theorem \ref{th ergo} we establish ergodic-type results with non-asymptotic convergence rates, ensuring that time averages converge to their stationary counterparts. These results extend known ergodic properties of Hawkes processes to the heterogeneous, multiscale and nonlinear HAR setting. In particular, they include ergodicity in the exponential framework \cite{abergel_long-time_2015,clinet_statistical_2017}, ergodicity of the intensity process for general subcritical univariate linear Hawkes processes \cite{kwan_ergodic_2025}, and non-asymptotic results for compactly supported interaction functions \cite{hansen_lasso_2015}. We also emphasize that these results are derived within linear frameworks, whereas Theorem \ref{th ergo} establishes analogous properties for nonlinear HAR processes. Alternative approaches based on renewal, regenerative, or mixing properties \cite{costa_renewal_2020, graham_regenerative_2021, boly_mixing_2023} also provide concentration results, but are not considered here.

\subsection{Outline}

In Section \ref{sec2} we rigorously define HAR processes and introduce the notations which will be used throughout this article. Section \ref{sec3} presents existence and stationary results. Linear HAR processes are introduced in Section \ref{sec4} where we also provide the cluster representation and application to exponential moments. In Section \ref{sec5}, we present results on stability, convergence to the stationary process and ergodic properties. The proofs of results of Sections \ref{sec3}, \ref{sec4} and \ref{sec5} are given, respectively, in Sections \ref{sec6}, \ref{sec7} and \ref{sec8}. Finally, some technical results are presented in Appendices \ref{appendixA}, \ref{appendixB}, \ref{appendixC} and \ref{appendixD}.

\textbf{Notations.} We denote by $\R$ the set of real numbers and $\Z$ the set of the integers. By $\R_+$ we mean the non-negative real numbers, i.e. $[0,\infty)$, and the non-negative integers, $\{0,1,2\cdots\}$, are denoted by $\N$. We denote $\N^*$ (resp. $\R^*$, $\R^*_+$, etc) the set $\N \setminus \{0\}$ (resp. $\R \setminus \{0\}$, $\R_+ \setminus \{0\}$, etc).

\section{Definition and Representation of HAR processes}\label{sec2}

Let $I\subset\R$. A simple point process $N$ on $I$ is a random set of distinct points in $I$ such that $\card(N\cap B)$ is a random variable for evry Borel subset $B$ of $I$. $N$ is said to be \textit{non-explosive} if almost surely $N$ is locally finite. The counting measure on $N$ is denoted by $dN$. For $C\subset I$ we define $N(C):=\card(N\cap C)$. A point process is fully determined either by its counting measure or by the collection $(N(C))_{C \subset I}$. See \cite{bremaud_point_1981} for more precise statements. In the sequel we will navigate freely between these different points of view. \\
Let $\M$ be a finite set. An $\M$-multivariate point process $(N^m)_{m\in\M}$ on $I$ is a simple point process $N$ on $I$ with marks $Z\in\M$, which means that for all $x\in N$, the mark $Z(x)$ is a random variable on $\M$, such that the following holds. For all $m\in\M$, $N^m = \{x \in N \mid Z(x)=m\}$ is a point process on $I$ \cite{bremaud_point_1981}. Note that, by definition, $N^m$ and $N^{m'}$ have no points in common for $m\neq m'$.

Consider two distinct sets of labels, $\M$ of size $M\in\N^*$ and $\Pg$ of size $P\in\N^*$. Each $p \in\Pg$ is associated with a domain $\D_p$ defined by a frequency $n_p\in\N^*$ and a phase $0\leq \phi_p < 1/n_p$ by the affine transformation of the integers $\D_p = \frac{1}{n_p}\Z + \phi_p$. We also define $\Q = \{(p,k) \mid p\in\Pg, k\in\D_p\}$ the set of all possible pairs $(p,k)$. 

A process $X = \big( (N^m)_{m\in\M},(W^p)_{p\in\Pg} \big)$ where $(N^m)_{m\in\M}$ is a multivariate point process on the real line, and for $p\in\Pg$, $W^p$ is a time series on domain $\D_p$, is called an $(\M,\Pg)$-process.

HAR processes consist of $(\M,\Pg)$-processes $X$ for which some evolution equations are satisfied. These equations involve parameters described below.

Parameters associated to a HAR process are denoted $\model$ and are given by 
\begin{equation}\label{eq def model}
	\model = \big((h^{\alpha}_{m})_{\alpha,m}, (\J^{\alpha}_p,h^{\alpha}_{p},b^{\alpha}_p)_{\alpha,p}, \ (\Phi^{m})_{m}, \ (\Phi^p)_p \big),
\end{equation}
where we have for $\alpha\in\M\cup\Pg$, $m\in\M$ and $p\in\Pg$
\begin{itemize}
    \item A function $h^{\alpha}_{m} : \R_+^* \longrightarrow \R$  called \textit{interaction function of $m$ on $\alpha$}.
    \item A function $\J^{\alpha}_{p} : \R_+^*\times \R \longrightarrow \R$  called 	\textit{interaction function of $p$ on $\alpha$}, and non-negative functions $h^{\alpha}_p,b^{\alpha}_p : \R_+^* \longrightarrow \R_+$, such that for all $s > 0$ and all $w, w' \in \R$,
    \begin{equation}
    	\left\{
    \begin{aligned}
    	& \vert \J^{\alpha}_p(s,w) \vert \leq h^{\alpha}_p(s) \vert w\vert + b^{\alpha}_p(s),\\
    	& \vert \J^{\alpha}_p(s,w)-\J^{\alpha}_p(s,w') \vert \leq h^{\alpha}_p(s) \vert w-w'\vert.
    \end{aligned}
    \right.
    \end{equation}
    \item A non-negative Lipschitz function $\Phi^{m} : \R \longrightarrow \R_+$ called \textit{link function of $m$}.
    \item A Lipschitz function $\Phi^{p} : \R \longrightarrow \R$ called \textit{link function of $p$} such that $\Phi^p(0) = 0$.
\end{itemize}

We define $\Model$ as the set of all such parameters $\model$ defined in \eqref{eq def model} and we denote $\model\in\Model$.

To construct the different processes on a common probability space, we follow the formalism of \cite{delattre_hawkes_2015} to represent Hawkes processes as the solution of a system of SDEs driven by Poisson random measures.

Let $(\Omega,\F,\mathds{P})$ be a probability space with independent Poisson random measures $\pi^m$ on $\R\times\R_+$ for $m\in\M$. Let $\xi = (\xi^p_k)_{(p,k)\in\Q}$ be a collection of independent random variables, referred to as \textit{random drifts}, which are also independent of the Poisson random measures. We also require that $\xi^p_k \sim \xi^p_{k+n}$ in distribution for any $(p,k)\in\Q$ and any $n\in\Z$. Almost surely each atom of all the Poisson random measures have distinct time coordinates, thus for simplicity we assume that it is \textit{always} the case.

We define a filtration $(\F_t)_{t\in\R}$ by 
\begin{equation}\label{filtration}
	\mathcal{F}_t = \sigma\bigg(\pi^m\cap\big((-\infty,t]\times \R_+\big), \ m\in\M \ \text{and} \ \xi^p_{k}, \ (p,k)\in\Q, \ k\leq t\bigg).
\end{equation}

\begin{definition}[HAR processes]
	Let $\model\in\Model$ and random drifts $\xi$. We say that a process $X=\big((N^m)_{m\in\M},(W^p)_{p\in\Pg}\big)$, adapted to $(\mathcal{F}_t)_t$ is a $\Equ{\model,\xi}{-\infty}$ process, if we have 
    \begin{subequations}\label{eq2.2}
    \begin{empheq}[left=\empheqlbrace]{align}
        & \lambda^m_t = \Phi^m\Bigg(\sum_{m'\in\M} \itg_{-\infty}^{t-} h^{m}_{m'}(t-s)dN^{m'}_s + \sum_{p\in\Pg} \itg_{-\infty}^{t-} \J^m_p(t-k,W^p_k) d\mp{p}{k} \Bigg) \label{eq2.2a}\\
        & N^m(C) = \itg_{C\times \R_+} \ind_{x\leq \lambda^m_s} d\pi^m(s,x) \label{eq2.2b}\\
        & W^p_k = \xi^p_k + \Phi^p\Bigg(\sum_{m\in\M} \itg_{-\infty}^{k-} h^p_m(k-s)dN^{m}_s + \sum_{p'\in\Pg} \itg_{-\infty}^{k-} \J^p_{p'}(k-k',W^{p'}_{k'})d\mp{p'}{k'}\Bigg), \label{eq2.2c}
    \end{empheq}
    \end{subequations}
    where \eqref{eq2.2a} holds $\mathds{P}(d\omega)\otimes dt$ almost everywhere in $\Omega\times \R$, \eqref{eq2.2b} holds almost surely for all $C\subset\R$ and \eqref{eq2.2c} holds almost surely for all $(p,k)\in\Q$. Such a process is called a HAR process with parameters $\model$, random drifts $\xi$, and started at $-\infty$.
\end{definition}

Before providing comments on this definition, we also introduce HAR processes with initial condition. Indeed, one may wish to initialise the process on $(-\infty,t_0)$ for $t_0\in\R$ and then turn on the HAR equations at $t=t_0$.\\
An initial condition at time $t_0$ is a random variable $\mathfrak{C}$ which is $\F_{t_0 -}$ measurable and is defined as follows
\begin{equation*}
    \mathfrak{C} = \Big( \big(\mathfrak{C}^m\big)_{m\in\M}, \ \big(\mathfrak{C}^p_k\big)_{(p,k)\in\Q, \ k<t_0} \Big),
\end{equation*} 
where $\mathfrak{C}^m$ is a subset of $(-\infty,t_0)$ that initialises the points of $N^m$ before $t_0$ and $\mathfrak{C}^p_k\in\R\cup\{\varnothing\}$ initialises $W^p$ at time $k$. If $\mathfrak{C}^p_k$ takes the value $\varnothing$ we define $\J^{\alpha}_p(s,\varnothing)=0$ for all $s\in\R$, $\alpha\in\M\cup\Pg$ and $p\in\Pg$. 

\begin{definition}[HAR processes with initial condition] 
    Let $\model\in\Model$, random drifts $\xi$, $t_0\in\R$ and an initial condition $\mathfrak{C}$ at time $t_0$. We say that a process $X=\big((N^m)_{m\in\M},(W^p)_{p\in\Pg}\big)$, adapted to $(\mathcal{F}_t)_t$, is a $\Equ{\model,\xi}{t_0,\mathfrak{C}}$ process, if 
    \begin{enumerate}
        \item we have $N^m\cap(-\infty,t_0) = \mathfrak{C}^m$ for all $m\in\M$ and $W^p_k = \mathfrak{C}^p_k$ for all $(p,k)\in\Q$ with $k<t_0$,
        \item \eqref{eq2.2a} holds $d\mathds{P}(\omega)\otimes dt$ almost everywhere in $\Omega\times [t_0,\infty)$, and almost surely \eqref{eq2.2b} and \eqref{eq2.2c} hold for all $C\subset[t_0,\infty)$ and all $(p,k)\in\Q$ with $k\geq t_0$.
    \end{enumerate} 
    Such a process is called a HAR process with parameters $\model$, random drifts $\xi$, and started at $t_0$ with initial condition $\mathfrak{C}$.
\end{definition}

Some remarks are in order.

Realisation of point processes as SDEs driven by Poisson random measures is done in \eqref{eq2.2b}. The intensity process, $\lambda^m_t$, is defined by \eqref{eq2.2a} and then $N^m$ is defined as the set of atoms of the Poisson random measures lying below the intensity, which is the analogue of the classical thinning for Hawkes processes \cite{delattre_hawkes_2015,phi_kalikow_2023}. As shown in Appendix \ref{appendixD}, $\lambda$ defined as such is indeed the stochastic predictable intensity of $N$. Also, if $\tilde{\lambda}$ is another predictable process such that $\tilde{\lambda} = \lambda$ almost everywhere for $d\P{\omega}\otimes dt$, then $N$ and the corresponding $\tilde{N}$ satisfy $\E{\card(N\vartriangle \tilde{N})} = \E{\int_{\R} \vert \lambda_t-\tilde{\lambda}_t\vert dt}=0$ and thus $N=\tilde{N}$ almost surely, which explains why $\lambda$ is only defined $d\P{\omega}\otimes dt$ almost everywhere. 

As for AR processes, time series $W$ of HAR processes are constructed with random noise terms: $(\xi^p_k)_{(p,k)\in\Q}$. Since we do not require them to be centered, we adopt the terminology \textit{random drift} which is less misleading than noise. 

Without loss of generality, we can suppose throughout this article that the greatest common divisor of the frequencies $n_p, \ p\in\Pg$ is 1, so that the set $\Q$ is exactly $1$-periodic, ie
\begin{equation}\label{eq2.1}
    \gcd(n_p, \ p\in\Pg) = 1.
\end{equation}

We allow $\mathfrak{C}^p_k$ to take the value $\varnothing$ for the following reason. One may want to have a HAR process started at time $t_0$ without initial condition, thus the logical choice is $\mathfrak{C}^m=\varnothing$ and $\mathfrak{C}^p_k=0$. However, since possibly $\J^{\alpha}_{p}(s,0)\neq 0$, this choice is not the desired one. Allowing the value $\varnothing$ and forcing $\J^{\alpha}_{p}(s,\varnothing)=0$ creates the desired initial condition.

For simplicity, we only write $\Equ{\model,\xi}{t_0}$ when the initial condition $\mathfrak{C}$ is chosen to be the empty one ($\mathfrak{C}^m=\varnothing$ for all $m$ and $\mathfrak{C}^p_k=\varnothing$ for all $(p,k)\in\Q$ with $k<t_0$) and we denote this initial condition by $\mathfrak{C}^{\varnothing}$. In practice, it is equivalent to replace $-\infty$ by $t_0$ in the lower bounds of the integrals in \eqref{eq2.2a}, \eqref{eq2.2b} and \eqref{eq2.2c}.

The point process $N^m$ as predictable intensity $\lambda^m$ which satisfies \eqref{eq2.2a}. Due to the function $\Phi^m$, $\lambda^m$ is nonlinear with respect to the parameters of the model. Thus, the class of HAR processes considered here is very general. The linear quantity 
\begin{equation}\label{eq2.3}
    \nu^m_t := \sum_{m'\in\M} \itg_{-\infty}^{t-} h^{m}_{m'}(t-s)dN^{m'}_s + \sum_{p\in\Pg} \itg_{-\infty}^{t-} \J^m_p(t-k,W^p_k) d\mp{p}{k},
\end{equation}
is such that $\lambda^m_t = \Phi^m(\nu^m_t)$. It is made of two quantities: interactions from $N^{m'}$ on $N^m$ (as in Hawkes processes) and interactions from $W^p$ on $N^m$. The quantity $\nu^m$ can take negative values, but since $\Phi^m\geq 0$ the intensity $\lambda^m$ is well defined. If parameters $h^m_{m'}, \ \J^m_{p}$ for $m'\in\M$ and $p\in\Pg$ are non-negative, then $\nu^m_t\geq 0$. Thus, in this case one can take $\Phi^m = \mu_m + \id$ with $\mu_m\in\R_+$ and has $\lambda^m_t = \mu_m + \nu^m_t$. If $\Pg=\varnothing$, one obtains the well known linear Hawkes process.

The time series $W^p$ satisfies \eqref{eq2.2c}. As in classical AR processes, this equation is made of two elements: a possibly non centered noise $\xi^p_k$ called random drift, and a term that depends on the past. The link function $\Phi^p$ is required to satisfy $\Phi^p(0)=0$. This is not restrictive since a situation where $\Phi^p(0)\neq 0$ fits our case via the following modifications: $(\Phi^p,\xi^p) \leftarrow (\Phi^p-\Phi^p(0),\xi^p+\Phi^p(0))$. As for equation of the intensity, link function $\Phi^p$ allows to consider general nonlinear case. Inside the link function, we have interactions from $(N^m)_{m\in\M}$ on $W^p$ and then classical AR term of interactions from $(W^{p'})_{p'\in\Pg}$ on $W^p$.

\textbf{Notations.} Lipschitz coefficients of link functions are denoted as follows
\begin{equation}
    L_{\alpha} := \lip(\Phi^{\alpha}) \quad p\in\Pg,\ \alpha\in\M\cup\Pg.
\end{equation}

For a function $f : \R\longrightarrow \R$ and $p\geq 1$ the classical $\Lf^p$ norm is denoted $\Vert f\Vert_p$. Given a positive real number $\alpha>0$ we also define the following norm which takes into account discrete behaviour at frequency scale $\alpha$ of $f$,
\[ \Vert f\VLd{\alpha} := \sup_{x\in\R} \sum_{k\in\Z} \vert f(x+k/\alpha) \vert. \]
If $\Vert f\VLd{\alpha} < \infty$ we denote $f\in \Ld{\alpha}$. In particular, if $f$ is supported on $\R_+^*$, then since $\D_p =\frac{1}{n_p} \Z+\phi_p$ we have \[\Vert f\VLd{n_p} =  \sup_{t\in\R} \itg_{-\infty}^{t-} \vert f(t-k)\vert d\mp{p}{k}.\]
For $p\geq 1$, we denote by $\Lrv^p = \Lrv^p(\Omega)$ the set of random variables with finite moment of order $p$.

The identity matrix is denoted by $\Id_n$ with $n$ the dimension or $\Id$ if there is no ambiguity. For a (complex valued) square matrix $N$ the spectral radius of $N$, defined as the largest magnitude of eigenvalues of $N$, is denoted as $\spr(N)$.

For $q\geq 1$, the $\ell^q$ norm of a vector $x\in\R^n$ is denoted by $\vert x\vert_{q} := \big(\sum_{i=1}^n \vert x_i\vert^q\big)^{1/q}$.
Given two vectors $x$ and $y$ (or matrices) of same dimension, we use $\preceq$ for entrywise comparison, 
\[x \preceq y \ \iff \ \forall i, \ x_i \leq y_i.\]
Finally, for a vector $x$ and a real function $f$, we denote by $f(x)$ the coordinate-wise application of $f$. 

Let $X=(N,W)$ and $\tilde{X}=(\tilde{N},\tilde{W})$ two $(\M,\Pg)$-processes. Let also $F=(F^{\alpha})_{\alpha\in\M\cup\Pg}$ non-negative real functions, we define the $F$-mass of $X$ and the $F$-mass of the difference between $X$ and $\tilde{X}$ as follows ($\vartriangle$ is the symmetric difference),
\begin{equation}\label{mass F}
	\begin{aligned}
		& \Mass_F(X) = \sum_{m\in\M} \itg_{\R}F^m(t)dN^m_t + \sum_{p\in\Pg}\itg_{\R}F^p(k) \vert W^p_k\vert d\mp{p}{k},\\
		& \Mass_F(X\circleddash\tilde{X}) = \sum_{m\in\M}\itg_{\R}F^m(t) d(N^m\vartriangle\tilde{N}^m)_t + \sum_{p\in\Pg}\itg_{\R}F^p(k) \vert W^p_k-\tilde{W}^p_k\vert d\mp{p}{k}.
	\end{aligned}
\end{equation}
With a slight abuse of notations, if $B\subset\R$ and $F=(\ind_{B})_{\alpha\in\M\cup\Pg}$, denote $\Mass_B(X) := \Mass_F(X)$ and similarly $\Mass_B(X\circleddash\tilde{X}) := \Mass_F(X\circleddash\tilde{X})$.

We also introduce similar quantities for initial conditions. Let $\mathfrak{C}$ and $\tilde{\mathfrak{C}}$ two initial conditions at time $t_0 \in\R$. Let also $F=(F^{\alpha})_{\alpha\in\M\cup\Pg}$ non-negative real functions. Define the following quantities.
\begin{equation}\label{massi F}
	\begin{aligned}
		& \Massi_F(\mathfrak{C}) = \sum_{m\in\M} \itg_{\R}F^m(t)d\mathfrak{C}^m_t + \sum_{p\in\Pg}\itg_{\R}F^p(k) \psi_1(\mathfrak{C}^p_k) d\mp{p}{k},\\
		& \Massi_F(\mathfrak{C}\circleddash \tilde{\mathfrak{C}}) = \sum_{m\in\M}\itg_{\R}F^m(t) d(\mathfrak{C}^m\vartriangle \tilde{\mathfrak{C}}^m)_t + \sum_{p\in\Pg}\itg_{\R}F^p(k) \psi_2(\mathfrak{C}^p_k,\tilde{\mathfrak{C}}^p_k) d\mp{p}{k},
	\end{aligned}
\end{equation}
where $\psi_1$ and $\psi_2$ are defined by
\begin{equation}\label{def psi}
\psi_1(x)=
\begin{cases}
1+|x| & \text{if } x\in\R\\
0 & \text{if } x=\varnothing
\end{cases}
\qquad \text{and} \qquad
\psi_2(x,y)=
\begin{cases}
\vert x-y \vert & \text{if } x,y\in\R\\
\vert x\vert + 1& \text{if } x\in\R, \ y=\varnothing\\
1 + \vert y\vert & \text{if } x=\varnothing, \ y\in\R\\
0 & \text{if } x=y=\varnothing
\end{cases}.
\end{equation}
The idea is, respectively, to measure the size of an initial condition and how different the two initial conditions are. Since initial conditions are allowed to take both real values and the value $\varnothing$, we need the functions $\psi_1$ and $\psi_2$ to deal with these two types of possible values. As in \eqref{mass F}, same abuse of notations is used if $F^{\alpha} = \ind_{B}$ for all $\alpha \in\M\cup\Pg$ with $B\subset \R$. 

\section{Existence results for HAR processes}\label{sec3}

The equations defining HAR processes are implicit, as both the point processes and the time series depend on each other. Therefore, it is not a priori clear that HAR processes exist. In this section, we derive conditions ensuring the existence of HAR processes. In the case of Hawkes processes, conditions for existence and uniqueness are well understood \cite{bremaud_stability_1996,bacry_limit_2013}. The main difference here is that the Hawkes process is coupled with discrete-time sequences driven by random drifts. For HAR processes started at a time $t_0\in\R$ with suitable initial condition (see Definition \ref{def3.2}), since the present value of the process depends only on its past we can construct the process on increasing time intervals $[t_0,t_0+t]$ with $t\to\infty$, see Theorem \ref{th3.4}. However, for HAR processes started at $t=-\infty$, see Theorem \ref{th3.9}, this constructive approach is no longer applicable and we use instead a fixed point method similar to the one implemented in \cite{bremaud_stability_1996}.

\subsection{Existence of non-explosive HAR processes with initial condition}

In this subsection, we present results on the existence of HAR processes started at a finite time with suitable initial condition. We are only interested in processes that do not blow up to infinity in finite time: we need them to be non-explosive, which means that the process is bounded on compact sets. This is explained by the following definition. 

\begin{definition}[non-explosive process]
    A $(\M,\Pg)-$process $X$ is said to be non-explosive if for all bounded set $B\subset \R$ we have $\Mass_B(X) < \infty$.
\end{definition}

For a HAR process to be non-explosive, one cannot allow arbitrary initial conditions. The following definition introduces \textit{integrable} initial conditions.

\begin{definition}[Integrable initial condition]\label{def3.2}
    An initial condition $\mathfrak{C}$ at time $t_0\in\R$ is said to be \textit{integrable} if there exists a constant $K<\infty$ and a decomposition $\mathfrak{C}^m = \mathfrak{C}^m_{reg} \cup \mathfrak{C}^m_{disc}$ for $m\in\M$, $\mathfrak{C}^p = \mathfrak{C}^p_{reg} + \mathfrak{C}^p_{disc}$ for $p\in\Pg$, such that $\Massi_{(-\infty,t_0)}(\mathfrak{C}_{disc})<\infty$ almost surely and for any Borel set $B\subset \R$ we have $\E{\Massi_B(\mathfrak{C}^m_{reg})} \leq K \vert B\vert$ with $\vert B\vert$ the Lebesgue measure of $B$ and $\E{\Massi_B(\mathfrak{C}^p_{reg})} \leq K \vert B\vert_{\Pg}$ where $\vert B\vert_{\Pg} = \sum_{p\in\Pg} \mp{p}{}(B)$.
\end{definition}

To define $\mathfrak{C}^p_{reg} + \mathfrak{C}^p_{disc}$ we use the natural convention $x+\varnothing = \varnothing+x = x$ for any $x\in\R\cup\{\varnothing\}$.\\
Since we are interested in existence and uniqueness results, we need to introduce the right notion of uniqueness for HAR processes. 
\begin{definition}[Modification]\label{def3.3}
    Two processes $\big((N^m)_{m\in\M},(W^p)_{p\in\Pg}\big)$ and $\big((\tilde{N}^m)_{m\in\M},(\tilde{W}^p)_{p\in\Pg}\big)$ are said to be modifications of each other if
    \begin{itemize}
        \item Almost surely, for all $m\in\M$, $N^m  = \tilde{N}^m$.
        \item Almost surely, for all $(p,k)\in\Q$, $W^p_k = \tilde{W}^p_k$. 
    \end{itemize}
\end{definition}

Let us introduce the following assumption.

\begin{assumption}[$\noexpl$]\hypertarget{noexpl}{}
   Parameters $\model\in\Model$ satisfy assumption \Cnoexpl\ if for all $m\in\M$ the link function $\Phi^m$ is Lipschitz, i.e. $L_m <\infty$, and for all $m\in\M$, $p,p'\in\Pg$, $\alpha\in\M\cup\Pg$, we have $h^{\alpha}_{m}, h^{m}_p, b^{m}_p \in \Lf^{1}$ and $h^{p'}_p, b^{p'}_p \in \Ld{n_p}$.
\end{assumption}

We can now state the main result of this subsection.

\begin{theorem}[Existence of non-explosive HAR processes]\label{th3.4}
    Let $\model\in\Model$, random drifts $\xi$, and an initial condition $\mathfrak{C}$ at time $t_0$. Suppose that assumption \Cnoexpl$[\model]$ holds and that $\mathfrak{C}$ is integrable. Then there exists a unique (up to modification) $\Equ{\model,\xi}{t_0,\mathfrak{C}}$ process among non-explosive processes.
\end{theorem}

The assumptions of Theorem \ref{th3.4} are mild and quite general. The initial condition can be taken as Poisson point processes for $\mathfrak{C}^m$, or even some kind of sufficiently evenly spaced deterministic points. For $\mathfrak{C}^p$, one can take i.i.d. random variables with finite expectation. However, if one wants to relax the $\Lf^1/\Ld{n_p}$ assumptions of \Cnoexpl$[\model]$ into only $\Lf^1_{loc}/\mathrm{L}^{1,n_p}_{\infty,loc}$ assumptions ($f\in\Lf^1_{loc}/\mathrm{L}^{1,n_p}_{\infty,loc}$ if $f\ind_{B} \in \Lf^1/\mathrm{L}^{1,n_p}_{\infty}$ for all $B$ compact) it is possible up to considering other initial conditions. To this end one should consider \textit{finite} initial conditions $\mathfrak{C}$ at time $t_0\in\R$, which means initial conditions such that
\[ \Massi_{(-\infty,t_0)}(\mathfrak{C}) <\infty \quad \text{almost surely.}\]
Theorem \ref{th3.4} applies with finite initial conditions and $\Lf^1_{loc}/\mathrm{L}^{1,n_p}_{\infty,loc}$ assumptions instead of $\Lf^1/\Ld{n_p}$. The proof is the same except that one has to adapt \eqref{CI fini 1}, \eqref{CI fini 2}, \eqref{CI fini 3} and \eqref{CI fini 4}. If $\Pg = \varnothing$, the $\Lf^1_{loc}$ assumption matches the classical criterion of non-explosion of linear Hawkes processes with empty initial condition: interaction functions are $\Lf^1_{loc}$, see Lemma 1 \cite{bacry_limit_2013}. This criterion naturally extends to nonlinear Hawkes process by assuming in addition Lipschitz link functions.

\subsection{Existence of stationary HAR processes}

Establishing the existence of stationary processes is crucial, for instance in statistical applications \cite{hansen_lasso_2015}, and because they arise as long-time limits of non-stationary processes \cite{bremaud_stability_1996}. In general, HAR processes with initial condition are not stationary. This motivates the study of HAR processes without initial condition and started at $t=-\infty$. Since HAR processes are a mix between discrete and continuous processes, we need to define what stationary means for HAR processes.

\begin{definition}[1-stationarity]
    A process $X=\Big(\big[(\lambda^m_t)_{t\in\R}, \ m\in \M\big], \big[N^m, \ m\in \M\big], \big[(W^p_{k})_{k\in \D_p}, \ p\in\Pg\big]\Big)$ is said to be 1-stationary if $X \overset{\text{law}}{=} \Big(\big[(\lambda^m_{t+1})_{t\in\R}, \ m\in \M\big], \big[N^m+1, \ m\in \M\big], \big[(W^p_{k+1})_{k\in \D_p}, \ p\in\Pg\big]\Big).$
\end{definition}

\begin{remark}
    Stationarity is defined with a period equal to 1. We cannot decrease the period since we assume that the greatest common divisor of the $n_p$ for $p\in\Pg$ is $1$ (and thus $\Q$ is $1$-periodic but not $r$-periodic for any $r<1$), see \eqref{eq2.1}.
\end{remark}

The classical condition to prove existence of stationary Hawkes processes is that the spectral radius of the matrix containing $\Lf^1$ norms of interaction functions has to be strictly smaller than one \cite{bremaud_stability_1996}. Our condition is of the same nature and also takes into account interaction functions from the AR part. The following definition introduces the matrix on which the major assumption for the existence of 1-stationary HAR processes relies.

\begin{definition}\label{def3.8}
    Let $\model\in\Model$. Recall that for $m\in\M, \ p\in\Pg$
    \begin{equation*}
    	L_m = \lip(\Phi^m), \ L_p = \lip(\Phi^p). 
    \end{equation*}
    We introduce the following matrices:
    \[ \Hmat^N_N = (L_{m} \Vert h^m_{m'}\Vert_1)_{m,m'\in\M}, \ \Hmat^N_W = (L_m\Vert h^m_{p}\VLd{n_p})_{m\in\M,p\in\Pg}, \ \Hmat^W_N=(L_p \Vert h^p_{m}\Vert_{1})_{p\in\Pg,m\in\M}\] 
    \[\text{and} \ \Hmat^W_W=(L_p \Vert h^p_{p'}\VLd{n_{p'}})_{p,p'\in\Pg}.\]
    Finally, the matrix $\Hmod{\model}$ is defined by $\Hmod{\model} = \begin{pmatrix} \Hmat^N_N &  \Hmat^N_W \\[4pt] \Hmat^W_N &  \Hmat^W_W\end{pmatrix}$. Infinite values are allowed with the convention that $\infty \times 0 = 0 \times \infty =\infty$.
\end{definition}

Let us introduce some assumptions before stating the theorem about 1-stationary HAR processes.

\begin{assumption}[$\spec$]\hypertarget{spec}{}
    Parameters $\model\in\Model$ satisfy assumption \Cspec\ if the spectral radius of the matrix $\Hmod{\model}$ introduced in Definition \ref{def3.8} is smaller than one, equivalently $\spr(\Hmod{\model})<1$, and if $b^{\alpha}_p \in \Ld{n_p}$ for all $\alpha\in\M\cup\Pg, \ p\in\Pg$.
\end{assumption}

\begin{assumption}[$\fime$]\hypertarget{fime}{}
    Random drifts $\xi$ satisfy \Cfime\ if they are bounded in $\Lrv^1$, equivalently, $\displaystyle\sup_{(p,k)\in\Q} \E{\vert \xi^p_k\vert} < \infty$.
\end{assumption}

Now, we can state the main result concerning the existence of HAR processes started at $-\infty$. It also answers questions about uniqueness and 1-stationary properties.

\begin{theorem}[Stationary HAR processes]\label{th3.9}
    Let $\model\in\Model$ and random drifts $\xi$. Suppose that assumptions \Cspec$[\model]$ and \Cfime$[\xi]$ hold. Then, up to modification, there exists a unique $\Equ{\model,\xi}{-\infty}$ process such that
    \[ \sup_{m\in \M, \ t\in\R} \E{ \vert \lambda^m_t \vert} < \infty \ \ \text{and} \ \sup_{p\in\Pg, \ k\in\D_p} \E{ \vert W^p_{k} \vert} < \infty.\]
    Moreover, this process is 1-stationary.
\end{theorem}

For classical Hawkes processes, if $\Pg=\varnothing$, Brémaud and Massoulié \cite{bremaud_stability_1996} have proved that a stationary version exists if $\spr(\Hmat^N_N)<1$. Assumption \Cspec$[\model]$ naturally extends this condition. For AR$(\infty)$ processes $W_k = \epsilon_k + \sum_{n\geq 1} a_n W_{k-n}$ with innovations $\epsilon_k, \ k\in\Z$ i.i.d., a sufficient condition for existence and stationarity in $\Lrv^1$ (resp. $\Lrv^2$) is $\epsilon_0 \in \Lrv^1$ (resp. $\Lrv^2$) and $\sum_{n\geq 1} \vert a_n \vert <1$. Assumption \Cfime$[\xi]$ is a very mild assumption on the random drifts and naturally extends $\epsilon_0 \in \Lrv^1$. Norms $\Vert h^p_{p'}\VLd{n_{p'}}$ are the analogue of the sum $\sum_{n\geq 1} \vert a_n \vert$ and, in the multivariate case, being less than $1$ is replaced by a spectral radius condition, see \cite{lutkepohl_new_2005}. In addition to providing an existence result, Theorem \ref{th3.9} also shows uniqueness among processes bounded in $\Lrv^1$. Finally, it validates the intuition that a process started at $-\infty$ is stationary since it has already evolved over an infinite time horizon. 

\section{Cluster representation for Linear HAR processes}\label{sec4}

In order to go further in the analysis, we now focus on a subclass of HAR processes for which explicit computations become tractable. In analogy with linear Hawkes processes and their cluster representation \cite{hawkes_cluster_1974}, linear HAR processes play a central role as they also admit a cluster representation, which allows one to isolate and quantify the contribution of each point to future events. We define the $\lm$ class, the class of linear parameters, as follows. 

\begin{definition}[$\lm$ class]\label{def4.1}
Parameters $\model\in\Model$ and random drifts $\xi$ are said to be linear, or equivalently $(\model,\xi) \in  \lm$ if 
\begin{itemize}
    \item (linearity 1) $\forall m\in \M$, $\Phi^m = \mu_m + \id$ with $\mu_m\in\R_+$ and $\forall p\in \Pg$, $\Phi^p = \id$,
    \item (linearity 2) $\forall \alpha\in \M\cup\Pg$, $\forall p\in\Pg$ we have $\J^{\alpha}_p(s,w) =  h^{\alpha}_p(s)w + b^{\alpha}_p(s)$ for any $s,w\in\R$,
    \item (positiveness 1) $\forall \alpha\M\cup\Pg$, $\forall m\in\M$, we have $h^{\alpha}_{m} \geq 0$,
    \item (positiveness 2) $\forall (p,k)\in\Q$, we have $\xi^p_k \geq 0$.
\end{itemize}
\end{definition}

\begin{remark}
    To ensure that the intensity remains non-negative, if there is an initial condition it has to be non-negative (ie $\mathfrak{C}^p_k \geq 0$ or $\mathfrak{C}^p_k = \varnothing$). Since by definition $h^{\alpha}_p,b^{\alpha}_p$ are already non-negative assumption (positiveness 1) is only about $h^{\alpha}_{m}$.
\end{remark}

For $(\model,\xi)\in\lm$, equation \eqref{eq2.2c} writes as follows,
\[W^p_k = \xi^p_k + \sum_{p'\in\Pg} \itg_{-\infty}^{k-} b^p_{p'}(k-k')d\mp{p'}{k'} + \sum_{m\in\M} \itg_{-\infty}^{k-} h^p_m(k-s)dN^{m}_s + \sum_{p'\in\Pg} \itg_{-\infty}^{k-} h^p_{p'}(k-k')W^{p'}_{k'}d\mp{p'}{k'}.\]
Equation \eqref{eq2.2c} is basically a classical linear AR equation on $W$. Heuristically one has,
\begin{align*}
    W^p_k & = \xi^p_k + (\text{non random term}) + (\text{points linear contribution up to} \ k)  + \sum_{p',k'<k} h^{p}_{p'}(k-k') W^{p'}_{k'}\\
    & = (\text{something that depends on k}) + \sum_{p',k'<k} h^{p}_{p'}(k-k') W^{p'}_{k'}.
\end{align*} 
This recursive relation is the starting point to derive the cluster representation for HAR processes. By solving it completely we will be able to express the chains only in terms of linear functions of points and random drifts. Then introducing this expression of the chains in the intensity equation will lead us to the cluster representation. In the next subsections we solve completely this type of recursive relation to then be able to apply these results to HAR processes. 

\subsection{Solving linear recursive relation of $W$}

In this section we solve the linear recursive relation of $W$ on itself. In the simple one dimensional case it looks like this. Let $(w_n)_{n\geq 0}$ a sequence defined as follows
\[ w_n = u_n + \sum_{0\leq k<n} h(n-k)w_k \quad \iff \quad w = u + h\star w,\]
where the sequence $u$ is given and $h$ is a fixed (say non-negative) function. A natural wish is to have an expression of $w_n$ depending only on the sequence $u$. To do so, we need to unroll the recursive relationship. Iterating this relation suggests that $w_n$ can be expressed as a linear combination of the sequence $u$, namely,
\[ w_n = \sum_{0\leq k\leq n} C^n_k u_k,\]
with $(C^n_k)_{n\geq k}$ be some coefficients depending only on $h$, basically they are $(\delta_0-h)^{\star -1}$, see classical AR theory \cite{box_box_2013,shumway_time_2025,lutkepohl_new_2005}. These coefficients grasp the complexity of the nested dependence of the sequence $w$ on itself. The goal of this section is to prove useful properties on these coefficients, in the multidimensional case (since we have $P$ times series) and on a domain $(\D_p)_{p\in\Pg}$, more complex than the integers.

The following lemma introduces coefficients $\coef{p,k}{p',k'}{\model}$ that encode the cumulative effect of the recursive interactions, analogues of the coefficients $C^k_{k'}$'s in the multidimensional case.

\begin{lemma}\label{lem4.3}
    Let $\model\in\lm$. There exists a unique set of coefficients $\left(\coef{p,k}{p',k'}{\model}\right)_{(p,k),(p',k')\in\Q, \ k\geq k'}$ such that,
    \begin{enumerate}
        \item For any $(p_0,k_0)\in\Q$, for any sequence $(Z^p_{k})_{(p,k)\in\Q, \ k\geq k_0}$, the sequence $(V^p_{k})_{(p,k)\in\Q, \ k\geq k_0}$ defined by
        \begin{equation}\label{eq4.1}
            V^p_{k} = Z^p_{k} + \sum_{p'\in\Pg} \itg_{k_0}^{k-} h^{p}_{p'}(k-k') V^{p'}_{k'} d\mp{p'}{k'}
        \end{equation} 
        satisfies the following equality, $V^p_{k} = \sum_{p'\in\Pg} \itg_{k_0}^{k} \coef{p,k}{p',k'}{\model} Z^{p'}_{k'} d\mp{p'}{k'}$.
        \item The coefficients $\left(\coef{p,k}{p',k'}{\model}\right)_{(p,k),(p',k')\in\Q, \ k\geq k'}$ are 1-periodic, which means that for every $n\in\Z$, $\coef{p,k+n}{p',k'+n}{\model} = \coef{p,k}{p',k'}{\model}$.
        \item For all $(p_0,k_0)\in\Q$ we have $\coef{p_0,k_0}{p_0,k_0}{\model} = 1$. If there exists $(p_1,k_1)\in\Q$ with $k_1=k_0$ then $\coef{p_1,k_0}{p_0,k_0}{\model}=\ind_{p_0=p_1}$. Finally for $(p',k')\in\Q$ with $k>k'$,
        \[ \coef{p,k}{p',k'}{\model} = \sum_{p''\in\Pg} \itg_{k'}^{k-} h^{p}_{p''}(k-k'') \coef{p'',k''}{p',k'}{\model} d\mp{p''}{k''} \geq 0.\]
    \end{enumerate}
\end{lemma}

Lemma \ref{lem4.3} is about the combinatorial properties of the coefficients. The coefficients only depend on $\model$, and they can be computed recursively with the third point of Lemma \ref{lem4.3}.

Recall from Definition \ref{def3.8} that $\Hmat^W_W$ is the matrix $(L_p \Vert h^p_{p'}\VLd{n_{p'}})_{p,p'\in\Pg}$. Let us consider the following assumption.

\begin{assumption}[$\specbis$]\hypertarget{specbis}{}
    Parameters $\model\in\Model$ satisfy assumption \Cspecbis\ if the spectral radius of the matrix $H_W^W$ defined in Definition \ref{def3.8} is smaller than one, equivalently $\spr(\Hmat^W_W)<1$.
\end{assumption}

Under assumption \Cspecbis$[\model]$ we obtain analytic results on the coefficients, in particular summability of coefficients, see Lemma \ref{lemA.1}, and we also extend the use of coefficients to random sequences, see Proposition \ref{propA.2}, both in Appendix \ref{appendixA}. These results are useful in the sequel, for example to obtain controls on clusters of HAR processes.

We end this section with the following result that gives a closed form for the coefficients by giving a sense to the multivariate convolution on the domain $(\D_p)_{p\in\Pg}$. 

Let $(p,k)\in\Q$ and $(p',k')\in\Q$ with $k'<k$. A path $\gamma$ from $(p',k')$ to $(p,k)$ is a sequence $(p_1,k_1),\cdots,(p_n,k_n) $ in $\Q$, where $n\geq 2$ is an integer and such that $(p_1,k_1)=(p',k')$, $(p_n,k_n)=(p,k)$ and $k_1<\cdots<k_n$. Given parameters $\model\in\lm$ and a path $\gamma = (p_1,k_1),\cdots,(p_n,k_n) $ from  $(p',k')$ to $(p,k)$ we denote 
\[Y_{\model}(\gamma) := \prod_{i=1}^{n-1} a^{p_{i+1}}_{p_i} h^{p_{i+1}}_{p_i} (k_{i+1}-k_i).\]

\begin{proposition}[Closed form for coefficients]\label{prop4.4}
    Let $\model\in\lm$. Let $(p',k')\in\Q$ and $(p,k)\in\Q$ with $k'<k$. Then we have
    \[ \coef{p,k}{p',k'}{\model} = \sum_{\gamma : (p',k')\to(p,k)} Y_{\model}(\gamma).\]
\end{proposition}

This formula shows that $\coef{p,k}{p',k'}{\model}$ can be interpreted as sums over all possible propagation paths from $(p',k')$ to $(p,k)$.

\subsection{Cluster representation for linear HAR processes}

In this section we use the coefficients defined in the previous one to express $W$ only in terms of $N$ and $\xi$ in order to derive a cluster representation for linear HAR processes. Let parameters and random drifts $(\model,\xi)\in\lm$ and $t_0\in \{-\infty\}\cup\R$. We state the result for HAR processes without initial condition.

Let us define quantities useful to express $W$ as functional of $\xi$ and $N$. For $(p,k)\in\Q$ with $k\geq t_0$ and $t_0 \leq s \leq k$, define,

\begin{equation}\label{eq4.2}
        \left\{
        \begin{aligned}
            & \Hd^{b}_{p,k}(t_0)  = \sum_{p'\in\Pg}\itg_{t_0}^k \coef{p,k}{p',k'}{\model} \left[\sum_{p''\in\Pg} \itg_{t_0}^{k'-} b^{p'}_{p''}(k'-k'')d\mp{p''}{k''}\right] d\mp{p'}{k'}\\
            & \Hd^{\xi}_{p,k}(t_0)  = \sum_{p'\in\Pg}\itg_{t_0}^k \coef{p,k}{p',k'}{\model} \xi^{p'}_{k'} d\mp{p'}{k'}\\
            & \Hd^m_{p,k}(t_0,s)  = \sum_{p'\in\Pg}\itg_{t_0}^k \coef{p,k}{p',k'}{\model}h^{p'}_m(k'-s)d\mp{p'}{k'}.
        \end{aligned}
        \right.
\end{equation}

Each of these quantities encloses the total impact at time $k$ on $W^p$ of, from up to down, functions $b_{p'}^{p}$, the random drifts $\xi$, point process $N^m$. If $t_0 = -\infty$ we may omit $t_0$ in the notation an only write only $\Hd^{b}_{p,k}$ and so on.

The following result is the main theorem of this section, it states a cluster representation for linear models. The idea is to express the process $W$ as a linear functional of past points and random drifts, and then plug this representation into the intensity. This leads to a representation of the intensity as a sum of independent contributions, which can be interpreted as clusters. 

\begin{theorem}[Cluster representation for linear HAR processes]\label{th4.5}
    Let $(\model,\xi)\in\lm$. If $t_0\in \R$ suppose that assumption \Cnoexpl$[\model]$ holds and if $t_0 = -\infty$ suppose that assumptions \Cspec$[\model]$ and \Cfime$[\xi]$ hold. Then the $\Equ{\model,\xi}{t_0}$ process from Theorem \ref{th3.4} (if $t_0\in\R$) or Theorem \ref{th3.9} (if $t_0=-\infty$) is given by
    \[ W^p_{k} = \Hd^{b}_{p,k}(t_0) + \Hd^{\xi}_{p,k}(t_0) + \sum_m \itg_{t_0}^{k-} \Hd^m_{p,k}(t_0,s) dN^m_s, \quad (p,k)\in\Q, \ k\geq t_0, \]
    and there exists a cluster representation given by
    \begin{equation}\label{eq cluster rep}
        \lambda^m_t  = \I^m(t_0,t) + \sum_{m'} \itg_{t_0}^{t-} \mathfrak{h}^m_{m'}(s,t)dN^{m'}_s, \quad m\in\M, \ t\geq t_0.
    \end{equation} 
    Any $s\in N^{m_0}$ increases the intensity at time $t>s$ of $N^m$ by
   \begin{equation}\label{eq4.3}
       \mathfrak{h}^m_{m_0}(s,t) = h^m_{m_0}(t-s) + \sum_{p,p'\in\Pg}\itg_{s^+}^{t-}\itg_{s^+}^k  h^m_{p}(t-k) \coef{p,k}{p',k'}{\model}h^{p'}_{m_0}(k'-s) d\mp{p'}{k'} d\mp{p}{k},
   \end{equation}
   and the immigrant rate is given by
   \begin{equation}\label{eq4.4}
       \I^m(t_0,t) = \mu_m + \sum_{p\in\Pg} \itg_{t_0}^{t-} \left[ h^m_{p}(t-k)\times(\Hd^{b}_{p,k}(t_0) + \Hd^{\xi}_{p,k}(t_0))+b^m_p(t-k)\right]d\mp{p}{k}.
   \end{equation}
    Finally if assumption \Cspecbis$[\model]$ holds (which is the case under assumption \Cspec$[\model]$), 
    \begin{equation}\label{eq4.5}
        \Vert \mathfrak{h}^m_{m_0}(s,\cdot)\Vert_{1} \leq \Vert h^m_{m_0}\Vert_1 + \sum_{p,p'\in\Pg} n_p \Vert h^m_p\Vert_1 \Big[ (\Id-\Hmat^W_W)^{-1} \Big]_{p,p'} \Vert h^{p'}_{m_0} \VLd{1}.
    \end{equation}
\end{theorem}

This representation shows that HAR processes can be decomposed into independent clusters, similarly to Hawkes processes, but with a more involved interaction structure due to the presence of the discrete component. Each immigrant point (coming from the immigration part of the intensity) is a new root for a new cluster. Indeed, each immigrant point, say from $N^m$ at time $s$, will have some children since the intensity of $N^{m'}$ at time $s+t$ is increased by $\mathfrak{h}^{m'}_m(s,s+t)$. Similarly, each child will have children and so on, creating a cluster, i.e. the whole family of the initial immigrant point. Conditionally to the immigrants, all the clusters are independent. Cluster representation is a powerful tool, to derive result on the whole process; one can look at what happens for a single cluster and then it only remains to superpose the contributions of each cluster, which can be performed easily since clusters are independent. Remark that if $\Pg=\varnothing$ then sums over $\Pg$ disappears and we are left with $\mathfrak{h}^m_{m_0}(s,t) = h^m_{m_0}(t-s)$ and $\I^m(t_0,t) = \mu_m$ with is the classical linear Hawkes process.

\subsection{Structure of clusters}\label{sec4.3}

In this subsection we give a precise description of HAR clusters. Clusters are branching structures where each point generates offspring according to the functions $\mathfrak{h}^m_{m'}$. The following definition recalls the expression of the $\mathfrak{h}^m_{m'}$'s and introduces a relevant matrix.

\begin{definition}[Cluster functions and matrix]\label{def4.7}
    Let parameters $\model\in\lm$ . The \textit{cluster functions} associated to linear parameters $\model$ are the functions $\mathfrak{h}^m_{m'}$ for $m,m'\in\M$ defined in Theorem \ref{th4.5} by, for $s<t$,
    \[\mathfrak{h}^m_{m'}(s,t) = h^m_{m'}(t-s) + \sum_{p,p'\in\Pg}\itg_{s+}^{t-}\itg_{s+}^k h^m_{p}(t-k) \coef{p,k}{p',k'}{\model}h^{p'}_{m'}(k'-s) d\mp{p'}{k'} d\mp{p}{k}.\]
    For $\model\in\lm$, the matrix $\mathfrak{H}_{\model}$ is defined by
    \begin{equation}\label{eq4.7}
    \mathfrak{H}_{\model} = \left( \sup_{s\in\R}  \Vert \mathfrak{h}^{m}_{m'}(s,\cdot)\Vert_{1} \right)_{m,m'\in\M}.
    \end{equation} 
\end{definition}

Given parameters $\model\in\lm$, we are interested in the clusters generated by the cluster functions $\mathfrak{h}^m_{m'}$ defined in Definition \ref{def4.7}. As we will see later, the matrix $(\mathfrak{H}_{\model})^{\intercal}$, defined in Definition \ref{def4.7}, gives a uniform bound on the expectation of the offspring distribution of clusters. This section follows \cite{leblanc_sharp_2025}. 

Let $\bm{\mathfrak{h}} = (\mathfrak{h}^{m}_{m'})_{m,m'\in\M}$. Clusters of linear HAR processes correspond to $\Pois( \bm{\mathfrak{h}}^{\intercal})$ clusters in \cite{leblanc_sharp_2025}. We explain below what it means, but for further details we refer the reader to \cite{leblanc_sharp_2025}. 

\begin{remark}
	The relevant matrix is $\bm{\mathfrak{h}}^{\intercal}$ and not $\bm{\mathfrak{h}}$ since $\mathfrak{h}^m_{m'}$ corresponds to how many (in expectation) children of type $m$ a point from $N^{m'}$ will have, whereas, in the proof of Theorem \ref{th3.9} for example, $\mathfrak{h}^m_{m'}$ corresponds to how much past points of $N^{m'}$ modify the intensity of $N^m$. This small change of perspective, backward or forward, is the explanation on why $\bm{\mathfrak{h}}^{\intercal}$ and $(\mathfrak{H}_{\model})^{\intercal}$ are the relevant quantities for clusters.
\end{remark}

We use the Ulam Harris Neveu notation to label the individuals of trees, see Figure \ref{fig2}, and thus a tree is viewed as a subset of $\U = \cup_{n\in\N} (\N^*)^n$ with $(\N^*)^0=\{\varnothing\}$. 

Given parameters $\model\in\lm$, a cluster $G^{m_0}_{t_0}$ with root of type $m_0\in\M$ and born at time $t_0\in\R$ is defined by $\big((u,\tp(u),\bd(u))\big)_{u\in\T}$ where
\begin{itemize}
    \item $\T\subset \U$ is a random tree. For $u\in\T$, $\tp(u)\in\M$ is the type of $u$ and $\bd(u)\in\R$ is the birth date of $u$.
    \item $\T$ has root $\varnothing$ of type $m_0$ (ie $\tp(\varnothing)=m_0$) and born at time $t_0$ (ie $\bd(\varnothing)=t_0$).
    \item Independently each individual $u\in\T$ reproduces as follows. If we denote $Z(u,m)$ the number of children of $u$ with type $m\in\M$, then for $m\in\M$ the random variables $Z(u,m)$ are independent and 
    \[Z(u,m) \sim \Pois\big(\Vert\mathfrak{h}^m_{\tp(u)}(\bd(u),\bd(u)+\cdot) \Vert_{1}\big) \quad \forall m\in\M.\]
    The birth dates of the children of $u$ are distributed as follows. Conditionally to the children of $u$ and their types, independently for any child $v\in\T$ of $u$, we have
    \[\bd(v) \sim \frac{\mathfrak{h}^{\tp(v)}_{\tp(u)}(\bd(u),t)}{ \Vert\mathfrak{h}^{\tp(v)}_{\tp(u)}(\bd(u),\bd(u)+\cdot) \Vert_{1}} \ind_{t>\bd(u)}dt.\]
\end{itemize}

One may remark that $ \mathfrak{h}^{\tp(v)}_{\tp(u)} = \big(\bm{\mathfrak{h}}^{\intercal}\big)_{\tp(u),\tp(v)}$, thus, as in \cite{leblanc_sharp_2025}, row indexes correspond to father/mother and column indexes correspond to children.

Finally, if one omits the birth dates of a cluster $G$, it remains $(u,\tp(u))_{u\in\T}$. It is clear that $(u,\tp(u))_{u\in\T}$ can be stochastically dominated by $(u,\tp_{\mathcal{A}}(u))_{u\in\mathcal{A}}$ (which means that $\T \subset \mathcal{A}$ and for $u\in\T$, we have $\tp(u)=\tp_{\mathcal{A}}(u)$) where $(u,\tp_{\mathcal{A}}(u))_{u\in\mathcal{A}}$ is a $\Pois(\mathfrak{\Hmod{\model}}^{\intercal})$ Galton Watson process. It means that independently, any point $u\in \mathcal{A}$ with type (in $\mathcal{A}$) $\tp_{\mathcal{A}}(u)$ reproduces as follows: independently for all $m\in\M$ it has $\Pois\big(\big[\mathfrak{H}_{\model}^{\intercal}\big]_{\tp_{\mathcal{A}}(u),m}\big)$ children with type $m$.

Given this precise structure of clusters, cluster representation takes the following form, which is a reformulation of \eqref{eq cluster rep}.

\begin{proposition}[Cluster representation]\label{prop4.9}
    Let $(\model,\xi)\in\lm$ and $t_0\in\R\cup\{-\infty\}$. Suppose that we are under the same assumptions as Theorem \ref{th4.5}. Let $(\Pi^m)_{m\in\M}$ be point processes such that, conditionally to $(\xi^p_k)_{(p,k)\in\Q}$, $(\Pi^m)_{m\in\M}$ are independent inhomogeneous Poisson processes with intensities given by the immigrant rates in Theorem \ref{th4.5}. Conditionally to $(\Pi^m)_{m\in\M}$ let independent $\Pois(\bm{\mathfrak{h}}^{\intercal})$ clusters $G^m_x$ with root of type $m$ and born at time $x$ for $m\in\M$ and $x\in\Pi^m$. Then we have the following equality in distribution
    \begin{equation*}
        (N^m)_{m\in\M} \overset{d}{=} \big(\{ s\in\R \mid \exists m'\in\M, \exists x\in\Pi^{m'}, \exists u\in \U \ \text{such that} \ (u,m,s)\in G^{m'}_x\}\big)_{m\in\M},
    \end{equation*}
    where $\big((N^m)_{m\in\M},(W^p)_{p\in\Pg}\big)$ is the corresponding $\Equ{\model,\xi}{t_0}$ process in Theorem \ref{th4.5}.
\end{proposition}

Obtaining controls on clusters is crucial to derive properties for the whole process. For example, exponential bounds on clusters lead to exponential moments for the whole process, as explained in the following section.

\begin{figure}[ht]
\begin{center}
\begin{tikzpicture}[scale=0.7]

\draw (0,0)--(-1,1)--(-1,2)--(-1,3)--(-1,4);
\draw (-1,2)--(-1.5,3);
\draw (-1,3)--(-1.5,4);
\draw (-1,3)--(-0.5,4);
\draw (-1.5,4)--(-1.5,5);
\draw (0,0)--(1,1)--(1,2)--(1,3)--(1,4);
\draw (1,1)--(0.5,2);
\draw (1,1)--(1.5,2);

\draw (0,-0.5) node {$(u=\varnothing,\tp(\varnothing)=m_0,\bd(\varnothing)=t_0)$};
\draw (-2,1) node {$u=0$};
\draw (2,1) node {$u=1$};
\draw (-1.5,2) node {$00$};
\draw (0.3,1.7) node {$10$};
\draw (1.22,2.3) node {$11$};
\draw (1.7,1.7) node {$12$};
\draw (2,4.5) node {$(u,\tp(u),\bd(u))$};

\tikzstyle{every node}=[circle, draw, fill=RoyalBlue!80, inner sep=0pt, minimum width=5pt]

\draw (0,0) node {}; 
\draw (-1,1) node {};
\draw (1,1) node {}; 
\draw (-1,2) node {};
\draw (0.5,2) node {}; 
\draw (1.5,2) node {};
\draw (1,2) node {};
\draw (-1.5,3) node {};
\draw (-1,3) node {};
\draw (1,3) node {};
\draw (-1,4) node {};
\draw (-1.5,4) node {};
\draw (-0.5,4) node {};
\draw (1,4) node {};
\draw (-1.5,5) node {};

\end{tikzpicture}
\end{center}
\caption{Illustration of a cluster. The tree $\T$ is represented with the types and birth dates. The ancestor is $u=\varnothing$, it has two children, $u=0$ and $u=1$. $u=0$ has one child ($u=00$), $u=1$ has three children ($u=10$, $u=11$ and $u=12$) and so on.}
\label{fig2}
\end{figure}

\subsection{Application to exponential moments}

Linear HAR processes are of interest since for general couples $(\model,\xi)$, one can find $(\model^+,\xi^+) \in\lm$ such that HAR processes from $(\model,\xi)$ are dominated by those from $(\model^+,\xi^+)$.

\begin{definition}[$\lm$ domination]\label{def5.1}
    Let parameters $(\model,\xi)$, possibly nonlinear, given by 
    \[\model = \big((h^{\alpha}_{m})_{\alpha,m}, (\J^{\alpha}_p,h^{\alpha}_{p},b^{\alpha}_p)_{\alpha,p}, \ (\Phi^{m})_{m}, \ (\Phi^p)_p \big).\]
    We define the parameters $(\model^+,\xi^+)\in \lm$ by $(\xi^+)^p_k = \vert \xi^p_k\vert$ for any $(p,k)\in\Q$ and the following tabular.
    \vspace{0.5\baselineskip}
    \begin{center}
    \begin{tabular}{ |c|c|c|c|c|c| } 
    \hline
    $\model$ & $h^{\alpha}_m$ & $\J^{\alpha}_{p}(s,w)$ & $(h^{\alpha}_p,b^{\alpha}_p)$ & $\Phi^m$ & $\Phi^p$ \\ 
    \hline
    $\model^+$ & $L_{\alpha}\vert h^{\alpha}_m \vert$ & $L_{\alpha}\big(h^{\alpha}_{p}(s) w + b^{\alpha}_{p}(s)\big)\ind_{w\neq \varnothing} $ & $(L_{\alpha}h^{\alpha}_p,L_{\alpha}b^{\alpha}_p)$ & $\id+\Phi^m(0)$ & $\id$\\
    \hline
    \end{tabular}
    \end{center}
    \vspace{0.5\baselineskip}
    For any initial condition $\mathfrak{C}$ at time $t_0$, we define the initial condition $\mathfrak{C}^+$ at time $t_0$ by $\mathfrak{C}^{+, m} = \mathfrak{C}^m$, for all $m\in\M$ and for all $(p,k)\in\Q$ with $k<t_0$, $\mathfrak{C}^{+, p}_k = \varnothing$ if $\mathfrak{C}^{p}_k = \varnothing$ and $\mathfrak{C}^{+, p}_k = \vert \mathfrak{C}^{p}_k\vert$ if $\mathfrak{C}^{p}_k \in\R$. 
\end{definition}

Precise statements on the domination of $\Equ{\model,\xi}{t_0,\mathfrak{C}}$ by $\Equ{\model^+,\xi^+}{t_0,\mathfrak{C}^+}$ can be found in Proposition \ref{propB.1} in Appendix \ref{appendixB}. 

To derive exponential moment for HAR processes we need exponential bounds on clusters. Such bounds require assumptions on the cluster functions $\mathfrak{h}^m_{m'}$. As for classical Galton Watson processes, the relevant assumption is about the spectral radius of the matrix that collects the size of the interactions. Let us introduce the assumption in our case.

\begin{assumption}[$\specter$]\hypertarget{specter}{} 
    Let $\model\in\Model$ and $\model^+\in\lm$ the associated linear parameters from Definition \ref{def5.1}. We say that $\model$ satisfies \Cspecter\ if the spectral radius of the matrix $\mathfrak{H}_{\model^+}$ defined by \eqref{eq4.7} is smaller than one, equivalently $\spr(\mathfrak{H}_{\model^+}) <1$, and if $\Vert h^p_m\VLd{1} < \infty$ for all $m\in\M$ and all $p\in\Pg$. 
\end{assumption}

\begin{remark}
	Since the spectrum of a matrix and its transpose matrix is the same, in assumption \Cspecter\ it is equivalent to require $\spr\big(\mathfrak{H}_{\model^+}^{\intercal}\big) < 1$, which is the condition to have subcritical clusters in \cite{leblanc_sharp_2025}.
\end{remark}

Finally, to obtain exponential moment we also need assumption on the random drifts.

\begin{assumption}[$\subg$]\hypertarget{subg}{}
    Random drifts $\xi$ satisfies assumption \Csubg\ if there exist constants $\mathfrak{e}, \mathfrak{s}>0$ such that for all $(p,k)\in\Q$, for all $t\geq 0$ we have, \[ \E{e^{t\vert \xi^p_k \vert}} \leq e^{\mathfrak{e} t + \mathfrak{s} t^2}.\]
\end{assumption}

Let $F=(F^{\alpha})_{\alpha\in\M\cup\Pg}$ non-negative real functions. Let $n\in\N^*$ and define $\kappa_n(F)$ by
\begin{align*}
	 \kappa_n(F) = \max\Big( \max_{m\in\M} \Vert F^m\VLd{1/n}, \ \max_{p\in\Pg}\Vert F^p\VLd{n_p}\Big).
\end{align*}

\begin{theorem}[Exponential moments for HAR processes]\label{th corps moment exp}
	Let $\model\in\Model$ and random drifts $\xi$. Assume that assumptions \Cspec$[\model]$, \Cspecter$[\model]$ and \Csubg$[\xi]$ hold. Let $X$ the $\Equ{\model,\xi}{t_0}$ from Theorem \ref{th3.9} if $t_0=-\infty$ or from Theorem \ref{th3.4} if $t_0\in\R$. Then there exist $r_0>0$ depending only on $\model$ and $K>0$ depending only on $\model$ and \Csubg$[\xi]$ such that for any $n\in\N^*$,
	\begin{equation}
		\kappa_n(F) \leq r_0 \quad \Longrightarrow \quad \E{e^{\Mass_{F}(X)}} \leq e^{K n^2 \kappa_n(F)}.
	\end{equation}
\end{theorem}

This result is an extension of exponential moments for Hawkes processes, see \cite{leblanc_exponential_2024}. The $n^2$ dependence comes from assumption \Csubg$[\xi]$. The fact that $K$ also depend on \Csubg$[\xi]$ means that $K$ depend on the constants $\mathfrak{e},\mathfrak{s}$ involved in the assumption. Note that if $L\in\N^*$, by taking $F^{m}(t) = r \ind_{[0,L)}(t)$ with $r\leq r_0$, $n=L$ we obtain $\Vert F^m\VLd{1/n} = r\leq r_0$ and $\Mass_{F^m}(N^m) = N^m([0,L))$. 

The question now is about links between spectral conditions in \Cspec$[\model]$ (assumption to have existence of stationary processes) and \Cspecter$[\model]$ (assumption to have subcritical clusters). More precisely the link between $\spr(\Hmod{\model})<1$ and $\spr(\mathfrak{H}_{\model^+})<1$. Both are equivalent in the case of Hawkes processes since in this case $\Hmod{\model}=\mathfrak{H}_{\model^+}=\Hmat^N_N$. For HAR processes it is more complex. Let $\model\in\Model$ such that $\spr(\Hmod{\model})<1$ with the matrix $\Hmod{\model}$ defined in Definition \ref{def3.8} by $\Hmod{\model} = \begin{pmatrix} \Hmat^N_N &  \Hmat^N_W \\[4pt] \Hmat^W_N &  \Hmat^W_W\end{pmatrix}$. The spectral part of assumption \Cspecter$[\model]$ is $\spr(\mathfrak{\Hmod{\model^+}})<1$ with $\mathfrak{H}_{\model^+}$ defined by \eqref{eq4.7}. The bound \eqref{eq4.5} given in Theorem \ref{th4.5} implies that
\begin{equation}\label{eq5.1}
\mathfrak{H}_{\model^+} \preceq \Hmat^N_N + \tilde{\Hmat}^N_W (\Id-\Hmat^W_W)^{-1} \tilde{\Hmat}^W_N,
\end{equation}
where $\tilde{\Hmat}^N_W = \Big(n_p L_m \Vert h^m_p\Vert_1\Big)_{m,p}$ and $\tilde{\Hmat}^W_N = \Big( L_p \Vert h^p_m\VLd{1}\Big)_{p,m}$. The difference between $(\tilde{\Hmat}^N_W,  \tilde{\Hmat}^W_N)$ and $(\Hmat^N_W, \Hmat^W_N)$ is that when the classic $\Lf^1$ norm is used on one side, the discrete $\Ld{\cdot}$ norm is used on the other side. Indeed,
\[\Hmat^N_W = \Big(L_m \Vert h^m_p\VLd{n_p}\Big)_{m,p} \quad \text{and} \quad \Hmat^W_N = \Big( L_p \Vert h^p_m\Vert_1\Big)_{p,m}.\]
This is due to the fact that we are integrating in the past (between $-\infty$ and $t$) to establish the existence of the stationary HAR processes (leading to $\Hmod{\model}$) while we are integrating in the future (between $s$ and $\infty$) to describe clusters (leading to $\mathfrak{H}_{\model^+}$).\\
On the one hand, $\tilde{\Hmat}^N_W$ and $\Hmat^N_W$ compare in the right sense: $\tilde{\Hmat}^N_W \preceq \Hmat^N_W$ since $n_p \Vert \cdot \Vert_1 \leq \Vert \cdot \VLd{n_p}$ but on the other hand, we have $\tilde{\Hmat}^W_N \succeq \Hmat^W_N$, the wrong sense, since $ \Vert \cdot \Vert_1 \leq \Vert \cdot \VLd{1}$. Thus we cannot conclude on $\text{\Cspec}[\model] \Rightarrow \text{\Cspecter} [\model^+]$. However, we are able to prove this statement.

\begin{proposition}\label{prop5.6}
    Let $\model\in\Model$. Then we have
    \[\spr\left(\tilde{\Hmat}_{\model} := \begin{pmatrix} \Hmat^N_N &  \Hmat^N_W \\[4pt] \tilde{\Hmat}^W_N &  \Hmat^W_W\end{pmatrix}\right) < 1 \ \ \Longrightarrow \ \ \spr(\Hmod{\model})<1 \ \ \text{of \Cspec}[\model] \ \ \text{and} \ \ \spr(\mathfrak{H}_{\model^+})<1 \ \ \text{of \Cspecter} [\model].\]
\end{proposition}

\section{Asymptotic stability and ergodic theorem}\label{sec5}

In the study of stochastic processes, understanding whether a process converges to a stationary regime plays a key role. A positive answer     allows us to transfer questions about the non-stationary process into the study of the stationary one. It also opens the door to some creative ideas such as the coupling presented at the end of this section. This coupling ultimately leads to powerful concentration inequalities with numerous applications, such as in statistics.

In this section, given $\delta >0$ and $t\in\R$, notations $\Mass^{\exp}_{\delta,t}$ and $\Massiexp_{\delta,t}$ stand for $\Mass_{F}$ and $\Massi_{F}$, respectively introduced in \eqref{mass F} and \eqref{massi F}, where, 
\begin{equation}\label{notation mass exp}
	F^{\alpha}(s) = e^{-\delta\vert t-s\vert} \quad \text{for all} \quad \alpha\in\M\cup\Pg.
\end{equation} 

\subsection{Rate of convergence to the stationary state}

Our main interest is about the convergence of HAR processes started at some time $t_0\in\R$ to the stationary HAR process. To remain as general as possible, the main theorem of this section states the convergence of two HAR processes with distinct initial conditions towards each other. Then one can specify this result with the empty initial condition and the stationary initial condition to obtain the convergence of the non-stationary process to the stationary one. This is done in Corollary \ref{cor6.4}.

Let us introduce the following assumption.

\begin{assumption}[$\exptail$]\hypertarget{exptail}{}
    Parameters $\model\in\Model$ satisfy assumption \Cexptail\ if there exist constants $K,c>0$ such that for all $\alpha,\beta\in\M\cup\Pg$, all $p\in\Pg$, and all $x > 0$, 
    \[ \vert h^{\alpha}_{\beta}(x) \vert  \leq K e^{-cx} \quad \text{and} \quad \vert b^{\alpha}_{p}(x) \vert  \leq K e^{-cx}.\]
\end{assumption}
This tail assumption is the starting point to derive rate of convergence of two HAR processes with different initial conditions towards each other. This tail assumption transfers to coefficients and clusters functions introduced in Section \ref{sec4} as explained in the following result.

\begin{proposition}\label{prop6.1}
    Let parameters $\model\in\lm$, suppose that \Cspecbis$[\model]$ and \Cexptail$[\model]$ hold. Then there exist constants $K,c>0$ such that 
    \[ \coef{p,k}{p',k'}{\model} \leq K e^{-c (k-k')},  \quad \forall (p,k),(p',k')\in\Q \ \text{with}  \ k'\leq k,\]
    \begin{flalign*}
    	\text{and} \hspace{0.28\linewidth} \mathfrak{h}^m_{m'}(s,t) \leq K e^{-c(t-s)}, \quad \forall m,m'\in\M \ \text{and} \ s<t. &&
    \end{flalign*}
\end{proposition}

The corner stone of this section is having precise control on the tail of clusters. Assumption \Cexptail$[\model]$ on the tails of the interaction functions leads to the following lemma, which is a specific case of Theorem 4.1 from \cite{leblanc_sharp_2025}. As for point processes, for any cluster $G_s^m$ born at time $s$ with a root of type $m$, for any $I\subset \R$, we denote by $G^m_s(I)$ the total number of individuals of $G^m_s$ whose birth dates are in $I$.

\begin{lemma}\label{lem6.2}
    Let $\model\in\lm$ and suppose that assumptions \Cexptail$[\model]$ and \Cspecter$[\model]$ hold. Let $G^m_{t_0}$ be a $\Pois(\bm{\mathfrak{h}}^{\intercal})$-cluster with root of type $m\in\M$ and born at time $t_0\in\R$. Then, there exists $r > 0$, depending on $\model$, such that, for all $t\geq 0$,
    \[\E{e^{r G^m_{t_0}([t_0+t,\infty))}} \leq 1 + K \exp(-ct),\]
    for constants $c>0, \ K>0$ depending on $\model$ but not on $m$ nor $t_0$.
\end{lemma}

Given two initial conditions $\mathfrak{C}_1$ and $\mathfrak{C}_2$ at time $t_0$, and $\delta >0$, the quantity $\Massiexp_{\delta,t_0}(\mathfrak{C}_1 \circleddash \mathfrak{C}_2)$, introduced in \eqref{notation mass exp}, gives a measure of how different $\mathfrak{C}_1$ and $\mathfrak{C}_2$ are with a memory vanishing exponentially fast at rate $\delta$.

The following theorem is the main result of this section.

\begin{theorem}[Asymptotic limit]\label{th6.3}
    Let $\model\in\Model$, random drifts $\xi$ and suppose that assumptions \Cnoexpl$[\model]$, \Cspecbis$[\model]$, \Cexptail$[\model]$ and \Cspecter$[\model]$ hold. Let $t_0\in\R$ and two integrable initial conditions, $\mathfrak{C}_1$ and $\mathfrak{C}_2$, at time $t_0$ for $\model$. For $i=1,2$, we denote by $X^{t_0,\mathfrak{C}_i}$ the unique non-explosive $\Equ{\model,\xi}{t_0,\mathfrak{C}_i}$ process from Theorem \ref{th3.4}. There exists $\delta>0$ depending only on $\model$ such that the following holds.
    \allowdisplaybreaks
    \begin{itemize}
        \item There exist $K,c>0$ depending only on $\model$ such that for all $t\geq 0$,
        \[\P{\Mass_{[t_0+t,\infty)}\Big(X^{t_0,\mathfrak{C}_1} \circleddash X^{t_0,\mathfrak{C}_2}\Big)\leq e^{-ct}} \geq 1-K \E{\Massiexp_{\delta,t_0}(\mathfrak{C}_1 \circleddash \mathfrak{C}_2)}e^{-ct}.\]
        \item The following convergences hold a.s. and in $\Lrv^q$ for $q\geq 1$ such that $\Mass_{\exp(\delta)}(\mathfrak{C}_1 \circleddash \mathfrak{C}_2) \in \Lrv^{q_0}$ for some $q_0>q$;
        \[\Mass_{[t_0+t,\infty)}\Big(X^{t_0,\mathfrak{C}_1} \circleddash X^{t_0,\mathfrak{C}_2}\Big) \xrightarrow[t\to\infty]{} 0.\]
        Convergence rates are at least $Ke^{-ct}$ with $K,c>0$ depending on $\model$, $q$, $q_0$ and $\E{ \Massiexp_{\delta,t_0}(\mathfrak{C}_1 \circleddash \mathfrak{C}_2)^{q_0}}$ for $K$.
    \end{itemize}
\end{theorem}

We obtained convergence, almost surely, in $L^q$ and in probability, of $\Equ{\model}{t_0,\mathfrak{C}_i}$, $i=1,2$ towards each other when time goes to infinity with explicit convergence rates. This shows that the asymptotic behavior does not depend on the initial condition, and thus HAR processes are decorrelated at infinity, which means, at least informally, that correlations between $X\big\vert_{(-\infty,s]}$ and $X\big\vert_{[t,\infty)}$ vanish as $t-s \to \infty$.

One can specify the result by taking for the first initial condition the empty initial condition and for the second one, the initial condition which is equal to the stationary HAR process. This allows to recreate the stationary HAR process through a HAR process with this particular initial condition. These two clever choices lead to the following result. 

\begin{corollary}[Convergence to the 1-stationary HAR process]\label{cor6.4}
    Let $\model\in\Model$, random drifts $\xi$, and suppose that assumptions \Cspec$[\model]$, \Cfime$[\xi]$, \Cexptail$[\model]$ and \Cspecter$[\model^+]$ hold. Let $t_0\in\R$. We denote by $X^{t_0}$ the $\Equ{\model,\xi}{t_0}$ process from Theorem \ref{th3.4} and $X^{-\infty}$ the $\Equ{\model,\xi}{-\infty}$ process from Theorem \ref{th3.9}. Then the following holds.
    \begin{itemize}
        \item There exist $K,c>0$ depending on $\model$ and \Cfime$[\xi]$ such that for all $t\geq 0$, 
        \[\P{\Mass_{[t_0+t,\infty)}\Big(X^{t_0} \circleddash X^{-\infty}\Big)\leq e^{-ct}} \geq 1-Ke^{-ct},\]
        \item The following convergences hold a.s. and, under assumption \Csubg$[\xi]$, in $\Lrv^q$ for all $q\geq 1$
        \[\Mass_{[t_0+t,\infty)}\Big(X^{t_0} \circleddash X^{-\infty}\Big) \xrightarrow[t\to\infty]{} 0.\]
        Convergence rates are at least $Ke^{-ct}$ with $K,c>0$ depending on $\model$, $q$ and \Csubg$[\xi]$.
    \end{itemize}
\end{corollary}

This result links two of the most natural HAR processes: the stationary process and the process started at a fixed time $t_0$. They converge towards each other with exponential rates.

\subsection{Ergodic theorem for HAR processes}

In this final section, we investigate the asymptotic behavior of functionals of stationary HAR processes. More precisely, we establish an ergodic theorem that connects time averages along a trajectory to expectations under the stationary distribution.

Recall notation \eqref{notation mass exp}. Let $\f$ be a functional on $(\M,\Pg)$-processes such that there exist $\delta, C_{\f},a >0$ such that for any $X,\tilde{X}$ we have
\begin{equation}\label{fonctionnelle}
	\begin{aligned}
		& \vert \f(X) \vert \leq C_{\f} \Big( 1 + \Mass^{\exp}_{\delta,0}\big(X\big)\Big)^{a},\\
		& \vert \f(X) - \f(\tilde{X}) \vert \leq C_{\f} \Big( 1 + \Mass^{\exp}_{\delta,0}(X) + \Mass^{\exp}_{\delta,0}(\tilde{X})\Big)^{a} \Big(\Mass^{\exp}_{\delta,0}\big( X \circleddash \tilde{X}\big)\Big)^{a}.
	\end{aligned}
\end{equation}

For any $t\in\R$ the shift by $t$ operator, denoted $\theta_t$, on $(\M,\Pg)$-processes $X=(N,W)$ is defined by $\theta_t N^m = N^m+t$ and $\theta_t W^p_k = W^p_{k+t}$ for any $p\in\Pg$, $k\in\D_p-t$.

\begin{theorem}[Ergodic Theorem for HAR processes]\label{th ergo}
	Let $\model\in\Model$, random drifts $\xi$, and assume \Cspec$[\model]$, \Cspecter$[\model]$, \Csubg$[\xi]$ and \Cexptail$[\model]$. Let $X$ the 1-stationary $\Equ{\model,\xi}{-\infty}$ process. The following holds:
	\begin{equation}
		\frac{1}{T}\itg_0^T \f(\theta_t X) dt \xrightarrow[T\to\infty]{a.s.} \E{\itg_0^{1} \f(\theta_t X) dt}.
	\end{equation}
	More precisely, for any $n,\tau,r$ such that $T=4n\tau + r$, $r<4\tau$, for any $q > 2$ and $x>0$, there exists constants $c,C,C_q >0$ depending on $\model$, \Csubg$[\xi]$, $\delta,C_{\f},a$ and also $q$ for $C_q$ such that
	\begin{equation}\label{non asymp th ergo}
		\P{ \Big\vert\frac{1}{T}\itg_0^T \f(\theta_t X) dt - \itg_0^{1} \E{\f(\theta_t X)} dt\Big\vert > x+C e^{-c \tau}} \leq C Te^{-c \tau} + \frac{C_q}{n^{q-1}}\big( \frac{1}{x^q} + \frac{1}{x^{2(q-1)}}\big).
	\end{equation}
\end{theorem}

\begin{remark}
This theorem provides a quantitative version of the ergodic theorem for HAR processes under fairly general assumptions. In addition to almost sure convergence, it also provides explicit non-asymptotic deviation bounds of time averages toward their stationary expectation. Theorem \ref{th6.3} proves that memory of HAR processes fades exponentially fast. To make use of this "weak independence" in the proof of \eqref{non asymp th ergo}, we construct a coupling \textit{à la Berbee} \cite{berbee_random_1979} and use techniques similar to what have been used for stationary time series analysis, see \cite{viennet_inequalities_1997,rio_functional_1995,doukhan_invariance_1995}. Equation \eqref{non asymp th ergo} highlights a trade-off between a block length $\tau$ and the corresponding number of blocks $n$, reflecting the interplay between vanishing dependence (controlled via $\tau$ and Corollary \ref{cor6.4}) and concentration (through $n$). Such bounds are especially useful for statistical applications, where conditions \eqref{fonctionnelle} are usually satisfied.
\end{remark}

\begin{remark}
	Ergodicity for Hawkes processes has been extensively studied in the literature. Clinet and Yoshida \cite{clinet_statistical_2017} established ergodicity for the exponential Hawkes processes (thus linear), which is also the framework of \cite{abergel_long-time_2015}. In \cite{kwan_ergodic_2025}, the assumptions are minimal: the authors prove ergodicity of the intensity of a univariate linear subcritical Hawkes process without additional conditions. Hansen et al. \cite{hansen_lasso_2015} obtained a non-asymptotic ergodic theorem for linear multivariate Hawkes processes in the case of interactions and functionals with finite memory. Our result is also non-asymptotic and extends both the exponential and finite-memory settings to the broader class of non-linear HAR processes with sub-exponential interaction functions. In contrast to \cite{hansen_lasso_2015} we allow the functional to depend on the whole process, provided that its memory decays exponentially fast at $\pm\infty$. In particular, this result applies to the stochastic intensity which can be viewed as a functional satisfying conditions \eqref{fonctionnelle} under assumption \Cexptail$[\model]$. It yields an ergodic theorem for any function $f(\lambda_t^1, \dots, \lambda_t^M)$ such that, for all $x,y \in \R^{\M}$, $\vert f(x)\vert \leq C (1+\vert x\vert_1)^{a}$ and $\vert f(x)-f(y)\vert \leq C (1+\vert x\vert_1+\vert y\vert_1)^{a} \vert x-y\vert^{a}$.
\end{remark}

\section{Proof of results of Section \ref{sec3}}\label{sec6}

\subsection{Proof of Theorem \ref{th3.4}}

Before diving into the proof, let us make some preliminary calculations on the integrable initial condition (see Definition \ref{def3.2}). Let $\mathfrak{C}$ an integrable initial condition, $m\in\M$ and $f\in \Lf^1$ non-negative. Recall that $\mathfrak{C}^m = \mathfrak{C}^m_{reg}\cup \mathfrak{C}^m_{disc}$ with $\card(\mathfrak{C}^m_{disc}) < \infty$ a.s. and that $\mathfrak{C}^m_{reg}$ has a mean intensity measure dominated by $K$ times the Lebesgue measure. Let $t\in\R$, we thus have $\int_{-\infty}^{t-} f(t-s)d(\mathfrak{C}^m_{disc})_s < \infty$ a.s. and $\E{\int_{-\infty}^{t-} f(t-s)d(\mathfrak{C}^m_{reg})_s} \leq K \Vert f\Vert_1 < \infty$. Therefore we have
\begin{equation}\label{CI fini 1}
	\forall f\in \Lf^1, \ \forall t\in \R, \ a.s. \quad \itg_{-\infty}^{t-} f(t-s)d\mathfrak{C}^m_s <\infty.
\end{equation}
Similarly, again for $f\in \Lf^1$ non-negative and for $a<b$ we have by Fubini $\int_a^b\int_{-\infty}^{t-} f(t-s)d(\mathfrak{C}^m_{disc})_s \leq \Vert f\Vert_1 \card(\mathfrak{C}^m_{disc})< \infty$ a.s. and $\E{\int_a^b\int_{-\infty}^{t-} f(t-s)d(\mathfrak{C}^m_{reg})_s dt} \leq (b-a)K \Vert f\Vert_1 < \infty$. Thus we have
\begin{equation}\label{CI fini 2}
	\forall f\in \Lf^1, \ \forall a<b\in \R, \ a.s. \quad \int_a^b\itg_{-\infty}^{t-} f(t-s)d\mathfrak{C}^m_s <\infty.
\end{equation}
Let $p\in\Pg$ and recall that $\mathfrak{C}^p = \mathfrak{C}^p_{reg}+\mathfrak{C}^p_{disc}$ with $\Massi_{(-\infty,t_0)}(\mathfrak{C}^p_{disc}) < \infty$ a.s., which means that a.s. only a finite number values of $\mathfrak{C}^p_{disc}$ are different from $\varnothing$. Also $\mathfrak{C}^p_{reg}$ satisfies $\E{\Massi_{B}(\mathfrak{C}^p_{reg})} \leq K \vert B \vert_{\Pg}$ for all Borel sets $B$, meaning that $\mathfrak{C}^p_{reg,k}$ is uniformly (in $k$) bounded in $\Lrv^1$ since $\vert B \vert_{\Pg} = \sum_{p\in\Pg} \mp{p}{}(B)$. Let $t\in\R$ and $f\in \Lf^1$ non-negative. Recall that $\psi_1$ is defined in \eqref{def psi}. By Fubini we have, $\int_a^b\int_{-\infty}^{t-} f(t-k)\psi_1(\mathfrak{C}^p_{disc,k}) d\mp{p}{k}dt \leq \Vert f\Vert_1 \Massi_{(-\infty,t_0)}(\mathfrak{C}^p_{disc})< \infty$ a.s., and $\E{\int_a^b\int_{-\infty}^{t-} f(t-k)\psi_1(\mathfrak{C}^p_{reg,k})d\mp{p}{k}dt} \leq [n_p(b-a)+1] K \Vert f\Vert_1 < \infty$ with $n_p$ the frequency of $\D_p$. Thus we have,
\begin{equation}\label{CI fini 3}
	\forall f\in \Lf^1, \ \forall a<b\in \R, \ a.s. \quad \int_a^b\itg_{-\infty}^{t-} f(t-k)\psi_1(\mathfrak{C}^p_k) d\mp{p}{k} <\infty.
\end{equation}
Finally, for $f\in \Ld{n_p}$ non-negative and $\ell\in\R$, we have, $\int_{-\infty}^{\ell-}f(\ell-k)\psi_1(\mathfrak{C}^p_{disc,k}) d\mp{p}{k} \leq \Massi_{(-\infty,t_0)}(\mathfrak{C}^p_{disc}) \Vert f\VLd{n_p} < \infty$ a.s., and $\E{\int_{-\infty}^{\ell-}f(\ell-k)\psi_1(\mathfrak{C}^p_{reg,k}) d\mp{p}{k}} \leq K \Vert f\VLd{n_p} <\infty$. Thus we have,
\begin{equation}\label{CI fini 4}
	\forall f\in \Ld{n_p}, \ \forall \ell\in \R, \ a.s. \quad \itg_{-\infty}^{\ell-} f(\ell-k)\psi_1(\mathfrak{C}^p_k) d\mp{p}{k} <\infty.
\end{equation}

We can now start the proof. On $(-\infty,t_0)$ the process $X$ is given by the initial condition. For $t=t_0$ we have $N^m(\{t_0\})=0$ almost surely since $\pi^m(\{t_0\}\times \R_+) = \varnothing$ almost surely. For $p$ such that $t_0\in\D_p$, values $W^p_{t_0}$. By \Cnoexpl$[\model]$ we have $h_m^{p} \in\Lf^1$ and $b_{p'}^{p},h^{p}_{p'} \in\Ld{n_{p'}}$ for $m\in\M,p,p'\in\Pg$, thus equations \eqref{CI fini 1} and \eqref{CI fini 4} combined to $\vert J^{p}_{p'}(s,w) \vert \leq \max(h^{p}_{p'},b^{p}_{p'})(s)\psi_1(w)$, ensure that $\vert W^p_{t_0} \vert < \infty$ almost surely.

Assume now that the process $X$ is constructed on $[t_0,s_0]$ for $s_0 \geq t_0$ and that almost surely $\Mass_{[t_0,s_0]}(X) <\infty$, which is the case for $s_0=t_0$. Let $t^*>0$ small enough (the precise condition is given at the end). Let 
\begin{equation}
	R(s_0)= \min\big(s_0+t^*, \inf\{k>s_0 \mid \exists p\in\Pg, \ k\in \D_p\}\big).
\end{equation}
We will prove that we can construct $X$ on $[t_0,R(s_0)]$ and that $\Mass_{[t_0,R(s_0)]}(X) <\infty$ almost surely. First one can remark that on time window $(s_0,R(s_0)]$ there is no values of any $W^p$, except possibly at $R(s_0)$. Thus we need to construct $N$ on $(s_0,R(s_0)]$ and then, if needed, to calculate values $W^p_{R(s_0)}$ for which $R(s_0)\in\D^p$. 

To do so we define \textit{temporary intensities} at level $s_0$ for $s_0 < t \leq R(s_0)$ defined by
\[ \temp_{s_0}(\lambda^m)_t = \Phi^m\left(\sum_{m'\in\M} \itg_{-\infty}^{s_0} h^{m}_{m'}(t-s)dN^{m'}_s + \sum_{p\in\Pg} \itg_{-\infty}^{s_0} \J^m_p(t-k,W^p_k)d\mp{p}{k}\right).\]
Basically, $\temp_{s_0}(\lambda^m)_t$ is the intensity of $N^m$ at time $t$ if there are no points (from any node) in $(s_0,t)$. Then one have access to the \textit{temporary points at level $s_0$}, being the atoms $(t,x)$ of the Poisson random measures $\pi^m$ such that $s_0< t \leq R(s_0)$ and $0\leq x\leq \temp_{s_0}(\lambda^m)_t$.\\
Considering that $\vert J^m_p(s,w)\vert \leq \max(h^{m}_{p},b^m_{p})(s)\psi_1(w)$, that the functions $h^{m}_{p},b^m_{p},h^m_{m'}$ are $\Lf^1$ and function $\Phi^m$ is Lipschitz by \Cnoexpl$[\model]$, that $\Mass_{[t_0,s_0]}(X) <\infty$ almost surely, and by equations \eqref{CI fini 2} and \eqref{CI fini 3}, we have
\begin{equation}
	\itg_{s_0}^{R(s_0)} \temp_{s_0}(\lambda^m)_t dt < \infty \quad \F_{s_0} \text{ - almost surely.}
\end{equation}
We denote by $E_0$ the set of all temporary points at level $s_0$. We have
\begin{equation}
	\E{ \card(E_0) | \F_{s_0}} = \sum_{m\in\M} \itg_{s_0}^{R(s_0)} \temp_{s_0}(\lambda^m)_t dt,
\end{equation}
and so, $\F_{s_0}$-almost surely, $\card(E_0) <\infty$. If $E_0 = \varnothing$ then $N$ have no points in $(s_0,R(s_0)]$ and $N$ is constructed on this interval. Else, one can find the element of $E_0$ with minimal time coordinate. We denote by $s_1$ this time coordinate. $s_1$ is a true point of its corresponding $N^m$ and $N$ has no other points in $(s_0,s_1]$. We can start again, calculate temporary intensities at level $s_1$ for $s_1< t \leq R(s_0)$ and find $E_1$ the set all temporary points at level $s_1$. The same argument shows that $E_1$ is almost surely finite. If $E_1$ is empty we stop and $N$ is constructed on $(s_1,R(s_0)]$, else we can find $s_2$ the smallest time coordinate in $E_1$, then find $E_2$ etc until at some point $E_k$ is empty for some $k<\infty$. To prove it will happen almost surely, we can look at the sequence $U_k = \card(E_k\setminus(E_0\cup \cdots \cup E_{k-1}))$ for $k\geq 1$. Conditionally on $\F_{s_{k-1}}$, $U_k$ is distributed as $\Pois(\alpha_k)$ where $\alpha_k$, measurable for $\F_{s_{k-1}}$, is given by (see Figure \ref{fig4}),
\begin{align*}
	\alpha_k & = \sum_{m\in\M}\itg_{s_k}^{R(s_0)} \left(\temp_{s_{k}}(\lambda^m)_t-\max_{i=0,\cdots,k-1} \temp_{s_{i}}(\lambda^m)_t\right)_+ dt \\
	& \leq \sum_{m\in\M}\itg_{s_k}^{R(s_0)} \left(\temp_{s_{k}}(\lambda^m)_t- \temp_{s_{k-1}}(\lambda^m)_t\right)_+ dt \\
	& \leq \max_{m'\in\M} \sum_{m\in\M} \itg_{s_k}^{R(s_0)} L_m \vert h^m_{m'}(t-s_k)\vert dt\\
	& \leq  \max_{m'\in\M} \sum_{m\in\M} L_m \Vert h^m_{m'}\Vert_{1,[0,t^*]}.
\end{align*}
Thus if $t^*$ small enough, there exists a constant $\alpha^*<1$ such that $\alpha_k \leq \alpha^*$ for all $k$. Thus $\card\Big(\bigcup_{k\geq 0} E_k\Big)$ is dominated by the cardinal of a $E_0$-forest of subcritical Galton Watson processes (with offspring distribution $\Pois(\alpha^*)$), and since $E_0$ is almost surely finite, it is clear that $\bigcup_{k\geq 0} E_k$ is almost surely finite and we will stop. It follows that $N$ is constructed on $(s_0,R(s_0)]$ and that the number of points of $N$ in $(s_0,R(s_0)]$ is almost surely finite since it is included in $\bigcup_{k\geq 0} E_k$.

\begin{figure}[h]
\begin{center}
\begin{tikzpicture}[scale=0.6]

\draw[-stealth, line width=0.6mm] (0,0)--(0,6);
\draw[-stealth, line width=0.6mm] (0,0)--(10,0);

\draw (10.2,-0.5) node {$t$};
\draw (-0.5,6.3) node {$\R_+$};
\draw (-0.2,-0.2) node {$t_0$};

\draw[red, line width=0.4mm] (0,4) ..controls +(1.5,-0.5) and +(-1,-0.5).. (2,1) ..controls +(1,0.5) and +(-1,1.5).. (4,3);

\draw[red, line width=0.4mm, dotted] (4,3) ..controls +(0.55,-0.7) and +(-0.2,0.2).. (4.5,2.4) ..controls +(0.7,-0.7) and +(-1,-0.2).. (6,2.8) ..controls +(1,0.2) and +(-0.6,-0.3).. (7.8,2.9) ..controls +(0.4,0.2) and +(-0.3,-0.5).. (8.5,3.9);

\draw[blue, line width=0.4mm, dotted] (4.5,2.4) ..controls +(0.5,-0.3) and +(-1,-0.7).. (6,3.4) ..controls +(1,0.7) and +(-0.6,0.45).. (7.8,2.9) ..controls +(0.4,-0.3) and +(-0.4,-0.3).. (8.5,3);

\begin{scope}[transparency group, opacity=0.13]
\fill (4.5,2.4) ..controls +(0.5,-0.3) and +(-1,-0.7).. (6,3.4) ..controls +(1,0.7) and +(-0.6,0.45).. (7.8,2.9) ..controls +(-0.6,-0.3) and +(1,0.2).. (6,2.8) ..controls +(-1,-0.2) and +(0.7,-0.7) .. (4.5,2.4);
\end{scope}

\draw[dotted] (4,0)--(4,3);
\draw[dotted] (8.5,0)--(8.5,3.9);
\draw[dotted] (4.5,0)--(4.5,2.4);
\draw (3.9,-0.5) node {$s_0$};
\draw (4.8,-0.5) node {$s_1$};
\draw (8.5,-0.5) node {$R(s_0)$};

\draw[red] (2,4) node {$\lambda^m_t$};
\draw[red] (8,2) node {$\temp_{s_0}(\lambda^m)_t$};
\draw[blue] (6,4.5) node {$\temp_{s_1}(\lambda^m)_t$};

\tikzstyle{every node}=[circle, draw, fill=green!100!black!30, inner sep=0pt, minimum width=4pt]

\draw (4.5,0.8) node {};
\draw (5.2,2.1) node {};
\draw (5.6,1.95) node {};
\draw (6.7,0.5) node {};
\draw (8.2,3.1) node {};

\tikzstyle{every node}=[circle, draw, fill=RoyalBlue!80, inner sep=0pt, minimum width=4pt]

\draw (7.1,3.3) node {};
    
\end{tikzpicture}
\end{center}
\caption{This figure illustrates the proof in $\pi^m$. In red there is $\lambda^m_t$, dotted in red there is the temporary intensity $\temp_{s_0}(\lambda^m)$. Green dots are the elements of $E_0$ in $\pi^m$. $s_1$ is the first point of $E_0$ (assumed to be also on $\pi^m$ here). Dotted in blue there is the temporary intensity at level $s_1$. The shaded area is where there might be elements of $(E_1\setminus E_0)\cap\pi^m$, here there is one represented by the blue dot.}
\label{fig4}
\end{figure}

It remains to calculate, if $R(s_0)\in\D_p$, values $W^p_{R(s_0)}$. By combining assumption \Cnoexpl$[\model]$ with $\vert J^p_{p'}(s,w)\vert \leq \max(h^{p}_{p'},b^p_{p'})(s)\psi_1(w)$, that $\Mass_{[t_0,s_0]}(X) <\infty$ almost surely, that almost surely $N((s_0,R(s_0)])$ is finite,  and equations \eqref{CI fini 1} and \eqref{CI fini 4}, we have almost surely,
\begin{equation}
	\vert W^p_{R(s_0)} \vert \leq \vert \xi^p_{R(s_0)}\vert + \sum_{m\in\M} \itg_{-\infty}^{R(s_0)-} L_p \vert h^p_m\vert (R(s_0)-s)dN^{m}_s + \sum_{p'\in\Pg} \itg_{-\infty}^{s_0} L_p\vert \J^p_{p'}\vert(R(s_0)-k',W^{p'}_{k'})d\mp{p'}{k'} < \infty.
\end{equation}
Thus $X$ is constructed on $[t_0,R(s_0)]$ and $\Mass_{[t_0,R(s_0)]}(X) < \infty$ almost surely. By induction we can construct the process on $[t_0,\infty)$ since $R^{\circ n}(t_0) \xrightarrow[n\to\infty]{} \infty$.

A quick word about uniqueness. If there is another non-explosive $\Equ{\model,\xi}{t_0,\mathfrak{C}}$ process then it must be the same on $(-\infty,t_0)$ as the one we have just constructed since the initial conditions are the same. Then they must also be the same at $t_0$ by applying the equations at time $t_0$. Then since there is a finite number of points until $R(t_0)$, the temporary intensities must match until the first point after $t_0$, which then has to be the same for both of the processes. By doing so, one can check that this other solution must equal as the one we have just constructed on $[t_0,R(t_0)]$ and so on.

\subsection{Proof of Theorem \ref{th3.9}}

To construct 1-stationary processes we use a Picard iteration whose scheme is defined as follows.
\begin{flalign*}
\text{\underline{Initialisation :}}  \hspace{0.1\linewidth} 
\left\{
\begin{aligned}
    & \lambda^{m,0}_t = \Phi^m(0) \\
    & N^{m,0}(C) = \itg_C \itg_{\R_+} \ind_{x \leq \lambda^{m,0}_t} d\pi^m(t,x) \\
    & W^{p,0}_{k} = \xi^p_k. \\
\end{aligned}
\right.
&&
\end{flalign*}
\underline{Iteration $n\to n+1$ :}
\begin{equation}\label{eq iteration}
\left\{
\begin{aligned}
    & \lambda^{m,n+1}_t = \Phi^m\bigg(\sum_{m'\in\M} \itg_{-\infty}^{t-} h^m_{m'}(t-s) dN^{m',n}_s + \sum_{p\in\Pg}\itg_{-\infty}^{t-} \J^m_{p}(t-k,W^{p,n}_{k})d\mp{p}{k}\bigg) \\
    & N^{m,n+1}(C) = \itg_C \itg_{\R_+} \ind_{x \leq \lambda^{m,n+1}_t} d\pi^m(t,x) \\
    & W^{p,n+1}_{k} = \xi^p_k + \Phi^p\bigg(\sum_{m\in\M} \itg_{-\infty}^{k-} h^{p}_{m}(k-s)dN^{m,n}_s + \sum_{p'\in\Pg}\itg_{-\infty}^{k-} \J^{p}_{p'}(k-k',W^{p',n}_{k'})d\mp{p'}{k'}\bigg).
\end{aligned}
\right.
\end{equation}

Integrals appearing in the iteration may not converge, we will show that almost surely it does not happen. 

The proof can be divided into 4 steps. 1) Checking that the Picard iteration is well defined, 2) The Picard iteration converges almost surely and in $\Lrv^1$, 3) The limit is a solution of HAR equations, 4) Showing uniqueness.
\begin{flalign*}
\text{Let for }n\geq 0, \hspace{0.2\linewidth} \Lambda^m_n &= \sup_{t\in\R} \E{|\lambda^{m,n+1}_t - \lambda^{m,n}_t|}, &&\\
\Gamma^p_n &= \sup_{k\in\D_p} \E{|W^{p,n+1}_{k}-W^{p,n}_{k}|},\\
\Psi_n & = \max\Big( \sup_{m\in\M, t\in\R} \E{\vert \lambda^{m,n}_t\vert}, \sup_{(p,k)\in\Q} \E{|W^{p,n}_{k}|}\Big). 
\end{flalign*}
\textbf{Step 1:} From assumption \Cfime$[\xi]$ and the fact that $\lambda^{m,0}_t = \Phi^m(0)$ is deterministic we have $\Psi_0 <\infty$. By \Cspec$[\model]$, we have, $\Vert h^{\alpha}_m\Vert_1, \Vert h^{\alpha}_p \VLd{n_p}, \Vert b^{\alpha}_p \VLd{n_p} <\infty$ for $m\in\M$, $p\in\Pg$ and $\alpha\in\M\cup\Pg$. Thus, for $t\in\R$ we have
\begin{equation}\label{eq step 1}
    \E{\itg_{-\infty}^{t-} \vert h^m_{m'}(t-s)\vert dN^{m',0}_s}  = \E{\itg_{-\infty}^{t-} \vert h^m_{m'}(t-s)\vert \lambda^{m',0}_s ds} \leq  \Psi_0 \itg_{-\infty}^{t-} \vert h^m_{m'}(t-s)\vert ds <\infty.
\end{equation}
Thus, almost surely $\int_{-\infty}^{t-} h^m_{m'}(t-s) dS^{m',0}_s$ is well defined. Similarly, since $\vert J^{\alpha}_p(s,w)\vert \leq h^{\alpha}_p(s)\vert w\vert + b_p^{\alpha}(s)$, it is the case for any other integral involved in \eqref{eq iteration} defining $\lambda^{m,1}_t$ or $W^{p,1}_k$. Thus, $W^{p,1}_k$ is well defined almost surely and $\lambda^{m,1}$ is predictable and well defined $d\P{\omega}\otimes dt$ almost everywhere, which implies that $N^{m,1}$ is well defined almost surely. Finally, calculations similar to \eqref{eq step 1} leads to $\Psi_1 \leq K(1+\Psi_0)<\infty$ with $K$ a constant depending on $\model$. By induction, the same reasoning ensures that $\Psi_n<\infty, \ n\in\N$. The Picard iteration is well defined.

\textbf{Step 2:} The goal of this step is to show that
\begin{itemize}
	\item $\forall t \in\R, \ \exists \Omega_t$ almost sure such that $\forall m\in\M, \ \lambda^{m,n}_t \ \xrightarrow[n\to\infty]{\Omega_t - a.s. \ \text{and} \ \Lrv^1} \lambda^m_t\in\R$,
	\item $\exists \tilde{\Omega}$ almost sure such that $\forall (p,k)\in\Q, \ \  W^{p,n}_{k} \ \xrightarrow[n\to\infty]{\tilde{\Omega}-a.s. \ \text{and} \ \Lrv^1} \  W^p_{k}\in\R$,
	\item $\exists \tilde{\Omega}$ almost sure $\exists N^m\subset\R$ such that $\forall m, \ \forall C\subset\R \ \text{bounded}, \ \tilde{\Omega}-a.s., \ N^{m,n}\cap C = N^m\cap C$ for $n$ large enough,
	\item and finally, $\sup_{m\in \M, \ t\in\R} \E{ \vert \lambda^m_t \vert} < \infty \ \text{and} \ \sup_{(p,k)\in\Q} \E{ \vert W^p_{k} \vert} < \infty$.
\end{itemize}

Let $t\in\R$ and $n\in\N^*$, since $\Phi^m$ is $L_m$ Lipschitz,
\begin{align*}
    \vert \lambda^{m,n+1}_t - \lambda^{m,n}_t \vert & \leq  L_m \bigg\vert\sum_{m'} \itg_{-\infty}^{t-} h^m_{m'}(t-s) \big(dN^{m',n}_s-dN^{m',n-1}_s\big) \\
    & \hspace{1.5cm} + \sum_{p\in\Pg}\itg_{-\infty}^{t-} \J^m_{p}(t-k,W^{p,n}_{k})-\J^m_{p}(t-k,W^{p,n-1}_{k})\big)d\mp{p}{k}\bigg\vert.
\end{align*}
\begin{flalign*}
    \text{It follows that,} \quad \E{|\lambda^{m,n+1}_t - \lambda^{m,n}_t|} & \leq \sum_{m'} \itg_{-\infty}^{t} L_m|h^m_{m'}(t-s)| \E{|\lambda^{m',n}_s - \lambda^{m',n-1}_s|}ds && \\
    & \hspace{1.5cm} + \sum_{p\in\Pg}\itg_{-\infty}^{t-} L_m \vert h^m_{p}(t-k)\vert \E{\vert W^{p,n}_{k}-W^{p,n-1}_{k}\vert}d\mp{p}{k}\\
    & \leq \sum_{m'} L_{m}\Vert h^m_{m'}\Vert_1 \Lambda^{m'}_{n-1} + \sum_{p\in\Pg} L_m \Vert h^m_{p}\VLd{n_p} \Gamma^p_{n-1}.
\end{flalign*}
\begin{flalign*}
	\text{Thus we obtain,} \hspace{0.1\linewidth} \Lambda^m_{n} \leq \sum_{m'} L_{m}\Vert h^m_{m'}\Vert_1 \Lambda^{m'}_{n-1} + \sum_{p\in\Pg} L_m \Vert h^m_{p}\VLd{n_p} \Gamma^p_{n-1}. &&
\end{flalign*}
\begin{flalign*}
	\text{Similarly, one can show that,} \hspace{0.1\linewidth} \Gamma^p_{n} \leq \sum_{m} L_{p}\Vert h^p_{m}\Vert_1 \Lambda^{m}_{n-1} + \sum_{p'\in\Pg} L_p\Vert h^p_{p'}\VLd{n_{p'}} \Gamma^{p'}_{n-1}. &&
\end{flalign*}
Let $\Lambda_n = (\Lambda^m_n)_{m\in\M}$ and similarly $\Gamma_n = (\Gamma^p_n)_{p\in\Pg}$. The two previous equations rewrites $\begin{pmatrix} \Lambda_{n} \\ \Gamma_{n} \end{pmatrix} \leq \Hmod{\model} \begin{pmatrix} \Lambda_{n-1} \\ \Gamma_{n-1} \end{pmatrix}$ with $\Hmod{\model}$ defined in Definition \ref{def3.8}. An induction gives immediately
\begin{equation*}
	\begin{pmatrix} \Lambda_{n} \\ \Gamma_{n} \end{pmatrix} \leq (\Hmod{\model})^n \begin{pmatrix} \Lambda_{0} \\ \Gamma_{0} \end{pmatrix}.
\end{equation*}
Since we have $\Lambda^m_0, \ \Gamma^p_0 \leq \Psi_0+\Psi_1 <\infty$, it follows from assumption \Cspec$[\model]$ that there exist constants $\spr(\Hmod{\model}) \leq r<1$ and $K>0$ such that, for all $m\in\M$, $p\in\Pg$,
\begin{equation}\label{eq6.9}
    \forall n\in\N, \ \Lambda^m_n, \ \Gamma^p_n \leq Kr^n.
\end{equation}
Since the $0\leq r<1$, the series $\sum_n \Lambda^m_n$ converges, and so, $\sum_n \E{|\lambda^{m,n+1}_t - \lambda^{m,n}_t|} < \infty$, which is the same as
\[ \E{ \sum_n |\lambda^{m,n+1}_t - \lambda^{m,n}_t|} <\infty.\]
Therefore, almost surely, $\sum_n |\lambda^{m,n+1}_t - \lambda^{m,n}_t| <\infty$ which implies that for all $t\in\R$, $\lambda^{m,n}_t \ \xrightarrow[n\to\infty]{a.s.} \lambda^m_t \in \R$.\\
The same reasoning gives, $\forall (p,k)\in\Q, \  W^{p,n}_{k} \ \xrightarrow[n\to\infty]{a.s.} \  W^{p}_{k} \in \R$. Since $\Q$ is countable, the almost sure event of convergence can be taken uniformly over $\Q$. 

Since $\Psi_0 <\infty$, by \eqref{eq6.9}, $\Lrv^1$ bounds and $\Lrv^1$ convergences are also clear.

Let $j\in\N^*$ and $m\in\M$, then we have
\begin{equation*}
    \E{\card\big((N^{m,n}\vartriangle N^{m,n+1})\cap[-j,j]\big)} = \itg_{[-j,j]} \E{\vert \lambda^{m,n}_s - \lambda^{m,n+1}_s\vert} ds \leq 2 j \Lambda^m_n. 
\end{equation*}
Since the cardinal only takes non-negative integer values and because $\sum_n 2 j \Lambda^m_n <\infty$, by Borel-Cantelli Lemma, there exists $\Omega_j$ almost sure and $n_0=n_0(j,\omega)$ such that for all $n\in\N$, for all $m\in\M$,
\[ N^{m,n_0+n}\cap[-j,j] = N^{m,n_0}\cap[-j,j], \quad \text{on} \ \Omega_j.\]
Let $\tilde{\Omega} = \cap_{j}\Omega_j$, almost sure since $\N^*$ is countable. Then, we can define $N^m$ on $\tilde{\Omega}$ by $N^{m}\cap[-j,j] = N^{m,n_0(j)}\cap[-j,j]$ for all $j$. Finally, if $C\subset\R$ is bounded, then there exists $j\in\N^*$ such that $C\subset[-j,j]$, thus, almost surely (on $\tilde{\Omega}$ which does not depend on $C$), we have, for $n$ large enough,
\[ N^{m,n}\cap C = N^{m,n}\cap[-j,j] \cap C = N^{m}\cap[-j,j] \cap C = N^m\cap C.\]

\textbf{Step 3:} In this step we show that the limit process $(\lambda,N,W)$ found in the Step 2 is indeed a $\Equ{\model}{-\infty}$ process and is also 1-stationary.

First we prove that for all $m\in\M$, the set $N^m$ matches its aimed value $\bar{N}^m$ being the points under $\lambda^m$. One can observe that $\lambda^m$ is predictable as limit of predictable terms. So, for $j\in\N^*$ and $m\in\M$, we have
\begin{align*}
    \E{\card\big( (N^{m,n}\vartriangle \bar{N}^{m})\cap[-j,j])\big)} & = \E{\itg_{-j}^j\itg_{\R_+} \ind_{\lambda^{m,n}_s \wedge \lambda^m_s < x \leq \lambda^{m,n}_s \vee \lambda^m_s} d\pi^m(s,x)} \\
    & = \itg_{-j}^j \E{\vert\lambda^{m,n}_s -\lambda^m_s\vert} ds\\
    & \leq 2j \dfrac{Kr^n}{1-r} \xrightarrow[n\to\infty]{}0.
\end{align*}

Borel-Cantelli Lemma applies once again and we conclude that $\forall j\in\N^*$, almost surely, $N^{m}\cap[-j,j] = \bar{N}^{m}\cap[-j,j]$. Which implies that almost surely, $N^{m}=\bar{N}^{m}$.

Let us denote $\bar{\lambda}^m_t := \Phi^m\bigg(\displaystyle\sum_{m'\in\M} \itg_{-\infty}^{t-} h^m_{m'}(t-s) dN^{m'}_s + \sum_{p\in\Pg}\itg_{-\infty}^{t-} \J^m_{p}(t-k,W^{p}_{k})d\mp{p}{k}\bigg)$. Let us now check that for all $m\in\M$ we have $\lambda^m_t = \bar{\lambda}^m_t \quad d\P{\omega}\otimes dt$ almost everywhere. Let $m\in\M$ and $t\in\R$, then,
\begin{align*}
    \E{\vert \lambda^{m,n}_t -\bar{\lambda}^m_t\vert} &  \leq L_m\bigg( \sum_{m'\in\M} \itg_{-\infty}^{t-} \vert h^m_{m'}(t-s)\vert \E{\vert \lambda^{m',n}_s-\lambda^{m'}_s\vert}ds + \sum_{p\in\Pg}\itg_{-\infty}^{t-} \vert h^m_{p}(t-k)\vert\E{\vert W^{p,n}_{k}-W^p_k\vert}d\mp{p}{k}\bigg)\\
    & \leq C r^n \xrightarrow[n\to\infty]{}0,
\end{align*}
with $C$ a constant depending on $\model$. Thus, for all $t\in\R$ we have $\P{\lambda^m_t = \bar{\lambda}^m_t}=1$. Then, Fubini theorem implies that $d\P{\omega}\otimes dt$ almost everywhere, $\lambda^m_t = \bar{\lambda}^m_t$.  
Similarly, almost surely, for all $(p,k)\in\Q$,
\[W^p_k = \xi^p_k + \Phi^p\bigg(\sum_{m} \itg_{-\infty}^{k-} h^{p}_{m}(k-s)dN^{m}_s + \sum_{p'\in\Pg}\itg_{-\infty}^{k-} \J^{p}_{p'}(k-k',W^{p'}_{k'})d\mp{p'}{k'}\bigg).\]
In the end we constructed a $\Equ{\model}{-\infty}$ process. It is also 1-stationary since the initialisation and the Picard iteration are 1-stationary, thus, the limit is also 1-stationary.

\textbf{Step 4:} In this last step we show uniqueness. Consider two solutions, indexed by 1 and 2, and define 
\[\beta^m = \sup_{t\in\R} \E{\vert (\lambda^m_t)_1 - (\lambda^m_t)_2\vert}<\infty \quad \text{and} \quad \omega^p = \sup_{k\in\D_p} \E{\vert (W^p_{k})_1-(W^p_{k})_2\vert}<\infty.\] 
Then, using equations \eqref{eq2.2}, one has,
\[ \E{\vert (\lambda^m_t)_1-(\lambda^m_t)_2\vert} \leq \sum_{m'} \itg_{-\infty}^{t-} L_{m} \vert h^m_{m'}(t-s)\vert \beta^{m'} ds + \sum_{p\in\Pg}\itg_{-\infty}^{t-}L_m\vert h^{m}_{p}(t-k)\vert \omega^p d\mp{p}{k}.\]

\begin{flalign*}
	\text{And thus} \hspace{0.2\linewidth}
	\left\{
	\begin{aligned}
		& \beta^m \leq \sum_{m'\in\M} L_{m}\Vert h^m_{m'}\Vert_1 \beta^{m'} + \sum_{p\in\Pg} L_m\Vert h^m_{p}\VLd{n_{p}} \omega^p, \\
		& \omega^p \leq \sum_{m\in\M} L_{p}\Vert h^p_{m}\Vert_1 \beta^m + \sum_{p'\in\Pg} L_p\Vert h^p_{p'}\VLd{n_{p'}} \omega^{p'}.
	\end{aligned}
	\right.
	&&
\end{flalign*}
If $\beta = (\beta^m)_{m\in\M}$ and $\omega= (\omega^p)_{p\in\Pg}$, we thus have $\begin{pmatrix} \beta \\ \omega \end{pmatrix} \preceq \Hmod{\model} \begin{pmatrix} \beta \\ \omega \end{pmatrix}$. Since $\spr(\Hmod{\model})<1$ we have $\beta=0$ and $\omega=0$. Since $\Q$ is countable, a.s., $\forall (p,k)\in\Q, \ (W^p_{k})_1=(W^p_{k})_2$. It also implies that $d\P{\omega}\otimes dt$ almost everywhere, $(\lambda^m_t)_1 = (\lambda^m_t)_2$. Finally, one also has,
\begin{align*}
    \E{ \card\big(N^m_1 \vartriangle N^m_{2}\big)}  =  \E{ \itg_{\R} \vert  (\lambda^m_t)_1-(\lambda^m_t)_2\vert dt} \leq \itg_{\R} \beta^m dt =0.
\end{align*} 
Thus, a.s., $\forall m, \ S^m_{1}=S^m_{2}$.

\section{Proof of results of Section \ref{sec4}}\label{sec7}

\subsection{Proof of Lemma \ref{lem4.3} and Proposition \ref{prop4.4}}

\begin{proof}[Proof of Lemma \ref{lem4.3}]
Given a sequence $Z$ it is clear that the sequence $V$ is well defined by \eqref{eq4.1} since recursively (on increasing values of $k$) one can express $V^p_k$ as linear combinations of the $Z^{p'}_{k'}$ with $k' \leq k$, which proves existence of the coefficients. Uniqueness is obtained by setting the sequence $Z$ to $(\ind_{(p,k)=(p_0,k_0)})_{(p,k)\in\Q, \ k\geq k_0}$ since then associated $V$ sequence satisfies
\[ V^p_{k} = \sum_{p'\in\Pg} \itg_{k_0}^{k} \coef{p,k}{p',k'}{\model} Z^{p'}_{k'} d\mp{p'}{k'} = \coef{p,k}{p_0,k_0}{\model},\]
which proves uniqueness in \textbf{point 1} since $V$ is well defined and forces the value of the coefficient. We now prove \textbf{point 3}. The fact that $\coef{p_1,k_0}{p_0,k_0}{\model}=\ind_{p_0=p_1}$ is quite clear, take $Z=(\ind_{(p,k)=(p_0,k_0)})_{(p,k)\in\Q, \ k\geq k_0}$. On one side, by \textbf{point 1} we have
\begin{equation*}
	V^p_k  = \sum_{p'\in\Pg} \itg_{k_0}^k \coef{p,k}{p',k'}{\model} Z^{p'}_{k'} d\mp{p'}{k'} = \coef{p,k}{p_0,k_0}{\model},
\end{equation*}
thus we have $V^{p_1}_{k_0} = \coef{p_1,k_0}{p_0,k_0}{\model}$. On the other side, again by \textbf{point 1} we have
\begin{equation*}
    V^{p_1}_{k_0} = Z^{p_1}_{k_0} + \sum_{p'\in\Pg} \itg_{k_0}^{k_0 -} h^{p_1}_{p'}(k_0-k') V^{p'}_{k'} d\mp{p'}{k'} = Z^{p_1}_{k_0} = \ind_{p_0=p_1},
\end{equation*}
which concludes on $\coef{p_1,k_0}{p_0,k_0}{\model}=\ind_{p_0=p_1}$. Let $(Z^p_k)_{(p,k)\in\Q, \ k\geq k_0}$ be an arbitrary sequence and $V$ the associated sequence. Let $(p,k)\in\Q$ with $k>k_0$. Then 
\begin{align*}
    V^p_{k} & = Z^p_k + \sum_{p''\in\Pg} \itg_{k_0}^{k-} h^{p}_{p''}(k-k'') V^{p''}_{k''} d\mp{p''}{k''}\\
    & = Z^p_k + \sum_{p''\in\Pg} \itg_{k_0}^{k-} h^{p}_{p''}(k-k'') \left[ \sum_{p'\in\Pg} \itg_{k_0}^{k''} \coef{p'',k''}{p',k'}{\model} Z^{p'}_{k'} d\mp{p'}{k'}\right] d\mp{p''}{k''}\\
    & = Z^p_k + \sum_{p'\in\Pg} \itg_{k_0}^{k-} \left[ \sum_{p''\in\Pg} \itg_{k'}^{k-} h^{p}_{p''}(k-k'') \coef{p'',k''}{p',k'}{\model} d\mp{p''}{k''} \right] Z^{p'}_{k'}d\mp{p'}{k'}.
\end{align*}
By identification of the coefficients one obtains \textbf{point 3}. \textbf{Point 2}, aka the 1-periodicity, can be proved by looking at the version of $Z$ shifted by $1$, denoted $\tilde{Z}$. Then the associated $\tilde{V}$ sequence is $V$ shifted by $1$ and thus the coefficients must be 1-periodic. It can also be proved by induction for increasing values of $k-k'$ by using the formula of \textbf{point 3}.
\end{proof}

\begin{proof}[Proof of Proposition \ref{prop4.4}]
It is quite clear that the closed form satisfies the recursive property of point 3 of Lemma \ref{lem4.3}. Indeed, one only needs to cut the paths in two sub paths: all the first steps except the last one and the last one. Thus the closed form must be equal to the coefficients.    
\end{proof}

\subsection{Proof of Theorem \ref{th4.5}}

To prove Theorem \ref{th4.5} we solve the linear recursive equation of $W$ on itself by using coefficients defined in Lemma \ref{lem4.3} and Proposition \ref{propA.2} in Appendix \ref{appendixA}. It leads to an expression of $W$ only in term of $N$ and $\xi$. Then we inject this new expression of $W$ in the intensity equation to obtain an expression of the intensity only in terms of $N$ and $\xi$ leading to cluster representation. $L^1$ bounds on cluster functions are derived with the help of Lemma \ref{lemA.1} in Appendix \ref{appendixA}.

Recall that $\J^{\alpha}_{p}(s,w)=h^{\alpha}_{p}(s) w + b^{\alpha}_{p}(s)$. Let
\[Z^p_{k} :=  \xi^p_k + \sum_{m\in\M} \itg_{t_0}^{k-} h^{p}_m(k-s)dN^m_s + \sum_{p'\in\Pg} \itg_{t_0}^{k-} b^p_{p'}(k-k')d\mp{p'}{k'}, \]
\begin{flalign*}
    \text{so that} \hspace{0.1\linewidth} W^p_{k} & = \xi^p_k + \sum_{m\in\M} \itg_{t_0}^{k-} h^p_m(k-s)dN^{m}_s + \sum_{p'\in\Pg} \itg_{t_0}^{k-} \big(h^p_{p'}(k-k')W^{p'}_{k'}+b^p_{p'}(k-k')\big)d\mp{p'}{k'} &&\\
    & = Z^p_{k} + \sum_{p'\in\Pg}\itg_{t_0}^{k-} h^{p}_{p'}(k-k') W^{p'}_{k'} d\mp{p'}{k'}.
\end{flalign*} 
If $t_0\in\R$, then \eqref{dem th 4.5 eq1} is free by Lemma \ref{lem4.3}. If $t_0 = -\infty$, in order to apply Proposition \ref{propA.2} in Appendix \ref{appendixA}, which extends point 1 of Lemma \ref{lem4.3} to $k_0 = -\infty$ and to random sequences, we need $Z$ to be uniformly bounded in $\Lrv^1$. We have
\[ \sup_{(p,k)\in\Q} \E{\vert Z^p_{k}\vert } \leq   \sup_{p,k}\E{\xi^p_k } + \sum_m (\sup_{p}\Vert h^{p}_m\Vert_1) \sup_{s\in\R}\E{\vert \lambda^m_{s}\vert } + \sum_{p'\in\Pg} \sup_{p} \Vert b^p_{p'}\VLd{n_{p'}}. \]
Thus from \Cfime$[\xi]$, \Cspec$[\model]$ and Theorem \ref{th3.9} we have, $\sup_{(p,k)\in\Q} \E{\vert Z^p_{k}\vert }  < \infty$, leading to
\begin{equation}\label{dem th 4.5 eq1}
	W^p_{k}  = \sum_{p'\in\Pg}\itg_{t_0}^{k} \coef{p,k}{p',k'}{\model} Z^{p'}_{k'} d\mp{p'}{k'} = \Hd^{b}_{p,k}(t_0) + \Hd^{\xi}_{p,k}(t_0) + \sum_m \itg_{t_0}^{k-} \Hd^m_{p,k}(t_0,t) dN^m_t.
\end{equation}
Then one can inject this expression into equation \eqref{eq2.2a} to get the expression of the intensity, which proves the first part of Theorem \ref{th4.5}. It remains to check the $\Lf^1$ bound under assumption \Cspecbis$[\model]$.
\begin{align*}
    \mathfrak{h}^m_{m_0}(s,t) & = h^m_{m_0}(t-s) + \sum_{p,p'\in\Pg}\itg_{s+}^{t-}\itg_{s+}^k  h^m_{p}(t-k) \coef{p,k}{p',k'}{\model}h^{p'}_{m_0}(k'-s) d\mp{p'}{k'} d\mp{p}{k} \\
    & = h^m_{m_0}(t-s) + \sum_{p,p'\in\Pg}\itg_{s+}^{\infty}\itg_{s+}^{\infty} \ind_{  k'\leq k<t}  h^m_{p}(t-k) \coef{p,k}{p',k'}{\model}h^{p'}_{m_0}(k'-s)  d\mp{p}{k} d\mp{p'}{k'}.
\end{align*}
Thus by Fubini,
\begin{equation}\label{eq9.1}
    \Vert \mathfrak{h}^m_{m_0}(s,\cdot)\Vert_1 \leq \Vert h^m_{m_0}\Vert_1 + \sum_{p,p'\in\Pg}\itg_{s+}^{\infty}\itg_{s+}^{\infty}  \Vert h^m_{p}\Vert_1  \ind_{  k'\leq k} \coef{p,k}{p',k'}{\model}h^{p'}_{m_0}(k'-s)  d\mp{p}{k} d\mp{p'}{k'}.
\end{equation} 
Lemma \ref{lemA.1} states bounds on sums of the coefficients, as we will need. Fix $p,p'\in\Pg$,
\begin{align*}
     \itg_{s+}^{\infty}\itg_{s+}^{\infty}  \ind_{  k'\leq k}  \coef{p,k}{p',k'}{\model}  h^{p'}_{m_0}(k'-s)  d\mp{p}{k} d\mp{p'}{k'} & = \sum_{k_0\in \D_{p'}\cap (s,s+1]} \sum_{n\in\N} \itg_{k_0+n}^{\infty}  \coef{p,k}{p',k_0+n}{\model}h^{p'}_{m_0}(k_0+n-s)  d\mp{p}{k} \\
    & =  \sum_{k_0\in \D_{p'}\cap (s,s+1]} \sum_{n\in\N} \itg_{k_0}^{\infty} \coef{p,k}{p',k_0}{\model}h^{p'}_{m_0}(k_0+n-s)  d\mp{p}{k} \\
    & \leq \sum_{k_0\in \D_{p'}\cap (s,s+1]}  \itg_{k_0}^{\infty} \coef{p,k}{p',k_0}{\model}  d\mp{p}{k}  \Vert h^{p'}_{m_0}\VLd{1} \\
    & \leq  n_{p}  \big[(\Id-\Hmat^W_W)^{-1}\big]_{pp'} \Vert h^{p'}_{m_0}\VLd{1},
\end{align*}
where we use, the fact that $\D_{p'} = \sqcup_{n\in\Z} (\D_{p'}\cap(0,1] + n)$, 1-periodicity of the coefficients and point 2 of Lemma \ref{lemA.1}. Combining this bound with \eqref{eq9.1} concludes the proof.

\subsection{Proof of Theorem \ref{th corps moment exp}}

The proof of Theorem \ref{th corps moment exp} is based on Proposition \ref{propB.1} in Appendix \ref{appendixB} and on Lemma \ref{lemC.1} in Appendix \ref{appendixC}. Proposition \ref{propB.1} gives the precise domination of a general HAR process by a linear HAR process. Thus it is enough to prove Theorem \ref{th corps moment exp} for linear parameters since assumptions of Theorem \ref{th corps moment exp} transfer from $(\model,\xi)$ to $(\model^+,\xi^+)$. Thus in the sequel we assume that $(\model,\xi) \in\lm$. It is also clear that for linear models, HAR processes started at time $t_0\in\R$ are dominated by the one stated at $t_0=-\infty$, thus we can assume $t_0=-\infty$. Thus in the sequel $X$ is the $\Equ{(\model,\xi)}{-\infty}$ process from Theorem \ref{th3.9}, with $(\model,\xi) \in\lm$. Theorem \ref{th4.5} applies thus we have a cluster decomposition for $X$. 

Remark that by assumption \Cspecter$[\model]$ we have $\Vert h^p_m\VLd{1} <\infty$, and thus for any $n\in\N^*, \ m\in\M,\ p\in\Pg$, we have
\begin{equation}
	\Vert h^p_m\VLd{1/n} \leq \Vert h^p_m\VLd{1} < \infty.
\end{equation}
Note that by \Cspec$[\model]$ we have $h_m^{\alpha} \in \Lf^1$ and $h^{\alpha}_{p},b^{\alpha}_{p} \in \Ld{n_p}$ for all $p\in\Pg,\ m\in\M, \ \alpha \in \M\cup\Pg$.
To apply Lemma \ref{lemC.1}, we need to fit both $\Mass_F(X)$ and the immigrant rate $(\I^m)_{m\in\M}$ to the framework of Lemma \ref{lemC.1}. From Theorem \ref{th4.5} we have, $\I^m_t = \mu_m + \sum_{p\in\Pg} \itg_{-\infty}^{t-} \left[ h^m_{p}(t-k)(\Hd^{b}_{p,k} + \Hd^{\xi}_{p,k})+b^m_p(t-k)\right]d\mp{p}{k}.$

From the definition of $\Hd^{b}_{p,k}$ and point 1 of Lemma \ref{lemA.1} (the coefficients are summable), we have
\begin{equation*}
	\Hd^{b}_{p,k}  \leq \sum_{p'\in\Pg}\itg_{-\infty}^k \coef{p,k}{p',k'}{\model} \Vert b^{p'}_{p''} \VLd{n_{p''}} d\mp{p'}{k'} \leq \sum_{p'\in\Pg} (\Id-\Hmat^W_W)^{-1}_{p,p'} \Vert b^{p'}_{p''}\VLd{n_{p''}}<\infty.
\end{equation*}
Thus we clearly have $\sup_{(p,k)\in\Q} \Hd^{b}_{p,k} < \infty$. Since $h^m_{p}, b^m_{p} \in \Ld{n_p}$, we have
\begin{equation}
	\mathscr{i}_m := \sup_{t\in\R}\left( \mu_m + \sum_{p\in\Pg} \itg_{-\infty}^{t-} \left[ h^m_{p}(t-k)\Hd^{b}_{p,k}+b^m_p(t-k)\right]d\mp{p}{k} \right) < \infty.
\end{equation}

Recall that $\Hd^{\xi}_{p,k} = \displaystyle\sum_{p'\in\Pg}\itg_{-\infty}^k \coef{p,k}{p',k'}{\model} \xi^{p'}_{k'} d\mp{p'}{k'}$. Thus it follows that
\begin{equation}
	\I^m_t \leq \mathscr{i}_m + \sum_{(p',k')\in\Q} A^m_{p',k'}(t) \xi^{p'}_{k'} \quad \text{with} \quad A^m_{p',k'}(t) = \displaystyle\sum_{p\in\Pg} \itg_{-\infty}^{t-} \ind_{k'\leq k <t} h^m_p(t-k) \coef{p,k}{p',k'}{\model} d\mp{p}{k}.
\end{equation}
Let $t\in\R$, $p'\in\Pg$ and $m\in\M$, applying point 1 of Lemma \ref{lemA.1} leads to,
\begin{align*}
	\itg_{\R} A^m_{p',k'}(t) d\mp{p'}{k'}  = \sum_{p\in\Pg} \itg_{-\infty}^{t-} \itg_{-\infty}^{k}  h^m_p(t-k) \coef{p,k}{p',k'}{\model} d\mp{p'}{k'} d\mp{p}{k} & \leq \sum_{p\in\Pg} \itg_{-\infty}^{t-}  h^m_p(t-k) (\Id-\Hmat^W_W)^{-1}_{p,p'} d\mp{p}{k}\\
	& \leq \sum_{p\in\Pg}  \Vert h^m_p\VLd{n_p} (\Id-\Hmat^W_W)^{-1}_{p,p'}.
\end{align*}
Finally we obtain, with the notations of Lemma \ref{lemC.1},
\begin{equation}
	A_{\infty} := \Big(\sup_{t\in\R}\itg_{\R} A^m_{p',k'}(t) d\mp{p'}{k'}\Big)_{p',m\in\Pg,\M} \preceq \big[(\Id-\Hmat^W_W)^{-1}\big]^{\intercal} \Hmat^N_W \prec \infty.
\end{equation}
The intensity fits now the framework of Lemma \ref{lemC.1}. We need now to fit $\Mass_F(X)$ into the framework of Lemma \ref{lemC.1}. Recall that by Theorem \ref{th4.5} we have $W^p_{k} = \Hd^{b}_{p,k} + \Hd^{\xi}_{p,k} + \sum_m \itg_{-\infty}^{k-} \Hd^m_{p,k}(t) dN^m_t$. Since $\Hd^{b}_{p,k}$ is bounded, there exists $K<\infty$ such that
\begin{align*}
	&\Mass_F(X) \leq \sum_m \itg F^m(t)dN^m_t + K \kappa_n(F) + \sum_p \itg F^p(k) \Big( \Hd^{\xi}_{p,k} + \sum_m \itg \Hd^m_{p,k}(t) dN^m_t \Big) d\mp{p}{k}\\
	& \ \ \ \ =  K \kappa_n(F) + \sum_m \itg \Big(\underbrace{F^m(t)+\sum_p \itg F^p(k)\Hd^m_{p,k}(t) d\mp{p}{k}}_{=: B^m(t)}\Big) dN^m_t + \sum_{p'} \itg \Big(\underbrace{\sum_p \itg F^p(k) \coef{p,k}{p',k'}{\model} d\mp{p}{k}}_{=: C^{p'}_{k'}}\Big) \xi^{p'}_{k'}d\mp{p'}{k'}.
\end{align*}
Considering the expression of $\Hd^m_{p,k}(t)$, we have $B^m(t) = F^m(t)+\sum_{p,p'\in\Pg} \itg \itg F^p(k) \coef{p,k}{p',k'}{\model} h^{p'}_m(k'-t) d\mp{p'}{k'} d\mp{p}{k}$. One can show, using summability of coefficients (Lemma \ref{lemA.1}), that there exists $K<\infty$, depending on $\model$, such that 
\begin{equation*}
	\Vert B^m\VLd{1/n} \leq K \kappa_n(F) \quad \text{and} \quad \Vert C^p \VLd{n_p} \leq K \kappa_n(F).
\end{equation*}
Recall in Lemma \ref{lemC.1} that $\L:\R_+^{\M} \longrightarrow \R_+^{\M}$ is the log-Laplace function of a $\Pois(\mathfrak{H}_{\model}^{\intercal})$ Galton Watson tree defined in Theorem 3.1 of \cite{leblanc_sharp_2025}. Applying Lemma \ref{lemC.1} gives, up to increasing $K$ and with $f(x,y) = x + y + x^2 + y^2$,
\begin{equation*}
	\E{e^{\Mass_F(X)}} \leq e^{K\kappa_n(F)} \times \exp\Big[f\Big(Kn\vert e^{\L(K\kappa_n(F)\ones)}-\ones)\vert_{\infty},K \kappa_n(F)\Big)\Big].
\end{equation*}
Since in Theorem 3.1 of \cite{leblanc_sharp_2025} it is proved that $\L$ is finite and Lipschitz on a neighborhood of $0$ for $\vert \cdot \vert_{\infty}$, and because $n\geq 1$, it is clear that $\E{e^{\Mass_F(X)}} \leq e^{K' n^2 \kappa_n(F)}$ whenever $\kappa_n(F)$ is small enough with $K'$ a finite constant.

\subsection{Proof of Proposition \ref{prop5.6}}

The fact that $\spr(\tilde{\Hmat}_{\model})<1 \Rightarrow \spr(\Hmod{\model})<1$ is a consequence of $\tilde{\Hmat}_{\model}\succeq \Hmod{\model}$. Indeed, for matrix with non-negative entries the spectral radius is increasing with the entries. Thus $\spr(\tilde{\Hmat}_{\model}) \geq \spr(\Hmod{\model})$, which implies $\spr(\Hmod{\model})<1$. For $\spr(\mathfrak{H}_{\model^+})<1$ it is a direct consequence of \eqref{eq5.1}, stating that 
\[\mathfrak{H}_{\model^+} \preceq \Hmat^N_N + \Hmat^N_W (\Id-H_W^W)^{-1}\tilde{\Hmat}^W_N, \]
and the following applied to $\tilde{\Hmat}_{\model}$; For any square matrix $R=\begin{pmatrix} A & B \\ C & D\end{pmatrix}$ with non-negative entries, $A,B$ also square matrices, if $\spr(R)<1$ then we have
\begin{equation}\label{borne rayon spec}
	\spr\Big( A+B(\Id-D)^{-1}C = A+B\sum_{n\in\N} D^n C\Big) \leq \spr(R).
\end{equation}
To prove \eqref{borne rayon spec}, one can check that for all $i\in\N$ we have the following,
\[\sum_{j\geq i} \bigg( A+B\sum_{n\in\N} D^n C \bigg)^j \preceq \begin{pmatrix} \Id & 0 \end{pmatrix} \bigg[\sum_{j\geq i} R^j\bigg] \begin{pmatrix} \Id \\ 0 \end{pmatrix}.\]
Thus one concludes with the following result: for any norm, we have
\[\forall 0 \leq x < 1, \quad \spr(R) \leq x \iff \forall \varepsilon >0, \ \limsup_{n\to\infty} \frac{ \big\Vert \sum_{j\geq n} R^j \big\Vert}{(x+\varepsilon)^n}  =0.\]

\section{Proof of results of Section \ref{sec5}}\label{sec8}

\subsection{Proof of Proposition \ref{prop6.1}}

Let $\model_{c}$ be parameters being equal to $\model$ except that $h^p_{p'}$ is replaced by $x\mapsto e^{ c x } h^p_{p'}(x)$ for all $p,p'\in\Pg$ where $c>0$ will be chosen later. Then, from the closed form of the coefficients (see Proposition \ref{prop4.4}) it is clear that we have
\begin{equation}\label{eq11.1}
    \coef{p,k}{p',k'}{\model_c} = e^{c (k-k')} \coef{p,k}{p',k'}{\model}.
\end{equation}
Considering assumption \Cexptail$[\model]$, by dominated convergence the function $c \mapsto \Vert e^{c \id} h^p_{p'}\VLd{n_{p'}}$ is continuous on an open set containing $0$. Thus, since the spectral radius is a continuous function, \Cspecbis$[\model]$ ensures that \Cspecbis$[\model_{c}]$ holds for $c$ small enough. Thus by point 1 of Lemma \ref{lemA.1},
\[\itg_{-\infty}^{k} \coef{p,k}{p',k'}{\model_{c}}d\mp{p'}{k'} \leq K <\infty,\]
with $K$ (which depends on $\model$ and $c$) uniformly over $(p,k)\in\Q$, which, combined to \eqref{eq11.1} gives the first result. For cluster functions, the result is a direct consequence of \Cexptail$[\model]$, exponential decay of coefficients we just proved and the expression of clusters functions given in Definition \ref{def4.7}. Indeed, for $t>s$, one has (the constants $K,c>0$ might change from line to line),
\begin{align*}
	\mathfrak{h}^m_{m'}(s,t) & = h^m_{m'}(t-s) + \sum_{p,p'\in\Pg}\itg_{s+}^{t-}\itg_{s+}^k h^m_{p}(t-k) \coef{p,k}{p',k'}{\model}h^{p'}_{m'}(k'-s) d\mp{p'}{k'} d\mp{p}{k} \\
	& \leq Ke^{c(t-s)} + \sum_{p,p'\in\Pg}\itg_{s+}^{t-}\itg_{s+}^k Ke^{c(t-k)} Ke^{c(k-k')}Ke^{c(k'-s)} d\mp{p'}{k'} d\mp{p}{k}\\
	& \leq Ke^{c(t-s)} + Ke^{c(t-s)} \sum_{p,p'\in\Pg}\itg_{s+}^{t-}\itg_{s+}^k  d\mp{p'}{k'} d\mp{p}{k}\\
	& \leq Ke^{c(t-s)} + K(t-s)^2 e^{c(t-s)} \leq Ke^{c(t-s)}.
\end{align*}

\subsection{Proof of Theorem \ref{th6.3}}

In this proof we denote $\model = \big((h^{\alpha}_{m})_{\alpha,m}, (\J^{\alpha}_p,h^{\alpha}_{p},b^{\alpha}_p)_{\alpha,p}, \ (\Phi^{m})_{m}, \ (\Phi^p)_p \big)$. In the sequel, $K$ and $c$ are positive constants depending only on $\model$. Usually $K$ is intended to be large and $c$ small. They might change from line to line. We denote by $\big((\lambda^{m,\mathfrak{C}_i})_{m\in\M},(N^{m,\mathfrak{C}_i})_{m\in\M},(W^{p,\mathfrak{C}_i})_{p\in\Pg}\big)$ the unique non-explosive $\Equ{\model}{t_0,\mathfrak{C}_i}$ process for $i=1,2$ (Theorem \ref{th3.4} applies). We define the deviation process $\big((\Delta \lambda^m)_{m\in\M}, (\Delta N^{m})_{m\in\M}, (\Delta W^{p})_{p\in\Pg}\big)$ by
\begin{itemize}
\item For $t\geq t_0$ let, $\quad \Delta\lambda^m_t = \vert \lambda^{m,\mathfrak{C}_1}_t - \lambda^{m,\mathfrak{C}_2}_t\vert$.
\item For $t \geq t_0$ let, $\quad d\Delta N^{m}_t = \vert dN^{m,\mathfrak{C}_1}_t - dN^{m,\mathfrak{C}_2}_t\vert = \pi^m\Big( \{t\} \times \big(\min(\lambda^{m,\mathfrak{C}_1}_t,\lambda^{m,\mathfrak{C}_2}_t),\max(\lambda^{m,\mathfrak{C}_1}_t,\lambda^{m,\mathfrak{C}_2}_t)\big]\Big)$.
\item For $(p,k)\in\Q$  with $k\geq t_0$ let, $\quad \Delta W^p_k = \vert W^{p,\mathfrak{C}_1}_k - W^{p,\mathfrak{C}_2}_k\vert$.
\end{itemize}

Lemma \ref{lemB.2} in Appendix \ref{appendixB} dominates the deviation process by a process similar to a vanishing linear HAR process with parameters $\model^+\in\lm$ defined in Definition \ref{def5.1}. Cluster fonctions of $\model^+$ defined in Definition \ref{def4.7} are denoted by $\tilde{\mathfrak{h}}^m_{m'}$. Lemma \ref{lemB.2} states the following.

Conditionally to $\F_{t_0-}$ there exist Poisson random measures $\mathscr{P}^m$ on $[t_0,\infty)\times \R_+$ for $m\in\M$, and a process $\big((\mathscr{l}^m)_{m\in\M},(\mathscr{N}^m)_{m\in\M},(\mathscr{W}^p)_{p\in\Pg}\big)$ adapted (and predictable for $\mathscr{l}$) to the filtration $\G_t =  \sigma \big(\mathscr{P}^m\big\vert_{[t_0,t]\times \R_+}, \ m\in\M\big)$ such that, 
\begin{itemize}
    \item for $m\in\M, \ t\geq t_0$, $ \Delta\lambda^m_t \leq \mathscr{l}^m_t$,
    \item for $m\in\M, \ t\geq t_0$, $d\Delta N^{m}_t \leq d\mathscr{N}^m_t$,
    \item for $(p,k)\in\Q, \ k\geq t_0$, $\Delta W^{p}_k \leq \mathscr{W}^p_k$,
\end{itemize}
\begin{flalign*}
\text{And for } t\geq t_0, \ C\subset [t_0,\infty) \text{ and } (p,k)\in\Q \text{ with } k\geq t_0, \quad
\left\{
\begin{aligned}
& \mathscr{l}^m_t  = u_0 e^{-c(t-t_0)} + \sum_{m'\in\M}\itg_{t_0}^{t-} \tilde{\mathfrak{h}}^m_{m'}(s,t) d\mathscr{N}^{m'}_s \\
& \mathscr{N}^m(C) = \itg_{C\times \R_+} \ind_{0< x\leq \mathcal{L}^m_s } d\mathscr{P}^m(s,x)\\
& \mathscr{W}^p_{k} = u_0 e^{-c(k-t_0)} + \sum_{m\in\M} \itg_{t_0}^{k-} Ke^{-c(k-s)} d\mathscr{N}^m_s,
\end{aligned}
\right. &&
\end{flalign*}
where $u_0 = K \Massiexp_{\delta,t_0}\big(\mathfrak{C}_1 \circleddash \mathfrak{C}_2\big)$ with $K, c, \delta>0$ constants depending only on $\model$. The process $\allowbreak\big((\mathscr{l}^m)_{m\in\M},\allowbreak (\mathscr{N}^m)_{m\in\M}, (\mathscr{W}^p)_{p\in\Pg}\big)$ is quite similar to a $\Equ{\model^+}{t_0}$ process. The main difference is that immigration rates of point processes vanish and for time series we have vanishing quantities instead of 1-stationary ones as in Theorem \ref{th4.5}. Thus one may have the following reasoning: since immigrant rates are vanishing, at some point there will be no more cluster roots and thus no more points when all the clusters will have ended. Similarly, values of time series will vanish.

Let $T_e$ the extinction time of $(\mathscr{N}^m)_{m\in\M}$, defined as the time of the last point, $T_e = \max_{m\in\M} \sup \ \{t\geq t_0 \mid d\mathscr{N}^m_t = 1\}$. Let us prove the following result.

\begin{proposition}\label{prop11.1}
    There exists $K,c>0$ constants depending only on $\model$ such that
    \begin{equation*}
        \forall t\geq 0, \ \P{T_e \leq t_0+t} \geq 1 - K \E{\Massiexp_{\delta,t_0}\big(\mathfrak{C}_1 \circleddash \mathfrak{C}_2\big)}e^{-ct}.
    \end{equation*}
\end{proposition}

\begin{proof}[Proof of Proposition \ref{prop11.1}]
    First let us work conditionally to $\F_{t_0-}$ and thus $u_0$ is deterministic. We use cluster representation. Denote by $\Pi^m$ the immigrant points of $\mathscr{N}^m$. We use the same notations as in Section \ref{sec4.3} so that \[\Pi^m = \{s \in [t_0,\infty) \mid (\varnothing,m,s) \ \text{is the root of a cluster}\}.\] For an immigrant point $(\varnothing,m,s)$ we denote by $\tau(m,s)$ the (temporal) length of the cluster it generates. We emphasize that the law of the length of the cluster depends on $s$ since cluster functions $\tilde{\mathfrak{h}}(s,t)$ depend on both $s$ and $t$ and not just on $t-s$. We also denote by $\tau^m(s)$ a random variable with law being the length of a cluster born at time $s$ and with a root of type $m$. Conditionally to $u_0$ we have,\allowdisplaybreaks
    \begin{align*}
        \E{\ind_{T_e \leq t_0+t}} = \E{\prod_{m\in\M}\prod_{s\in\Pi^m} \ind_{\tau(m,s)\leq t_0+t-s}} & = \E{\prod_{m\in\M}\prod_{s\in\Pi^m} \P{\tau^m(s)\leq t_0+t-s}}\\
        & = \E{\exp\left(\sum_{m\in\M} \itg_{t_0}^{\infty} \log \big[\P{\tau^m(s)\leq t_0+t-s}\big] d\Pi^m_s\right)}\\
        & = \exp\left(\sum_{m\in\M}\itg_{t_0}^{\infty} \big(\P{\tau^m(s)\leq t_0+t-s}-1\big)u_0 e^{-c(s-t_0)} ds\right)\\
        & = \exp\left(-\sum_{m\in\M}\itg_{t_0}^{\infty} \P{\tau^m(s) > t_0+t-s} u_0 e^{-c(s-t_0)} ds\right),
    \end{align*}
    with $u_0 e^{-c(s-t_0)}$ the immigrants rate. Let $G^m_s$ a cluster born at time $s$ with root of type $m$, then for $t_0+t\geq s$ we have
    \begin{align*}
        \P{\tau^m(s) > t_0+t-s} & = \P{ G^m_s([s+t_0+t-s,\infty)) \geq 1} \\
        & = \P{e^{r \card(G^m_s\cap[s+t_0+t-s,\infty))}-1 \geq e^{r}-1}\leq Ke^{-c(t_0+t-s)},
    \end{align*}
    where in the last step we use Markov inequality and Lemma \ref{lem6.2} (which applies since \Cexptail$[\model]$, \Cspecter$[\model]$ implies \Cexptail$[\model^+]$ and \Cspecter$[\model^+]$), the constant $r$ being the one of Lemma \ref{lem6.2}. If $s>t_0+t$ then we trivially have $\P{\tau^m(s) > t_0+t-s} = 1$. It follows that
    \[ \P{T_e \leq t_0+t} \geq \exp\left(-\itg_{t_0}^{t_0+t} u_0 K e^{-c(t_0+t-s)} e^{-c (s-t_0)}ds-\itg_{t_0+t}^{\infty} u_0 e^{-c (s-t_0)}ds\right).\]
    From this point basic calculations leads to $\P{T_e \leq t_0+ t} \geq \exp\left(-u_0Ke^{-ct}\right)\geq 1-u_0Ke^{-ct}$.
    It remains only to integrate with respect to $\F_{t_0-}$ which concludes the proof.
\end{proof} 
Proposition \ref{prop11.1} deals with the part of $\Mass_{[t_0+t,\infty)}\big(X^{t_0,\mathfrak{C}_1} \circleddash X^{t_0,\mathfrak{C}_2}\big)$ coming from point propcesses. To control the time series part, we need bounds on the total number of points. Let $\mathscr{N}_{\infty} = \sum_{m\in\M} \mathscr{N}^m([t_0,\infty))$ the total number of points. Let us prove that moments of $\mathscr{N}_{\infty}$ are linked to moments of $\Massiexp_{\delta,t_0}\big(\mathfrak{C}_1 \circleddash \mathfrak{C}_2\big)$. More precisely we have the following.

\begin{proposition}\label{prop11.2}
    For all $q\geq 1$, there exists $K_q$ depending on $q$ and $\model$ such that
    \begin{equation*}
        \E{ \big(\mathscr{N}_{\infty}\big)^q}  \leq K_q \E{\Massiexp_{\delta,t_0}\big(\mathfrak{C}_1 \circleddash \mathfrak{C}_2\big)+\Massiexp_{\delta,t_0}\big(\mathfrak{C}_1 \circleddash \mathfrak{C}_2\big)^q}.
    \end{equation*}
\end{proposition}

\begin{proof}[Proof of Proposition \ref{prop11.2}]
    We use the same notations as in Proposition \ref{prop11.1}. For $s\in\Pi^m$, the cluster generated by $(\varnothing,m,s)$ is $G^{m}_{s}$. As in the proof of Proposition \ref{prop11.1} we work conditionally to $\F_{t_0-}$ and integrate at the end. Let $\mathcal{A}^m$ be a $\Pois(\mathfrak{H}_{\model^+}^{\intercal})$ Galton Watson process with root of type $m\in\M$. As explained in Section \ref{sec4.3}, for any $s\in\Pi^m$ and any $m\in\M$, there exists a coupling such that $\card(\mathcal{A}^m)\geq \card(G^m_s)$. Since conditionally to $(\Pi^m)_{m\in\M}$ clusters are independent we have,
    \begin{align*}
        \E{ \mathscr{N}_{\infty}^q} & = \E{\E{\bigg(\sum_{m\in\M}\sum_{s\in\Pi^m} \card(G^{m}_{s}) \bigg)^q \cond (\Pi^m)_m }} \leq \E{\bigg(\sum_{m\in\M}\sum_{s\in\Pi^m} \E{\card(G^{m}_{s})^q}^{1/q}\bigg)^q}\\
        	& = \E{\bigg(\sum_{m\in\M} \card(\Pi^m) \E{\card(\mathcal{A}^{m})^q}^{1/q}\bigg)^q} \leq \E{\card(\Pi)^q} \max_{m\in\M}\E{\card(\mathcal{A}^m)^q},
    \end{align*}
    where we used Minkowski inequality and where $\Pi = \bigcup_{m\in\M} \Pi^m$. Since conditionally to $\F_{t_0-}$, $\card(\Pi)$ follows a law $\Pois(M u_0/c)$ where $u_0=K \Massiexp_{\delta,t_0}\big(\mathfrak{C}_1 \circleddash \mathfrak{C}_2\big)$ we have
    \begin{equation*}
        \E{\card(\Pi)^q}  = \E{\E{\card(\Pi)^q\cond \F_{t_0-}}}\leq K_q \E{\Massiexp_{\delta,t_0}\big(\mathfrak{C}_1 \circleddash \mathfrak{C}_2\big) + \Massiexp_{\delta,t_0}\big(\mathfrak{C}_1 \circleddash \mathfrak{C}_2\big)^q},
    \end{equation*}
    where $K_q$ is a constant depending on $q$ and $K M/c$ (which depends on $\model$) such that the moment of order $q$ of the law $\Pois(\alpha K M /c)$ is bounded by $K_q(\alpha + \alpha^q)$. Which concludes the proof of Proposition \ref{prop11.2}.
\end{proof}
Let us now state the last intermediate result.
\begin{proposition}\label{prop11.3}
    There exists $c,K>0$ constants depending only on $\model$ such that for all $t\geq 0$ we have
    \begin{equation*}
        \P{\sum_{p\in\Pg}\itg_{t_0+t}^{\infty} \mathscr{W}^p_{k} d\mp{p}{k} \leq e^{-c t}} \geq 1 - K \E{\Massiexp_{\delta,t_0}\big(\mathfrak{C}_1 \circleddash \mathfrak{C}_2\big)} e^{-c t}.
    \end{equation*}
\end{proposition}
\begin{proof}[Proof of Proposition \ref{prop11.3}]
    Recall that $\mathscr{W}^p_{k} = u_0 e^{-c(k-t_0)} + \sum_{m\in\M} \int_{t_0}^{k-} Ke^{-c(k-s)} d\mathscr{N}^m_s$. Thus 
    \begin{equation}\label{eq prop 8.3}
    	\sum_{p\in\Pg}\itg_{t_0+t}^{\infty} \mathscr{W}^p_{k} d\mp{p}{k} \leq K u_0 e^{-ct} + \sum_{m\in\M} \itg_{t_0}^{\infty} K e^{-c(t_0+t-s)} d\mathscr{N}^m_s \leq K u_0 e^{-ct} + K \mathscr{N}_{\infty} e^{-c(t_0+t-T_e)}.
    \end{equation}
    From Proposition \ref{prop11.1} we have $\P{T_e > t_0+t/4} \leq K' \E{\Massiexp_{\delta,t_0}\big(\mathfrak{C}_1 \circleddash \mathfrak{C}_2\big)} e^{-c' t}$.\\
    From Proposition \ref{prop11.2} and Markov inequality $\P{\mathscr{N}_{\infty} > e^{ct/4}/2K} \leq K' \E{\Massiexp_{\delta,t_0}\big(\mathfrak{C}_1 \circleddash \mathfrak{C}_2\big)} e^{-c't}$.\\
    Finally, Markov inequality also gives $\P{u_0 > e^{ct/2}/2K} \leq K' \E{\Massiexp_{\delta,t_0}\big(\mathfrak{C}_1 \circleddash \mathfrak{C}_2\big)} e^{-c' t}$.\\
    Combining previous inequalities and \eqref{eq prop 8.3} leads to
    \[\P{K u_0 e^{-ct} + K \mathscr{N}_{\infty} e^{-c(t_0+t-T_e)} > e^{-ct/2}} \leq K' \E{\Massiexp_{\delta,t_0}\big(\mathfrak{C}_1 \circleddash \mathfrak{C}_2\big)} e^{-c' t}.\]
    Thus, $\P{\sum_{p\in\Pg}\int_{t_0+t}^{\infty} \mathscr{W}^p_{k} d\mp{p}{k} > e^{-ct/2 }} \leq  K' \E{\Massiexp_{\delta,t_0}\big(\mathfrak{C}_1 \circleddash \mathfrak{C}_2\big)} e^{-c' t}$, which concludes the proof up to reducing either $c$ or $c'$ so that $c/2 = c'$.
\end{proof}

The proof of Theorem \ref{th6.3} follows easily from this point. First remark that 
\begin{equation*}
	\Big( \Mass_{[t_0+t,\infty)}\big(X^{t_0,\mathfrak{C}_1} \circleddash X^{t_0,\mathfrak{C}_2}\big) \leq e^{-ct} \Big) \supset \bigg( \big(T_e \leq t_0 + t\big) \cap \Big( \sum_{p\in\Pg}\int_{t_0+t}^{\infty} \mathscr{W}^p_{k} d\mp{p}{k} \leq e^{-c t} \Big) \bigg).
\end{equation*}
Thus Propositions \ref{prop11.1} and \ref{prop11.3} conclude on the first part of Theorem \ref{th6.3}. Let us focus on the almost sure convergence. We have
\begin{equation}\label{eq th 5.3}
	\Mass_{[t_0+t,\infty)}\big(X^{t_0,\mathfrak{C}_1} \circleddash X^{t_0,\mathfrak{C}_2}\big) \leq \mathscr{N}_{\infty} \ind_{T_e \leq t_0+t}+ u_0 K e^{-ct} + K\mathscr{N}_{\infty}e^{-c(t_0+t-T_e)}.
\end{equation}
Since from Proposition \ref{prop11.2} the expectation of $\mathscr{N}_{\infty}$ is finite, it follows that $\mathscr{N}_{\infty}<\infty$ almost surely and thus $T_e<\infty$ almost surely which, with \eqref{eq prop 8.3}, conclude the almost sure convergence. It remains the $\Lrv^q$ convergence. It is a consequences of the following statement.
\begin{equation}\label{eq11.2}
    \forall 1 \leq q < q_0, \ \forall x,z\geq 0, \ \  \Big(\P{Y > x} \leq z \ \Longrightarrow \ \E{Y^q} \leq x^q + z^{1-q/q_0} \E{Y^{q_0}}^{q/q_0}\Big),
\end{equation}
where $q_0>1$ and $Y$ is non-negative random variable. For $t>0$ we apply \eqref{eq11.2} to $Y =Y_t= \Mass_{[t_0+t,\infty)}\big(X^{t_0,\mathfrak{C}_1} \circleddash X^{t_0,\mathfrak{C}_2}\big)$ with $x = x_t = e^{-ct}$ and $z = z_t = K \E{\Massiexp_{\delta,t_0}\big(\mathfrak{C}_1 \circleddash \mathfrak{C}_2\big)} e^{-c t}$. Since by assumption $\E{\Massiexp_{\delta,t_0}\big(\mathfrak{C}_1 \circleddash \mathfrak{C}_2\big)^{q_0}}<\infty$ it is clear by \eqref{eq th 5.3} and Proposition \ref{prop11.2} that $Y_t \in \Lrv^{q_0}$ uniformly in $t$. It follows that
\[ \E{\Mass_{[t_0+t,\infty)}\big(X^{t_0,\mathfrak{C}_1} \circleddash X^{t_0,\mathfrak{C}_2}\big)^q} \leq e^{-qct} + \big(K \E{\Massiexp_{\delta,t_0}\big(\mathfrak{C}_1 \circleddash \mathfrak{C}_2\big)}e^{-ct}\big)^{1-q/q_0} \E{\Massiexp_{\delta,t_0}\big(\mathfrak{C}_1 \circleddash \mathfrak{C}_2\big)^{q_0}}.\]
The proof is complete since the RHS goes to $0$ as $t\to\infty$. 

\subsection{Proof of Corollary \ref{cor6.4}}

To prove Corollary \ref{cor6.4} we apply Theorem \ref{th6.3} with the initial conditions $\mathfrak{C}^{-\infty}$, given by the stationary HAR process on $(-\infty,t_0)$, and $\mathfrak{C}^{\varnothing}$, the empty one. From this choice, it is clear that the processes of interest are, respectively, the stationary process and the process started at time $t_0$. Let $\delta > 0$. We need to prove that $\Massiexp_{\delta,t_0}\big(\mathfrak{C}^{-\infty} \circleddash \mathfrak{C}^{\varnothing}\big)$ has a finite mean and, under \Csubg$[\xi]$, moments of any order. We have
\[\Massiexp_{\delta,t_0}\big(\mathfrak{C}^{-\infty} \circleddash \mathfrak{C}^{\varnothing}\big) := \Massi_F\big(\mathfrak{C}^{-\infty} \circleddash \mathfrak{C}^{\varnothing}\big) \leq \Mass_{F}\big(X^{-\infty}\big) + \sum_{p\in\Pg} \itg_{\R} F^p(k)d\mp{p}{k}\leq \Mass_{F}\big(X^{-\infty}\big) + K_{\delta},\]
with $F^{\alpha}(s) = e^{-\delta\vert t_0-s\vert}$ for all $\alpha \in\M\cup\Pg$ and $K_{\delta}$ a constant depending on $\delta$. By Theorem \ref{th3.9} \[\sup_{m\in\M, \ t\in\R} \E{\lambda^{m,-\infty}_t}<\infty \quad \text{and} \quad \sup_{(p,k)\in\Q}\E{\vert W^{p,-\infty}_k\vert}<\infty,\] thus $\Massiexp_{\delta,t_0}\big(\mathfrak{C}^{-\infty} \circleddash \mathfrak{C}^{\varnothing}\big)$ has an expectation. If \Csubg$[\xi]$ holds then Theorem \ref{th corps moment exp} applies to $\Mass_{F}\big(X^{-\infty}\big)$. Thus there exists finite exponential moments and so moments of any order.

\subsection{Proof of Theorem \ref{th ergo}}

In this proof $c$ and $C$ are respectively a small and a large constant, depending on $\model$, \Csubg$[\xi]$ and also possibly on $a,\delta,C_{\f}$. When they also depend on $q\geq 1$ we denote them $c_q,C_q$. Constants may change from line to line.

We denote by $X^{\varnothing}$ the empty process, meaning that $N^{\varnothing,m} = \varnothing$ for all $m\in\M$ and $W^{\varnothing,p}_k = \varnothing$ for all $(p,k)\in\Q$. For $t_0\in\R$ we also denote by $X^{t_0}$ the process $\Equ{\model,\xi}{t_0}$ started at time $t_0$ with the empty initial condition $\mathfrak{C}^{\varnothing}$. Recall that $T=4n\tau + r$ with $0\leq r<4\tau$. For $j\in\N$ we define the process $\mathbb{X}^j$ by 
\begin{equation}
	\mathbb{X}^j =
\begin{cases}
X^{(j-2)\tau} & \text{on } [(j-1)\tau,(j+2)\tau)\\
X^{\varnothing} & \text{on } \R\setminus [(j-1)\tau,(j+2)\tau]
\end{cases}.
\end{equation}
With these definitions we have
\begin{align*}
	Z_T := \Big\vert\itg_0^T \f(\theta_t X) dt - T\E{\itg_0^{1} \f(\theta_t X) dt}\Big\vert  \leq & \ \Big\vert\itg_0^{4n\tau} \f(\theta_t X) dt - \itg_0^{4n\tau} \f(\theta_t \mathbb{X}^{\floor{t/\tau}}) dt\Big\vert  \quad \textcircled{1} \\
	& + \Big\vert \itg_0^{4n\tau} \f(\theta_t \mathbb{X}^{\floor{t/\tau}}) dt - 4n \E{\itg_0^{\tau} \f(\theta_t \mathbb{X}^0) dt} \Big\vert  \quad \textcircled{2} \\
	& + \Big\vert 4n \E{\itg_0^{\tau} \f(\theta_t \mathbb{X}^0) dt} - 4n \E{\itg_0^{\tau} \f(\theta_t X) dt} \Big\vert \quad \textcircled{3}\\
	& + \Big\vert\itg_{3n\tau}^T \f(\theta_t X) dt - r \E{\itg_0^{1} \f(\theta_t X) dt} \Big\vert. \quad \textcircled{4}
\end{align*}
Almost sure convergence $Z_T/T \rightarrow 0$ is obtained by applying \eqref{non asymp th ergo} with $x = \varepsilon >0$ arbitrary, $n \approx \tau \approx \sqrt{T}$ and $q>3$. We conclude by Borel Cantelli Lemma that almost surely, $\liminf Z_T/T \leq \varepsilon$. We now prove \eqref{non asymp th ergo}.

Let us introduce the following quantities,
\begin{equation*}
	\mathcal{Y} = \max_{0\leq t < T} \Mass^{\exp}_{\delta,t}\big(X\circleddash \mathbb{X}^{\floor{t/\tau}}\big), \quad Y = \max_{0\leq t < T} \Mass^{\exp}_{\delta,t}\big(X\big), \quad \text{and} \quad \mathbb{Y} = \max_{0\leq t <T} \Mass^{\exp}_{\delta,t}\big(\mathbb{X}^{\floor{t/\tau}}\big).
\end{equation*}
As preliminaries to the proof we prove the following bounds. For all $q\geq 1$ there exist constants $c,C,c_q,C_q>0$ such that, 
\begin{equation}\label{eq Y diff}
	\P{\mathcal{Y} > Ce^{-c \tau}} \leq C T e^{-c \tau},
\end{equation}
\begin{equation}\label{eq Y X}
	\P{Y > x} \leq C T e^{-c x}, \quad x\geq 0,
\end{equation}
\begin{equation}\label{eq Y bbX}
	\P{\mathbb{Y} > x} \leq C T e^{-c x}, \quad x\geq 0,
\end{equation}
\begin{equation}\label{Lp Y diff}
	\max_{0\leq t < T} \E{\big\vert \Mass^{\exp}_{\delta,t}\big(X\circleddash\mathbb{X}^{\floor{t/\tau}}\big)\big\vert^q} \leq C_q e^{-c_q \tau},
\end{equation}
\begin{equation}\label{Lp Y X et bbX}
	\max_{0\leq t < T} \E{\big\vert \Mass^{\exp}_{\delta,t}\big(X\big)\big\vert^q} \leq C_q, \quad \text{and} \quad \max_{0\leq t < T} \E{\big\vert \Mass^{\exp}_{\delta,t}\big(\mathbb{X}^{\floor{t/\tau}}\big)\big\vert^q} \leq C_q.
\end{equation}

For equations \eqref{eq Y diff}, \eqref{eq Y X} and \eqref{eq Y bbX} the idea is to transform continuous maximums into discrete maximums. Let $0\leq t <\tau$, we have
\begin{equation}\label{eq domi X bbX}
	\Mass^{\exp}_{\delta,t}\big(X\circleddash \mathbb{X}^0\big) \leq \Mass_{[-\tau,2\tau]}\big(X\circleddash X^{-2\tau}\big) + e^{-\delta \tau} \Mass^{\exp}_{\delta,-\tau}\big(X\big) + e^{-\delta \tau} \Mass^{\exp}_{\delta,2\tau}\big(X\big).
\end{equation}
By Theorem \ref{th corps moment exp} there exist finite exponential moments for $\Mass^{\exp}_{\delta,-\tau}\big(X\big)$ and $\Mass^{\exp}_{\delta,2\tau}\big(X\big)$. Thus one can combine Markov inequality and Corollary \ref{cor6.4} to obtain
\begin{equation*}
	\P{\max_{0\leq t<\tau} \Mass^{\exp}_{\delta,t}\big(X\circleddash \mathbb{X}^0\big) > e^{-c \tau}+2\tau e^{-\delta \tau}} \leq C e^{-c \tau}. 
\end{equation*}
Similarly the same holds for $j\tau \leq t < (j+1)\tau$ for $j\in\N$, thus by union bound, up to reducing $c$ and increasing $C$ one gets \eqref{eq Y diff}.\\ 
Remark now that both $\Mass^{\exp}_{\delta,t}\big(X\big)$ and $\Mass^{\exp}_{\delta,t}\big(\mathbb{X}^{\floor{t/\tau}}\big)\leq \Mass^{\exp}_{\delta,t}\big(X^{(\floor{t/\tau}-2)\tau}\big)$ have finite exponential moments, uniformly in $t$, by Theorem \ref{th corps moment exp}. Which proves \eqref{Lp Y X et bbX}.\\
For \eqref{eq Y X} we can use the following fact; $\Mass^{\exp}_{\delta,t}\big(X\big) \leq e^{\delta}\Mass^{\exp}_{\delta,\floor{t}}\big(X\big)$.  Thus we can transform $\max_{0\leq t\leq T}$ into $\max_{t=0,\cdots,T-1}$. Theorem \ref{th corps moment exp} gives finite exponential moments for $\Mass^{\exp}_{\delta,j}\big(X\big)$ for $0\leq j\leq T-1$, thus Markov inequality and union bound leads to the result.\\
Similarly, for \eqref{eq Y bbX} remark that $\Mass^{\exp}_{\delta,t}\big(\mathbb{X}^{\floor{t/\tau}}\big) \leq e^{\delta}\Mass^{\exp}_{\delta,\floor{t}}\big(X^{(\floor{t/\tau}-2)\tau}\big)$. Since $\card\big(\{(\floor{t},\floor{t/\tau}) \mid 0\leq t < T\} \big) = T$, again Theorem \ref{th corps moment exp} combined with Markov inequality and union bound lead to \eqref{eq Y bbX}.\\
By periodicity it is enough to prove \eqref{Lp Y diff} for $0\leq t < \tau$. By \eqref{eq domi X bbX} we have $\Mass^{\exp}_{\delta,t}\big(X\circleddash \mathbb{X}^0\big) \leq \Mass_{[-\tau,2\tau]}\big(X\circleddash X^{-2\tau}\big) + e^{-\delta \tau} V_{\tau}$ with $V_{\tau}$ uniformly, in $\tau$, bounded in $\Lrv^q$ by \eqref{Lp Y X et bbX}. Corollary \ref{cor6.4} concludes since it gives $\Lrv^q$ convergence towards $0$, with exponential rates, of $\Mass_{[-\tau,2\tau]}\big(X\circleddash X^{-2\tau}\big)$.

To bound \textcircled{1}, first remark that by \eqref{fonctionnelle},
\begin{equation*}
	\textcircled{1} \leq T C_{\f} \big( 1 + Y + \mathbb{Y}\big)^{a} \mathcal{Y}^{a}.
\end{equation*}
Then we use bounds \eqref{eq Y diff} and \eqref{eq Y X}, \eqref{eq Y bbX} with $x = \tau$, leading to $\P{\textcircled{1} > T C \tau^{a} e^{-a c \tau}} <  T C e^{-c\tau}$. Up to reduce (resp. increase) $c$ (resp. $C$),
\begin{equation}\label{control 1}
	\P{\textcircled{1} > T C e^{-c \tau}} <  T C e^{-c\tau}.
\end{equation}
To bound \textcircled{2} we use a concentration result, Fuk Nagaev inequality,  Corollary 1.8 of \cite{nagaev_large_1979}, given in Lemma \ref{lem12.3} below.

\begin{lemma}[Fuk-Nagaev]\label{lem12.3}
	Let $X_1,\cdots,X_{n}$ be centered and independent random variables. Let $S_n = \sum_{k=1}^n X_k$. Let also $M(q,n)= \sum_{k=1}^n \E{\vert X_k\vert ^q}$ for $q\geq 2$ and $\sigma_n^2 = \sum_{k=1}^n \E{X_k^2} = M(2,n)$. Then, for all $q>2$ and $x>0$ we have
\[\P{\vert S_n\vert \geq x }\leq \dfrac{(1+2/q)^q M(q,n)}{x^q} + 2 \exp\left(-\dfrac{2x^2}{(q+2)^2e^q\sigma_n^2}\right).\] 
\end{lemma}
Notice that we have $\itg_0^{4n\tau} \f(\theta_t \mathbb{X}^{\floor{t/\tau}}) dt = \sum_{j=0}^{4n-1} \itg_{j\tau}^{(j+1)\tau} \f(\theta_t \mathbb{X}^j) dt$ and denote $U_j := \itg_{j\tau}^{(j+1)\tau} \f(\theta_t \mathbb{X}^j) dt$. It is clear by 1-periodicity that random variables $(U_j)_j$ are identically distributed. By definition of $\mathbb{X}^j$, $U_j$ depends on $(\pi^m)_{m\in\M}$ and $\xi$ only through the time window $[(j-2)\tau,(j+2)\tau)$, thus, for $x=0,1,2,3$, random variables $(U_j)_{j\equiv x\pmod{4}}$ are independent, and thus i.i.d.. Let us focus on $(U_{4j})_{j=0,\cdots,n-1}$. For $j=0$ we have $U_0 = \sum_{i=0}^{\tau-1} \int_i^{i+1} \f(\theta_t \mathbb{X}^{0})dt$. Using a similar argument as for deriving \eqref{eq Y bbX} we have
\begin{equation*}
	\Big\vert \itg_i^{i+1} \f(\theta_t \mathbb{X}^{0})dt\Big\vert \leq C_{\f} \Big( 1 + e^{\delta}\Mass^{\exp}_{\delta,i}\big(\mathbb{X}^{0}\big) \Big)^{a}, \quad 0\leq i < \tau.
\end{equation*} 
By \eqref{Lp Y X et bbX}, it follows that $U_0$ is a sum of $\tau$ random variables uniformly bounded in $\Lrv^q$ for any $q\geq 1$. Thus for $\bar{U}_j := U_j - \E{\int_0^{\tau} \f(\theta_t \mathbb{X}^0) dt}$, centered random variables, we can apply Lemma \ref{lem12.3} with $M(q,n) \leq n\tau^q C_q$ leading to
\begin{equation*}
	\P{\vert \bar{U}_0+\bar{U}_4+\cdots +\bar{U}_{4n-4}\vert \geq x/4} \leq \frac{C_{q}n\tau^q}{x^q} + 2\exp\Big(-\frac{c_{q} x^2}{n \tau^2}\Big),\quad x\geq 0.
\end{equation*}
The same holds for the three other sequences and by union bound we have
\begin{equation}\label{control 2}
	\P{ \textcircled{2} \geq x} \leq \frac{C_q n\tau^q}{x^q} + 8\exp\Big(-\frac{c_{q} x^2}{n \tau^2}\Big), \quad x\geq 0.
\end{equation}

For \textcircled{3} we have the following.
\begin{align*}
	\Big\vert 3n \E{\itg_0^{\tau} \f(\theta_t \mathbb{X}^{0}) dt} - & 3n \E{\itg_0^{\tau} \f(\theta_t X) dt} \Big\vert  \leq 3n \itg_0^{\tau} \E{\vert \f(\theta_t \mathbb{X}^{0}) - \f(\theta_t X)\vert}dt\\
	& \leq 4n\tau  C_{\f} \max_{0\leq t <\tau} \E{\big( 1 + \Mass^{\exp}_{\delta,t}(X) + \Mass^{\exp}_{\delta,t}(\mathbb{X}^{0})\big)^{a} \Mass^{\exp}_{\delta,t}\big(X\circleddash \mathbb{X}^{0}\big)^{a}}\\
	&\leq T  C_{\f} \max_{0\leq t <\tau}\E{\big( 1 + \Mass^{\exp}_{\delta,t}(X) + \Mass^{\exp}_{\delta,t}(\mathbb{X}^0)\big)^{2a}}^{1/2} \max_{0\leq t <\tau}\E{\Mass^{\exp}_{\delta,t}\big(X\circleddash \mathbb{X}^0\big)^{2a}}^{1/2}.
\end{align*}
By equations \eqref{Lp Y X et bbX} and \eqref{Lp Y diff} it is clear that $\textcircled{3} \leq T  C e^{-c \tau}$. For the last term, \textcircled{4}, it is a variable of typical size $r < 4\tau$ such that $\E{ \textcircled{4}^q} \leq C_{q} r^q$ (same arguments as for moments of $U_j$). Thus for any $x >0$ we have,
\begin{equation}\label{control 4}
	\P{\textcircled{4} > x} \leq \frac{C_{q} \tau^q}{x^q} \leq \frac{C_{q} n \tau^q}{x^q}.
\end{equation}
By combining equations \eqref{control 1}, \eqref{control 2} applied with $x \leftarrow xT/2$, the bound $\textcircled{3} \leq T  C e^{-c \tau}$ and \eqref{control 4} with $x \leftarrow xT/2$ we have
\[\P{Z_T / T \geq x + Ce^{-c\tau}} \leq CTe^{-c\tau} + \frac{C_q n\tau^q}{x^q T^q} + 8\exp\Big(-\frac{c_q x^2 T^2}{n\tau^2}\Big).\]
It concludes the proof of \eqref{non asymp th ergo} since $\frac{n\tau^q}{T^q} \leq \frac{4^{-q}}{n^{q-1}}$ and $\exp\big(-\frac{c_q x^2 T^2}{n\tau^2}\big) \leq \exp(-16 c_q nx^2 ) \leq \frac{C_q}{n^{q-1}x^{2(q-1)}}$ by inequality $e^{-cy} \leq (s/ec)^s y^{-s}$ for any $y,s,c>0$.

%%%%%%%%%%%%%%%%%%%%%%%%%%%%%%%%%%%%%%%%%%%%%%
%% Single Appendix:                         %%
%%%%%%%%%%%%%%%%%%%%%%%%%%%%%%%%%%%%%%%%%%%%%%
%\begin{appendix}
%\section*{???}%% if no title is needed, leave empty \section*{}.
%\end{appendix}

%%%%%%%%%%%%%%%%%%%%%%%%%%%%%%%%%%%%%%%%%%%%%%
%% Multiple Appendixes:                     %%
%%%%%%%%%%%%%%%%%%%%%%%%%%%%%%%%%%%%%%%%%%%%%%
\begin{appendix}

\section{Auxiliary results for linear HAR processes}\label{appendixA}

In this appendix we present two additional results useful to the study of linear HAR processes. First, in Lemma \ref{lemA.1}, we derive $\Lf^1$ bounds and analytic properties of the coefficients defined in Lemma \ref{lem4.3}. Then we extend the use of the coefficients to random sequences in Proposition \ref{propA.2}.

\begin{lemma}\label{lemA.1}
    Let parameters $\model\in\lm$ and assume that assumption \Cspecbis$[\model]$ holds. Consider the coefficients $\left(\coef{p,k}{p',k'}{\model}\right)_{(p,k),(p',k')\in\Q, \ k\geq k'}$ defined by Lemma \ref{lem4.3}. Then the following properties hold.
    \allowdisplaybreaks
    \begin{enumerate}
        \item For any $(p_0,k_0)\in\Q$ for any $p\in\Pg$ we have, $\itg_{-\infty}^{k_0} \coef{p_0,k_0}{p,k}{\model}d\mp{p}{k} \leq \Big[(\Id-\Hmat^W_W)^{-1}\Big]_{p_0,p}$.
        \item For any $p_0,p\in\Pg$ we have, $\displaystyle\sum_{k_0\in\D_{p_0}\cap[0,1)}\itg_{k_0}^{\infty} \coef{p,k}{p_0,k_0}{\model}d\mp{p}{k} \leq n_{p} \times  \Big[ (\Id-\Hmat^W_W)^{-1} \Big]_{p,p_0}$.
    \end{enumerate}
\end{lemma}

Point 1 and point 2 give bounds on, respectively, the ascending and descending sums of the coefficients in term of the matrix $\Hmat^W_W$. The following result extends the use of coefficients to random sequences indexed by $\Z$ and also deals with uniqueness.

\begin{proposition}\label{propA.2}
    Let $\model\in\lm$ and suppose that \Cspecbis$[\model]$ holds. Suppose that $Z$ is a sequence of random variable such that $\sup_{(p,k)\in\Q} \E{\vert Z^p_{k}\vert } < \infty$. Then the equation
    \[ V^p_{k} = Z^p_{k} + \sum_{p'\in\Pg} \itg_{-\infty}^{k-} a^{p}_{p'}h^{p}_{p'}(k-k') V^{p'}_{k'} d\mp{p'}{k'}\] 
    has a unique solution among sequences $V$ for which $\sup_{(p,k)\in\Q} \E{\vert V^p_{k}\vert } < \infty$, being $\displaystyle V^p_{k} = \sum_{p'\in\Pg} \itg_{-\infty}^{k}  \coef{p,k}{p',k'}{\model} Z^{p'}_{k'} d\mp{p'}{k'}$.
\end{proposition}

\begin{proof}[Proof of Lemma \ref{lemA.1}]
We start with \textbf{point 1}.
For $p,p_0\in\Pg$, $k_0\in\D_{p_0}$ and $\tau \geq 1$, let $S^{p_0,k_0}_{p}(\tau) := \itg_{k_0-\tau}^{k_0} \coef{p_0,k_0}{p,k}{\model} d\mp{p}{k}$.
We also denote $S^{p_0}_{p}(\tau) = \sup_{k_0\in\D_{p_0}} S^{p_0,k_0}_{p}(\tau)$. Then by point 3 of Lemma \ref{lem4.3} we have,
\allowdisplaybreaks
\begin{align*}
    S^{p_0,k_0}_{p}(\tau) & = \ind_{p=p_0} + \itg_{k_0-\tau}^{k_0-} \coef{p_0,k_0}{p,k}{\model} d\mp{p}{k} = \ind_{p=p_0} + \itg_{k_0-\tau}^{k_0-} \left[ \sum_{p'\in\Pg} \itg_{k}^{k_0-} h^{p_0}_{p'}(k_0-k') \coef{p',k'}{p,k}{\model} d\mp{p'}{k'} \right] d\mp{p}{k}\\
    & = \ind_{p=p_0} + \sum_{p'\in\Pg} \itg_{k_0-\tau}^{k_0-}  \itg_{k_0-\tau}^{k'} h^{p_0}_{p'}(k_0-k') \coef{p',k'}{p,k}{\model} d\mp{p}{k}  d\mp{p'}{k'}\\
    & \leq \ind_{p=p_0} +  \sum_{p'\in\Pg} \itg_{k_0-\tau}^{k_0-}  h^{p_0}_{p'}(k_0-k') S^{p'}_{p}(\tau)  d\mp{p'}{k'} \leq \ind_{p=p_0} + \sum_{p'\in\Pg} \Vert h^{p_0}_{p'} \VLd{n_{p'}} S^{p'}_{p}(\tau).
\end{align*}
Taking the supremum on $k_0$ gives $S^{p_0}_p(\tau)\leq \ind_{p=p_0} + \sum_{p'\in\Pg} \Vert h^{p_0}_{p'} \VLd{n_{p'}} S^{p'}_p(\tau)$.
Thus if $S(\tau) = (S^p_{p'}(\tau))_{p,p'\in\Pg}$ we have, $N(\tau) \preceq \Id + \Hmat^W_W N(\tau)$. Equivalently,
\[ (\Id-\Hmat^W_W) N(\tau) \preceq \Id.\]
Since all the entries of $(\Id-\Hmat^W_W)^{-1} = \Id + \Hmat^W_W + (\Hmat^W_W)^2 + \cdots$ are non-negative, one can multiply both sides by $(\Id-\Hmat^W_W)^{-1}$ leading to $N(\tau) \preceq (\Id-\Hmat^W_W)^{-1}$, and so $N_{p_0,p}(\tau) \leq \big[(\Id-\Hmat^W_W)^{-1}\big]_{p_0,p}$. Letting $\tau\to\infty$ gives the intended result.\\
\textbf{Point 2} is a consequence of \textbf{point 1} and 1-periodicity. Indeed,
\allowdisplaybreaks
\begin{align*}
    \sum_{k_0\in\D_{p_0}\cap[0,1)}\itg_{k_0}^{\infty} \coef{p,k}{p_0,k_0}{\model}d\mp{p}{k} & =  \sum_{k_0\in\D_{p_0}\cap[0,1)}\itg_{k_0}^{\infty} \coef{p,k-\floor{k}}{p_0,k_0-\floor{k}}{\model}d\mp{p}{k}\\
    &=\sum_{k_1\in\D_{p}\cap[0,1)} \sum_{k_0\in\D_{p_0}\cap[0,1)}\itg_{k_0}^{\infty} \coef{p,k_1}{p_0,k_0-\floor{k}}{\model} \ind_{k-\floor{k}=k_1}d\mp{p}{k}\\
    & = \sum_{k_1\in\D_{p}\cap[0,1)} \sum_{k_0\in\D_{p_0}\cap[0,1)} \sum_{ k_0-k_1 \leq j\in\Z} \coef{p,k_1}{p_0,k_0-j}{\model} \\
    & = \sum_{k_1\in\D_{p}\cap[0,1)} \itg_{-\infty}^{k_1} \coef{p,k_1}{p_0,k}{\model} d\mp{p_0}{k}.
\end{align*}
Here \textbf{point 1} concludes the proof of \textbf{point 2} since there $\card(\D_p \cap [0,1)) =n_p$. 
\end{proof}

\begin{proof}[Proof of Proposition \ref{propA.2}]
Let $V^p_k = \sum_{p'\in\Pg} \int_{-\infty}^{k} \coef{p,k}{p',k'}{\model} Z^{p'}_{k'}d\mp{p'}{k'}$. By triangular inequality, Fubini Theorem and Lemma \ref{lemA.1}, it is clear that $\sup_{p,k \in\Q} \E{\vert V^p_k\vert} <\infty$. Thus $V$ is well defined and uniformly bounded in $L^1$ as required. Let us check that $V$ is indeed a solution. By point 3 of Lemma \ref{lem4.3} We have
\allowdisplaybreaks
\begin{align*}
    V^p_{k} & = \sum_{p'\in\Pg} \itg_{-\infty}^{k}  \coef{p,k}{p',k'}{\model} Z^{p'}_{k'} d\mp{p'}{k'} = Z^p_{k} + \sum_{p'\in\Pg} \itg_{-\infty}^{k-}  \coef{p,k}{p',k'}{\model} Z^{p'}_{k'} d\mp{p'}{k'}\\
    & = Z^p_{k} + \sum_{p'\in\Pg} \itg_{-\infty}^{k-}  \left[\sum_{p''\in\Pg} \itg_{k'}^{k-} h^{p}_{p''}(k-k'') \coef{p'',k''}{p',k'}{\model} d\mp{p''}{k''}\right] Z^{p'}_{k'} d\mp{p'}{k'}\\
    & = Z^p_{k} + \sum_{p''\in\Pg} \itg_{-\infty}^{k-} h^{p}_{p''}(k-k'') \left[\sum_{p'\in\Pg} \itg_{-\infty}^{k''} \coef{p'',k''}{p',k'}{\model} Z^{p'}_{k'} d\mp{p'}{k'}\right] d\mp{p''}{k''}\\
    & = Z^p_{k} + \sum_{p'\in\Pg} \itg_{-\infty}^{k-} h^{p}_{p'}(k-k') V^{p'}_{k'} d\mp{p'}{k'}.
\end{align*}
Swapping integrals from line 2 to line 3 is justified since Fubini Theorem applies almost surely. Indeed, by Lemma \ref{lemA.1} and if $\mathcal{Z} = \sup_{p,k\in\Q} \E{\vert Z^p_k\vert}$ we have
\allowdisplaybreaks
\begin{align*}
	\mathds{E}\bigg[\sum_{p'',p'\in\Pg} &\itg_{-\infty}^{k-}\itg_{-\infty}^{k''} h^{p}_{p''}(k-k'')  \coef{p'',k''}{p',k'}{\model} \vert Z^{p'}_{k'}\vert d\mp{p'}{k'} d\mp{p''}{k''}\bigg]\\
	& \leq \mathcal{Z}\sum_{p'',p'\in\Pg} \itg_{-\infty}^{k-}\itg_{-\infty}^{k''} h^{p}_{p''}(k-k'')  \coef{p'',k''}{p',k'}{\model} d\mp{p'}{k'} d\mp{p''}{k''}\\
	& \leq \mathcal{Z} \bigg[\Hmat^W_W \big(\Id-\Hmat^W_W\big) \ones\bigg]_{p}<\infty.
\end{align*}
For uniqueness we can suppose by linearity that $Z=0$, and we have to show that if $V^p_k = \sum_{p'\in\Pg} \int_{-\infty}^{k-} h^p_{p'}(k-k') V^{p'}_{k'}d\mp{p'}{k'}$ and $\sup_{p,k\in\Q} \E{|V^p_k|}<\infty$ then $V=0$. Let $\mathcal{V}_p=\sup_{k\in\D_p} \E{|V^p_k|}$ for $p\in\Pg$. Then for any $p\in\Pg$ we have,
\[ \E{|V^p_{k}|} \leq \sum_{p'\in\Pg}\itg_{-\infty}^{k-} h^{p}_{p'}(k-k') \E{\vert V^{p'}_{k'}\vert} d\mp{p'}{k'} \leq \sum_{p'\in\Pg} (\Hmat^W_W)_{pp'} \mathcal{V}_{p'}.\]
Thus, by taking a supremum on $k\in\D_p$, we have $\mathcal{V}_p \leq \sum_{p'\in\Pg} (\Hmat^W_W)_{pp'} \mathcal{V}_{p'} = (\Hmat^W_W \mathcal{V})_p$ with $\mathcal{V} =(\mathcal{V}_p)_{p\in\Pg}$. Iterating this formula gives $\mathcal{V} \preceq (\Hmat^W_W)^n \mathcal{V} \xrightarrow[n\to\infty]{} \bm{0}$ which concludes.
\end{proof}

\section{Dominations for HAR processes}\label{appendixB}

\subsection{Linear domination of HAR processes}

The following proposition gives a domination of HAR process with parameters $(\model,\xi)$ by the one with the associated linear parameters $(\model^+,\xi^+)$ defined in Definition \ref{def5.1}. This proposition allows us to transfer results from linear HAR processes to general HAR processes. Recall that we denote by $\nu^m$ the quantity defined by \eqref{eq2.3} so that $\lambda^m = \Phi^m(\nu^m)$.

\begin{proposition}[$\lm$ domination]\label{propB.1}
    Let parameters and random drifts $(\model,\xi)$ and the associated linear couple $(\model^+,\xi^+)\in\lm$ defined in Definition \ref{def5.1}. 
    \begin{enumerate}
        \item Suppose that assumption \Cnoexpl$[\model]$ holds, then \Cnoexpl$[\model^+]$ also holds. For any $t_0\in\R$, for any integrable initial condition $\mathfrak{C}$, the initial condition $\mathfrak{C}^+$ is also integrable and processes $X=\Equ{\model,\xi}{t_0,\mathfrak{C}}$ and $\tilde{X}=\Equ{\model^+,\xi^+}{t_0,\mathfrak{C}^+}$ from Theorem \ref{th3.4} compare as follows:
        \begin{itemize}
            \item[--] $\forall m, \ N^m \subset \tilde{N}^m$, and $\forall t\geq t_0, \ \lambda^m_t \leq \tilde{\lambda}^m_t$, more precisely $L_m\vert\nu^m_t\vert \leq \tilde{\lambda}^m_t - \Phi^m(0)$,
            \item[--] $\forall (p,k)\in\Q$ with $k\geq t_0, \ \vert W^p_{k}\vert \leq \tilde{W}^p_{k}.$
        \end{itemize}
        \item Suppose that assumptions \Cspec$[\model]$ and \Cfime$[\xi]$ hold, then \Cspec$[\model^+]$ and \Cfime$[\xi^+]$ also hold and processes $X=\Equ{\model,\xi}{-\infty}$ and $\tilde{X}=\Equ{\model^+,\xi^+}{-\infty}$ from Theorem \ref{th3.9} compare as follows:
        \begin{itemize}
            \item[--] $\forall m, \ N^m \subset \tilde{N}^m$, and $\forall t\in\R, \ \lambda^m_t \leq \tilde{\lambda}^m_t$, more precisely $L_m\vert\nu^m_t\vert \leq \tilde{\lambda}^m_t - \Phi^m(0)$,
            \item[--] $\forall (p,k)\in\Q, \ \vert W^p_{k}\vert \leq \tilde{W}^p_{k}.$
        \end{itemize}
    \end{enumerate}
\end{proposition}

\begin{proof}[Proof of Proposition \ref{propB.1}]
Let $\model = \big((h^{\alpha}_{m})_{\alpha,m}, (\J^{\alpha}_p,h^{\alpha}_{p},b^{\alpha}_p)_{\alpha,p}, \ (\Phi^{m})_{m}, \ (\Phi^p)_p \big)$. We denote with a $\tilde{\ }$ corresponding quantity in $\model^+$.\\ \textbf{Point 1.} That fact that assumption \Cnoexpl$[\model^+]$ holds is clear since it holds for $\model$. It is also clear that the initial condition $\mathfrak{C}^+$ is integrable since $\mathfrak{C}$ is integrable. 

Let $t_0=T_0<T_1<T_2<\cdots<T_n<\cdots$ be the ordered random sequence of both the points $(N^m\cap[t_0,\infty))_{\in\M}$ and the domains $\cup_{p\in\Pg} \D_p\cap[t_0,\infty)$. We will prove that the result holds for all $t\in[T_0, T_n]$ and all $(p,k)\in\Q$ with $T_0 \leq k\leq T_n$ by induction on $n$.

\underline{$n\to n+1$:} Let us first remark that by induction and by construction of $\mathfrak{C}^+$ we have, for all $\alpha\in\M\cup\Pg$ and all $(p,k)\in\Q$ with $k\leq T_n$ (which is the same as $k<T_{n+1}$),
\begin{equation}\label{eq domi lin 1}
	L_{\alpha} \vert\J^{\alpha}_p(t-k,W^p_k)\vert \leq \tJ^{\alpha}_p(t-k,\tilde{W}^p_{k}).
\end{equation} 
Since for $X$ there are no values of $W$ nor points of $N$ in $(T_n,T_{n+1})$, we have, for $t\in(T_n,T_{n+1}]$,
\[\nu^m_t = \sum_{m'\in\M} \itg_{-\infty}^{T_n} h^{m}_{m'}(t-s)dN^{m'}_s + \sum_{p\in\Pg}\itg_{-\infty}^{T_n} \J^m_{p}(t-k,W^p_{k}) d\mp{p}{k}.\]
By definition one has, $\tilde{\lambda}^m_{t} = \Phi^m(0)  + \sum_{m'\in\M} \itg_{-\infty}^{t-}  \tilde{h}^{m}_{m'}(t-s) d\tilde{N}^{m'}_s + \sum_{p\in\Pg}\itg_{-\infty}^{t-} \tJ^m_{p}(t-k,\tilde{W}^p_{k}) d\mp{p}{k}$. But since by induction we have that every point of $N$ before $T_n$ is also a point for $\tilde{N}$, and by \eqref{eq domi lin 1}, we have,
\begin{equation*}
    \tilde{\lambda}^m_t \geq \Phi^m(0) + \sum_{m'\in\M} \itg_{-\infty}^{T_n} L_m \vert h^{m}_{m'}(t-s)\vert dN^{m'}_s + \sum_{p\in\Pg}\itg_{-\infty}^{T_n} L_m  \vert \J^m_{p}(t-k, W^p_{k})\vert d\mp{p}{k}.
\end{equation*}
From this it is clear that $L_m\vert\nu^m_t\vert \leq \tilde{\lambda}^m_t - \Phi^m(0)$, which implies, since $\Phi^m$ is $L_m$-Lipschitz, that $\lambda^m_t \leq \tilde{\lambda}^m_t$. Therefore in $(T_n,T_{n+1}]$ we still have $N^m\subset \tilde{N}^m$. Finally, if there exists $p\in\Pg$ such that $T_{n+1})\in\D_p$ we have
\begin{align*}
    \vert  W^p_{T_{n+1}} \vert & =\bigg\vert  \xi^p_{T_{n+1}} + \Phi^p\bigg(\sum_{m\in\M} \itg_{-\infty}^{T_{n+1}-} h^{p}_{m}(T_{n+1}-s)dN^{m}_s + \sum_{p'\in\Pg}\itg_{-\infty}^{T_{n+1}-} \J^{p}_{p'}(T_{n+1}-k',W^{p'}_{k'}) d\mp{p'}{k'}\bigg) \bigg\vert\\
    & \leq  \vert\xi^p_{T_{n+1}}\vert + \sum_{m\in\M} \itg_{-\infty}^{T_{n+1}-} L_p \vert h^{p}_{m}(T_{n+1}-s)\vert d\tilde{N}^{m}_s + \sum_{p'\in\Pg}\itg_{-\infty}^{T_{n+1}-} \tJ^{p}_{p'}(T_{n+1}-k')(\tilde{W}^{p'}_{k'})d\mp{p'}{k'} \\
    & = \tilde{W}^p_{T_{n+1}}.
\end{align*}
\underline{$n=0$:} It is clear from the definition of $\model^+$, $\mathfrak{C}^+$, $\xi^+$ and the arguments in the step $n\to n+1$. This concludes the proof of \textbf{point 1}.

\textbf{Point 2.} First let us check that $\model^+$ satisfies the assumptions of Theorem \ref{th3.9}. Assumption \Cfime$[\xi^+]$ is clear since $\xi^+ = \vert \xi\vert$. By construction $\Hmod{\model^+}=\Hmod{\model}$, thus assumption \Cspec$[\model^+]$ is also clear. Finally one easily checks, by induction on $n$, using same ideas as in proof of \textbf{point 1}, that the following is true in the Picard iteration from the proof of Theorem \ref{th3.9},
\begin{equation*}
    \forall n\in\N, \ 
    \left\{ 
    \begin{aligned}
        & \forall m\in\M, \ \forall t\in\R, \  L_m\vert\nu^{m,n}_t\vert \leq \tilde{\lambda}^{m,n}_t - \Phi^m(0) \ \text{and} \ \lambda^{m,n}_t \leq \tilde{\lambda}^{m,n}_t\\
        & \forall (p,k)\in\Q, \ \vert W^{p,n}_{k} \vert \leq \tilde{W}^{p,n}_{k}
    \end{aligned}
    \right.
    .
\end{equation*}
Since the result is true at each step of the Picard iteration, it is true for the limit, which concludes.
\end{proof}

\subsection{Domination of the difference of two HAR processes}

This subsection is dedicated to Lemma \ref{lemB.2}, a technical Lemma used in the proof of Theorem \ref{th6.3}. Recall that in the proof of Theorem \ref{th6.3} we have two HAR processes with different initial conditions, indexed by $i=1,2$ at time $t_0\in\R$. The process $\Big(( \Delta \lambda^m)_{m\in\M},(\Delta N^{m})_{m\in\M},(\Delta W^{p})_{p\in\Pg}\Big)$ corresponds to the difference, in absolute value, between these two processes. We recall that we denote by $\tilde{\mathfrak{h}}$ the cluster functions for $\model^+$.

Lemma \ref{lemB.2} bounds the difference process by a linear HAR process with vanishing base rates.

\begin{lemma}\label{lemB.2}
    Suppose that assumptions \Cspecbis$[\model]$ and \Cexptail$[\model]$ hold. Conditionally to $\F_{t_0-}$ there exist Poisson random measures $\mathscr{P}^m$ on $[t_0,\infty)\times \R_+$ for $m\in\M$ and a process $\Big((\mathscr{l}^m)_{m\in\M},(\mathscr{N}^m)_{m\in\M},(\mathscr{W}^p)_{p\in\Pg}\Big)$ adapted to $\mathcal{G}_t = \sigma\big( \mathscr{P}^m\big\vert_{[t_0,t]\times \R_+}, \ m\in\M \big)$ such that
    \begin{itemize}
        \item for $m\in\M, \ t\geq t_0$, $ \Delta\lambda^m_t \leq \mathscr{l}^m_t$ and $\mathscr{l}^m$ is $\mathcal{G}_t$-predictable ,
        \item for $m\in\M, \ t\geq t_0$, $d\Delta N^{m}_t \leq d\mathscr{N}^m_t$,
        \item for $(p,k)\in\Q, \ k\geq t_0$, $\Delta W^{p}_k \leq \mathscr{W}^p_k$.
    \end{itemize}
    \begin{flalign*}
    \text{And we have,} \qquad
    \left\{
    \begin{aligned}
    & \mathscr{l}^m_t  = u_0 e^{-c(t-t_0)} + \sum_{m'\in\M}\itg_{t_0}^{t-} \tilde{\mathfrak{h}}^m_{m'}(s,t) d\mathscr{N}^{m'}_s, \quad t\geq t_0,\\
    & \mathscr{N}^m(C) = \itg_{C\times \R_+} \ind_{0< x\leq \mathscr{l}^m_s } d\mathscr{P}^m(s,x), \quad C\subset [t_0,\infty),\\
    & \mathscr{W}^p_{k} = u_0 e^{-c(k-t_0)} + \sum_{m\in\M} \itg_{t_0}^{k-} Ke^{-c(k-s)} d\mathscr{N}^m_s, \quad (p,k)\in\Q \text{ with } k\geq t_0,
    \end{aligned}
    \right.&&
    \end{flalign*}
    where $u_0 = K \Massiexp_{\delta,t_0}\big(\mathfrak{C}_1 \circleddash \mathfrak{C}_2\big)$, see \eqref{notation mass exp}, with $K, c, \delta>0$ constants depending only on $\model$.
\end{lemma}

\begin{proof}[Proof of Lemma \ref{lemB.2}]
    As in the proof of Proposition \ref{propB.1} we denote with an additional $\tilde{\ }$ the elements of $\model^+$. Recall that for $i=1,2$ we have for $t\geq t_0$, $C\subset [t_0,\infty)$ and $k\geq t_0$
    \begin{equation*}
        \left\{
        \begin{aligned}
            & \lambda^{m,\mathfrak{C}_i}_t = \Phi^m\bigg( \sum_{m'\in\M} \itg_{-\infty}^{t-} h^{m}_{m'}(t-s)dN^{m',\mathfrak{C}_i}_s + \sum_{p\in\Pg} \itg_{-\infty}^{t-} \J^m_p(t-k,W^{p,\mathfrak{C}_i}_k) d\mp{p}{k}\bigg) \\
            & N^{m,\mathfrak{C}_i}(C) = \itg_{C\times \R_+} \ind_{x\leq \lambda^{m,\mathfrak{C}_i}_s} d\pi^m(s,x)\\
            & W^{p,\mathfrak{C}_i}_k = \xi^p_k + \Phi^p\bigg(\sum_{m\in\M} \itg_{-\infty}^{k-} h^p_m(k-s)dN^{m,\mathfrak{C}_i}_s + \sum_{p'\in\Pg} \itg_{-\infty}^{k-} \J^p_{p'}(k-k',W^{p',\mathfrak{C}_i}_{k'})d\mp{p'}{k'}\bigg).
        \end{aligned}
        \right.
    \end{equation*}
    Until the end we work conditionally to $\F_{t_0-}$ and thus initial conditions are deterministic.
    Let $\mathscr{P}^m \subset [t_0,\infty)\times\R_+$ defined by $\mathscr{P}^m(s,x) = \pi^m\Big(s,x+\min\big[\lambda^{m,\mathfrak{C}_1}_s,\lambda^{m,\mathfrak{C}_2}_s\big]\Big)$. Random measure $\mathscr{P}^m$ is a $\F_t$-predictable vertical shift of $\pi^m$ and thus it is also a Poisson random measure. Then it follows that for $C\subset [t_0,\infty)$,
    \begin{equation}\label{eqD.1}
        \Delta N^m (C) =  \itg_{C\times \R_+} \ind_{0<x\leq  \Delta\lambda^m(s)} d\mathscr{P}^m(s,x).
    \end{equation}
    Since $\Phi^m$ is $L_m$ Lipschitz, by triangular inequality we have
    \begin{align*}
        \Delta\lambda^m_t \leq  \sum_{m'\in\M} & \itg_{t_0}^{t-}  \tilde{h}^{m}_{m'}(t-s)  d\Delta N^{m'}_s + \sum_{p\in\Pg} \itg_{t_0}^{t-} \tilde{h}^m_p(t-k)\Delta W^{p}_k d\mp{p}{k}+ \sum_{m'\in\M} \itg_{-\infty}^{t_0-} \tilde{h}^{m}_{m'}(t-s)  d(\mathfrak{C}^{m'}_1 \vartriangle d\mathfrak{C}^{m'}_2)_s  \\
        & + \sum_{p\in\Pg} \itg_{-\infty}^{t_0-} L_m\vert\J^m_p(t-k,\mathfrak{C}^p_{1,k}) - \J^m_p(t-k,\mathfrak{C}^p_{2,k})\vert d\mp{p}{k}.
    \end{align*}
    From \Cexptail$[\model]$, which implies \Cexptail$[\model^+]$, there exists constant $c_1,K_1>0$ such that $\tilde{h}^{m}_{m'}(s)\leq K_1 e^{-c_1 s}$, thus we have,
    \begin{align*}
        \sum_{m'\in\M} \itg_{-\infty}^{t_0-} \tilde{h}^{m}_{m'}(t-s)  d(\mathfrak{C}^{m'}_1 \vartriangle \mathfrak{C}^{m'}_2)_s 
        & \leq K e^{-c(t-t_0) } \Massiexp_{\delta,t_0}\big(\mathfrak{C}_1 \circleddash \mathfrak{C}_2\big),    
	\end{align*}
    with $c  = \delta = c_1 >0$ and $K=K_1$. The function $\psi_2$ defined in \eqref{def psi} is tuned such that
    \[L_m\vert\J^m_p(t-k,\mathfrak{C}^p_{1,k}) - \J^m_p(t-k,\mathfrak{C}^p_{2,k})\vert \leq \big( \tilde{h}^m_p(t-k) + \tilde{b}^m_p(t-k)\big) \psi_2(\mathfrak{C}^p_{1,k},\mathfrak{C}^p_{2,k}).\]
    Thus, by \Cexptail$[\model^+]$ we also have for some $c,K>0$,
    \[\sum_{p\in\Pg} \itg_{-\infty}^{t_0-} L_m\vert\J^m_p(t-k,\mathfrak{C}^p_{1,k}) - \J^m_p(t-k,\mathfrak{C}^p_{2,k})\vert d\mp{p}{k} \leq K e^{-c(t-t_0) } \Massiexp_{\delta,t_0}\big(\mathfrak{C}_1 \circleddash \mathfrak{C}_2\big).\]
    Thus, if $u_0 = K \Massiexp_{\delta,t_0}\big(\mathfrak{C}_1 \circleddash \mathfrak{C}_2\big)$, up to increasing $K$ we have, (same arguments lead to the inequality for $\Delta W^p_k$)
    \begin{equation}\label{eqD.2}
    	\left\{
    	\begin{aligned}
    		& \Delta\lambda^m_t \leq u_0 e^{-c(t-t_0)} + \sum_{m'\in\M} \itg_{t_0}^{t-} \tilde{h}^{m}_{m'}(t-s)  d\Delta S^{m'}_s + \sum_{p\in\Pg} \itg_{t_0}^{t-} \tilde{h}^m_p(t-k) \Delta W^{p}_k d\mp{p}{k}.\\
    		& \Delta W^p_k \leq u_0 e^{-c(k-t_0)} + \sum_{m\in\M} \itg_{t_0}^{k-} \tilde{h}^{p}_{m}(k-s) d\Delta S^{m}_s + \sum_{p'\in\Pg} \itg_{t_0}^{k-} \tilde{h}^p_{p'}(k-k')\Delta W^{p'}_{k'} d\mp{p'}{k'}.
    	\end{aligned}
        	\right.
    \end{equation}
    Thus a process $\Big((\check{\mathscr{l}}^m)_{m\in\M},(\check{\mathscr{N}}^m)_{m\in\M},(\check{\mathscr{W}}^p)_{p\in\Pg}\Big)$ satisfying equality in equations \eqref{eqD.1} and \eqref{eqD.2} and, ie
    \begin{equation}\label{eqD.4}
        \left\{
        \begin{aligned}
            & \check{\mathscr{l}}^m_t = u_0 e^{-c(t-t_0)} + \sum_{m'\in\M} \itg_{t_0}^{t-} \tilde{h}^{m}_{m'}(t-s)  d\check{\mathscr{N}}^{m'}_s + \sum_{p\in\Pg} \itg_{t_0}^{t-} \tilde{h}^m_p(t-k)  \check{\mathscr{W}}^p_k d\mp{p}{k} \\
            & \check{\mathscr{N}}^m(C) = \itg_{C\times \R_+} \ind_{x\leq \check{\mathscr{l}}^m_s} d\mathscr{P}^m(s,x)\\
            & \check{\mathscr{W}}^p_k = u_0 e^{-c(k-t_0)} + \sum_{m\in\M} \itg_{t_0}^{k-} \tilde{h}^{p}_{m}(k-s) d\check{\mathscr{N}}^m_s + \sum_{p'\in\Pg} \itg_{t_0}^{k-} \tilde{h}^p_{p'}(k-k')  \check{\mathscr{W}}^{p'}_{k'} d\mp{p'}{k'},
        \end{aligned}
        \right.
    \end{equation}
    must dominate $\Big((\Delta \lambda^m)_{m\in\M}, (\Delta S^{m})_{m\in\M}, (\Delta W^{p})_{p\in\Pg}\Big)$. Equations \eqref{eqD.4} are in fact linear HAR equations with parameter from $\model^+$ up to $\tilde{b}^{\alpha}_p=0$, vanishing base rates instead of $\Phi^m(0)$ for $\check{\mathscr{N}}$ and vanishing deterministic drifts instead of $\xi^+$ for $\check{\mathscr{W}}$. Thus one can use the results on linear models. More precisely, we can express exactly the process using cluster functions $\tilde{\mathfrak{h}}$ and coefficients $\coef{p,k}{p',k'}{\model^+}$ as follows,
    \begin{equation*}
        \left\{
        \begin{aligned}
            & \check{\mathscr{l}}^m_t = u_0 e^{-c(t-t_0)} + \sum_{p,p'\in\Pg}\itg_{t_0}^{t-} \itg_{t_0}^k \tilde{h}^m_p(t-k)\coef{p,k}{p',k'}{\model^+}u_0 e^{-c(k'-t_0)}  d\mp{p'}{k'} d\mp{p}{k} + \sum_{m'\in\M} \itg_{t_0}^{t-} \tilde{\mathfrak{h}}^{m}_{m'}(t-s) d\check{\mathscr{N}}^{m'}_s \\
            & \check{\mathscr{N}}^m(C) = \itg_{C\times \R_+} \ind_{x\leq \check{\mathscr{l}}^m_s} d\mathscr{P}^m(s,x)\\
            & \check{\mathscr{W}}^p_k = \sum_{p'\in\Pg} \itg_{t_0}^k \coef{p,k}{p',k'}{\model^+}u_0 e^{-c(k'-t_0)}d\mp{p'}{k'} + \sum_{p'\in\Pg} \itg_{t_0}^k \sum_{m\in\M} \itg_{t_0}^{k'-} \coef{p,k}{p',k'}{\model^+}\tilde{h}^p_m(k'-s) d\check{\mathscr{N}}^m_s d\mp{p'}{k'}
        \end{aligned}
        \right. .
    \end{equation*}

	Then one can use the exponential decay of the coefficients stated in Proposition \ref{prop6.1}, and \Cexptail$[\model^+]$ to prove that we have (up to reducing $c$ and increasing $K$),
    \begin{equation}\label{eqD.5}
        \left\{
        \begin{aligned}
            & \check{\mathscr{l}}^m_t \leq u_0 K e^{-c(t-t_0)} + \sum_{m'\in\M} \itg_{t_0}^{t-} \tilde{\mathfrak{h}}^{m}_{m'}(t-s) d\check{\mathscr{N}}^{m'}_s \\
            & \check{\mathscr{N}}^m(C) = \itg_{C\times \R_+} \ind_{x\leq \check{\mathscr{l}}^m_s} d\mathscr{P}^m(s,x)\\
            & \check{\mathscr{W}}^p_k \leq u_0 K e^{-c(k-t_0)} + \sum_{m\in\M} \itg_{t_0}^{k-} K e^{-c(k-s)} d\check{\mathscr{N}}^m_s
        \end{aligned}
        \right. .
    \end{equation}
    The searched process is the one which satisfies equality in \eqref{eqD.5}, which is clearly adapted (predictable for $\mathscr{l}$) to $\mathcal{G}_t$. Then doing $u_0 \leftarrow K u_0$ concludes.
\end{proof}

\section{Exponential moments}\label{appendixC} 

In this appendix we state a general result on exponential moments for point processes $(N^m)_{m\in\M}$ with a 1-periodic cluster structure and immigrant rate given by linear functionals of random drifts. Recall that the random drifts $(\xi^p_k)_{(p,k)\in\Q}$ are independent random real variables, assumed to be non-negative: $\xi^p_k \geq 0$. In the framework of HAR processes we also require that for all $(p,k)\in\Q$ we have $\xi^p_k \sim \xi^p_{k+n}$ for any $n\in\Z$ but it is not necessary here. 

Consider the random variable $\mathcal{Z} = \displaystyle\sum_{m\in\M} \itg_{\R} B^m(t) dN^m_t + \sum_{p\in\Pg} \itg_{\R} C^p_k \xi^p_k d\mp{p}{k} = \Mass_B(N) + \Mass_C(\xi)$, where we have
\begin{itemize}
	\item $(N^m)_{m\in\M}$ a multivariate point process such that conditionally to the $(\xi^p_k)_{(p,k)\in\Q}$ it has immigrant rates $(\I^m)_m$ and intensity $(\lambda^m)_m$ given by
\begin{equation}\label{cluster rep f}
	\lambda^m_t = \I^m_t + \sum_{m'\in\M} \itg_{-\infty}^{t-} \mathfrak{f}^{m}_{m'}(s,t)dN^{m'}_s, \quad \I^m_t = \mathscr{i} + \sum_{(p,k)\in\Q} A^m_{p,k}(t) \xi^p_k,
\end{equation}
with non-negative bivariate functions $\bm{\mathfrak{f}}=(\mathfrak{f}^m_{m'})_{m,m'\in\M}$ such that $\mathfrak{f}^m_{m'}(s+1,t+1) =\mathfrak{f}^m_{m'}(s,t)$, a non-negative constant $\mathscr{i}$ and where $A^m_{p,k}$ are a non-negative real functions for all $m\in\M$, $(p,k)\in\Q$. We also define the following matrices; $A_{\infty} = \Big( \displaystyle\sup_{t\in\R} \itg_{\R} A^m_{p,k}(t)d\mp{p}{k}\Big)_{(p,m)\in \Pg\times\M}$ and $\mathfrak{F} = \Big( \displaystyle\sup_{s\in\R} \Vert \mathfrak{f}^{m}_{m'}(s,\cdot)\Vert_1 \Big)_{m,m'\in\M}$.
	\item $B=(B^m)_{m\in\M}$ are non-negative real functions and for $n\in\N^*$ we define $B_{\infty,n} = \big( \Vert B^m\VLd{1/n}\big)_{m\in\M}$. 
	\item $C=(C^p)_{p\in\Pg}$ are non-negative functions on the $\D_p$'s and we define $C_{\infty} = \Big(\itg_{\R} C^p_k d\mp{p}{k}\Big)_{p\in\Pg}$.
\end{itemize}
\begin{lemma}\label{lemC.1}
 	 Let $n\geq 1$ an integer. Suppose that assumption \Csubg$[\xi]$ holds. Then we have
    \[\E{e^{\mathcal{Z}}} \leq \exp\Big[ \mathscr{i} \vert n e^{\bm{\L}(B_{\infty,n})}-n \ones \vert_1 + \mathfrak{e} \vert\bm{\mathcal{X}}\vert_1 + \mathfrak{s}\vert\bm{\mathcal{X}}\vert_2^2 \Big],\]
where $\bm{\mathcal{X}} = C_{\infty} +  A_{\infty}  ( n e^{\bm{\L}(B_{\infty,n})}-n \ones)$, with $\L$ the log-Laplace function of a $\Pois(\mathfrak{F}^{\intercal})$ Galton Watson tree defined in Theorem 3.1 of \cite{leblanc_sharp_2025}. If $\spr(\mathfrak{F}^{\intercal})<1$ then $\L$ is finite and Lipschitz for $\vert \cdot \vert_{\infty}$ on a vicinity of $0$.
\end{lemma}

\begin{proof}[Proof of Lemma \ref{lemC.1}]
Conditioning on the random drifts gives $\E{e^{\mathcal{Z}}}  = \E{\E{e^{\mathcal{Z}}\cond \xi}}  = \E{e^{\Mass_C(\xi)} \E{e^{\Mass_B(N)}\cond \xi}}$.
Let us investigate the term $\E{e^{\Mass_B(N)}\cond \xi}$. We will use the cluster representation given in \eqref{cluster rep f}. Let $G^m_t$ a $\Pois( \bm{\mathfrak{f}}^{\intercal})$ born at time $t\in\R$ with root of type $m\in\M$. From classical calculations we have
\begin{equation*}
	\E{e^{\Mass_B(N)}\cond \xi} = \exp\Big[ \sum_{m\in\M} \itg_{\R} (\psi^m(t,B)-1)\I^m_t dt \Big] \quad \text{where} \quad \psi^m(t,B) = \E{e^{\Mass_B(G^m_t)}}.
\end{equation*}
We can now use the expression of the immigrant rates $\I^m$,
\begin{equation}\label{exp cond}
	\E{e^{\Mass_B(N)}\cond \xi} = \exp\Big[ \sum_{m\in\M} \mathscr{i} \itg_{\R} (\psi^m(t,B)-1) dt + \sum_{p\in\Pg} \itg_{\R} \xi^p_k \sum_{m\in\M}\itg_{\R} A^m_{p,k}(t) (\psi^m(t,B)-1) dt d\mp{p}{k}\Big].
\end{equation}
Denote $\Psi^m = \int_{\R}(\psi^m(t,B)-1) dt$ and $\Psi = (\Psi^m)_{m\in\M}$. By \eqref{exp cond} we have
\begin{equation*}
	\E{e^{\mathcal{Z}}} = \exp(\mathscr{i} \vert \Psi\vert_1) \mathds{E}\Big[\exp\big( \Mass_D(\xi)\big)\Big] \quad \text{with} \quad D^p_k = C^p_k + \sum_{m\in\M}\itg_{\R} A^m_{p,k}(t) (\psi^m(t,B)-1) dt.
\end{equation*}
Thus from assumption \Csubg$[\xi]$ we have $\E{e^{\mathcal{Z}}} \leq \exp\Big[ \mathscr{i} \vert \Psi\vert_1 + \displaystyle\mathfrak{e}\sum_{p\in\Pg}  \itg_{\R} D^p_k d\mp{p}{k} + \mathfrak{s}\sum_{p\in\Pg}  \itg_{\R} (D^p_k)^2 d\mp{p}{k}\Big]$.
Let us focus on $\itg_{\R} D^p_k d\mp{p}{k}$. By Fubini Tonelli, we have for $p\in\Pg$, $\itg_{\R} D^p_k d\mp{p}{k} \leq (C_{\infty})_p + \sum_{m\in \M} (A_{\infty})_{p,m} \Psi^m$.
To continue the proof we have to bound $\Psi^m$ for $m\in\M$. Recall that $n$ is a positive integer. We use the same ideas as in Theorem 3.10 of \cite{leblanc_sharp_2025} except the we have to deal with $1$-periodicity instead of complete stationarity.

For $t\in\R$, denote by $\floor{t}_n$ and $\{t\}_n$ the unique real numbers such that $t =  \floor{t}_n + \{t\}_n$ with $\floor{t}_n \in n\Z$, $0\leq \{t\}_n <n$. By 1-periodicity of clusters, if $\tau_x(s) = x + s$, we have $\psi^m(t,B) = \psi^m(\{t\}_n,B\circ \tau_{\floor{t}_n})$. Remark that by definition, $\sum_{j\in\Z} B\circ \tau_{nj} \preceq B_{\infty,n} \ind_{\R}$. It follows that
\allowdisplaybreaks
\begin{align*}
	\Psi^m & = \itg_0^n \sum_{j\in\Z} (\psi^m(t+nj,B)-1) dt  = \itg_0^n \sum_{j\in\Z} (\psi^m(t,B\circ \tau_{nj})-1) dt = \itg_0^n \mathds{E}\bigg[\sum_{j\in\Z} (e^{\Mass_{B\circ \tau_{nj}}(G^m_t)}-1)\bigg] dt\\
	& \leq \itg_0^n \mathds{E}\bigg[-1+\prod_{j\in\Z} e^{\Mass_{B\circ \tau_{nj}}(G^m_t)}\bigg] dt  = \itg_0^n \E{-1+ \exp\Big(\Mass_{\sum_{j\in\Z}B\circ \tau_{nj}}(G^m_t)\Big)} dt \\
	& \leq \itg_0^n \E{-1+ e^{\Mass_{B_{\infty,n}\ind_{\R}}(G^m_t)}} dt  \leq \itg_0^n \big[ e^{\bm{\L}(B_{\infty,n})}-\ones\big]_m dt = n \big[ e^{\bm{\L}(B_{\infty,n})}-\bm{1}\big]_m.
\end{align*}
The last inequality comes from the fact that $\{(u,m') \mid (u,m',s)\in G^m_t\}$ (ie $G^m_t$ without temporal embedding) is stochastically dominated by a $\Pois(\mathfrak{F}^{\intercal})$ Galton Watson process, and thus, with the notations of Theorem 3.1 of \cite{leblanc_sharp_2025}, $\Mass_{B_{\infty,n} \ind_{\R}}(G^m_t) \leq B_{\infty,n} \cdot \card_{\M}(\mathcal{T}^m)$ for a well coupled $\Pois(\mathfrak{F}^{\intercal})$ Galton Watson process $\mathcal{T}^m$ with root of type $m$. We conclude that
\begin{equation*}
	\itg_{\R} D^p_k d\mp{p}{k} \leq (C_{\infty})_p +  \big[A_{\infty}  ( n e^{\bm{\L}(B_{\infty,n})}-n \ones)\big]_p.
\end{equation*}
Finally, since $\itg_{\R} (D^p_k)^2 d\mp{p}{k} \leq \Big(\itg_{\R} D^p_k d\mp{p}{k}\Big)^2$ we can conclude that
\begin{equation*}
	\E{e^{\mathcal{Z}}}  = \exp\Big[  \mathscr{i} \vert \Psi\vert_1 + \displaystyle\mathfrak{e}\sum_{p\in\Pg}  \itg_{\R} D^p_k d\mp{p}{k} + \mathfrak{s}\sum_{p\in\Pg}  \itg_{\R} (D^p_k)^2 d\mp{p}{k}\Big] \leq \exp\Big[ \mathscr{i} \vert n e^{\bm{\L}(B_{\infty,n})}-n \ones\vert_1 + \mathfrak{e} \vert\bm{\mathcal{X}}\vert_1 + \mathfrak{s}\vert\bm{\mathcal{X}}\vert^2_2 \Big],
\end{equation*}
where we denoted $\bm{\mathcal{X}} := C_{\infty} +  A_{\infty}  ( n e^{\bm{\L}(B_{\infty,n})}-n \ones)$.
\end{proof}

\section{Extended filtration and martingale}\label{appendixD}

Recall that $\pi^m, \ m\in\M$ are independent Poisson random measures, that random drifts $\xi=(\xi^p_k)_{(p,k)\in\Q}$ are independent random variables also independent of $\pi^m, \ m\in\M$. In this appendix we prove that $(d\pi^m - dtdx)_{m\in\M}$ is a $(\F_t)_t$ martingale even if $\F_t$, defined in \eqref{filtration}, is bigger than the natural history of $(\pi^m)_{m\in\M}$ since it also contains information about the random drifts $\xi$.

Let $(\Omega_1, \mathcal{A}_1,P_1)$ a probability space, $(E,\mathcal{E})$ a measurable space ($E=\R_+$ in HAR framework), $\mu$ a measure on $E$ (Lebesgue measure in HAR framework) and $\pi$ a Poisson random measure on $\R\times E$ with intensity $dt\times d\mu(x)$ and the natural filtration
\[\F^1_t = \sigma\Big( \pi\big\vert_{(-\infty,t]\times E}\Big).\] 
Let $(\Omega_2, \mathcal{A}_2,P_2)$ another probability space and $(Z_t)_{t\in\R}$ a process (the analogue of random drifts $\xi$ in HAR framework) and denote by $\F^2$ the natural history of $Z$. Consider $(\Omega,\F,\mathds{P})$ the probability space defined by 
\[\Omega = \Omega_1 \times \Omega_2, \quad \F = \mathcal{A}_1 \otimes \mathcal{A}_2, \quad \text{and} \quad \mathds{P} = P_1\otimes P_2.\]
Then $\pi$ can be viewed as a random variable on $\Omega$ by $\pi(\omega_1,\omega_2) = \pi(\omega_1)$ for any $(\omega_1,\omega_2)\in\Omega$, and similarly for $Z$. In $\Omega$ it is clear that $\pi$ and $Z$ are independent. Consider the filtration 
\[\F_t = \sigma\Big( (Z_s)_{s\leq t}, \ \pi\big\vert_{(-\infty,t]\times E} \Big).\] It is clear that we have $\F_t = \F^1_t \otimes \F^2_t$ and $(\F_t)_t$ is the analogue of \eqref{filtration} in the context of HAR processes. 

Denote by $\Pi$ the $(\F_t)_t$-predictable $\sigma$-algebra for processes $X:\R\times E\times \Omega \rightarrow \R$, i.e. the $\sigma$-algebra generated by sets of the form $(s,t]\times B\times A$ with $s<t$, $B\in\mathcal{E}$ and $A\in\F_s$. Similarly, let $\Pi^1$ the analogue for $(\F^1_t)_t$.

Let $X : (t,x,\omega_1,\omega_2) \in\R \times E \times \Omega \longmapsto X_t(x,\omega_1,\omega_2) \in\R$ be a $(\F_t)_t$-predictable process, which means that $X$ is $\Pi$ measurable. We want to show that for any $\omega_2\in\Omega_2$, the process $X^{\omega_2} : (t,x,\omega_1) \in\R \times E \times \Omega_1 \longmapsto X_t(x,\omega_1,\omega_2) \in\R$ is $\F^1_t$- predictable, i.e. $\Pi^1$-measurable.

Denote by $\mathcal{G}$ the class of processes for which the above property holds. Clearly for any $s<s'$, $B\in \mathcal{E}$, $A_1\in\F^1_s$ and $A_2\in\F^2_s$, the process $\ind_{(s,s']\times B\times A_1 \times A_2}$ is clearly in $\mathcal{G}$. Since $\mathcal{G}$ is stable by linear combinations and by limits, and because $\F_t = \F^1_t \otimes \F^2_t$, it is clear that all the processes of the form $\ind_{(s,s']\times B\times A}$, with $s<s'$, $B\in\mathcal{E}$, $A\in\F_s$ are in $\mathcal{G}$. Thus for any $D\in \Pi$, we have $\ind_D \in \mathcal{G}$. Then it is clear that $\mathcal{G}$ is exactly the class of $\F_t$-predictable processes. We proved that for all $\omega_2\in\Omega_2$, the process $X^{\omega_2}$ is $\F^1_t$-predictable.

Thus, if $X$ is a non-negative $(\F_t)_t$-predictable process, by Fubini Theorem and since $d\pi(t,x)-dtd\mu(x)$ is a $(\F^1_t)_t$ martingale, we have
\begin{align*}
	\E{\itg_{\R^2} X(t,x) d\pi(t,x)} & = \itg_{\Omega} \itg_{\R^2} X(t,x,\omega) d\pi(t,x) d\mathds{P}(\omega) \\
	& =  \itg_{\Omega_2}\bigg[\itg_{\Omega_1} \itg_{\R^2} X^{\omega_2}(t,x,\omega_1) d\pi(t,x) dP_1(\omega_1) \bigg] dP_2(\omega_2) \\
	& = \itg_{\Omega_2}\bigg[\itg_{\Omega_1} \itg_{\R^2} X^{\omega_2}(t,x,\omega_1) dtd\mu(x) \ dP_1(\omega_1) \bigg] dP_2(\omega_2) \\
	& = \E{\itg_{\R^2} X(t,x) dtd\mu(x)}.
\end{align*}
Thus by taking $E=\R_+$, $\mu$ the Lebesgue measure, $Z=\xi$ and $X(t,x) = \ind_{0\leq x \leq \lambda_t}$ with $\lambda$, defined in \eqref{eq2.2a}, $(\F_t)_t$-predictable, it is clear that the $(\F_t)_t$-compensator of 
\[t\longmapsto dN_t = \pi\big(\{t\}\times [0,\lambda_t]\big) \quad \text{is} \quad \lambda_t dt.\]
\end{appendix}

%%%%%%%%%%%%%%%%%%%%%%%%%%%%%%%%%%%%%%%%%%%%%%
%% Support information, if any,             %%
%% should be provided in the                %%
%% Acknowledgements section.                %%
%%%%%%%%%%%%%%%%%%%%%%%%%%%%%%%%%%%%%%%%%%%%%%

\begin{acks}[Acknowledgments]
I would like to thank my PhD advisors, Vincent Rivoirard and Patricia Reynaud-Bouret, for their advice and helpful comments which greatly improved this article.
\end{acks}

%%%%%%%%%%%%%%%%%%%%%%%%%%%%%%%%%%%%%%%%%%%%%%
%% Funding information, if any,             %%
%% should be provided in the                %%
%% funding section.                         %%
%%%%%%%%%%%%%%%%%%%%%%%%%%%%%%%%%%%%%%%%%%%%%%
%\begin{funding}
% The first author was supported by ...
%
% The second author was supported in part by ...
%\end{funding}

%%  imsart-number.bst  will be used to                     %%
%%  MR numbers will be added by VTeX.                      %%
%%                                                         %%
%%  Use \cite{...} to cite references in text.             %%
%%                                                         %%
%%%%%%%%%%%%%%%%%%%%%%%%%%%%%%%%%%%%%%%%%%%%%%%%%%%%%%%%%%%%%

%% if your bibliography is in bibtex format, uncomment commands:
\bibliographystyle{imsart-number} % Style BST file
\bibliography{biblio.bib}    % Bibliography file (usually '*.bib')

%% or include bibliography directly:
% \begin{thebibliography}{}
% \bibitem{b1}
% \end{thebibliography}

\end{document}